\documentclass[reqno,final]{amsart}
\usepackage{natbib}  %Nature-like bibliography
\usepackage{fancyhdr} %Headers
\usepackage{color} %Color definition
\usepackage{hyperref} %Internal and external links
\usepackage{graphicx} %Graphics inclusion
\usepackage{chngcntr}
\usepackage{tocvsec2}
\usepackage{pstricks}
\usepackage{amssymb}
\newcommand{\rhohat}{\hat{\rho}}

\newcommand{\cenxi}{\stackrel{_\leadsto}{\xi}}

\newcommand{\Zch}{\check{Z}}

\definecolor{aleacolor}{rgb}{0.16,0.59,0.78}
   \def\beq{\begin{eqnarray}}
    \def\eeq{\end{eqnarray}}
    \def\beqq{\begin{eqnarray*}}
    \def\eeqq{\end{eqnarray*}}
	\newcommand{\re}{{\rm{Re}}}
	\newcommand{\im}{{\rm{Im}}}
    \def\p{{\mathbb P}}

    \def\r{{\mathbb R}}
    \def\d{{\textnormal d}}
        \def\D{{\textnormal d}}

    \def\i{{\textnormal i}}

\hypersetup{
breaklinks,
colorlinks=true,
linkcolor=aleacolor,
urlcolor=aleacolor,
citecolor=aleacolor}

\renewcommand{\headheight}{11pt}
\renewcommand{\cite}{\citet}

\theoremstyle{plain}
\newtheorem{theorem}{Theorem}[section]                                          
\newtheorem{proposition}[theorem]{Proposition}                          
\newtheorem{lemma}[theorem]{Lemma}
\newtheorem{corollary}[theorem]{Corollary}

\theoremstyle{definition}
\newtheorem{definition}[theorem]{Definition}
\theoremstyle{remark}
\newtheorem{remark}[theorem]{Remark}

\makeatletter \@addtoreset{equation}{section} \makeatother
\renewcommand\thefigure{\thesection.\arabic{figure}}
\renewcommand\thetable{\thesection.\arabic{table}}
\renewcommand{\thepart}{\Roman{part}} 
\renewcommand\thesection{\thepart.\arabic{section}}
\renewcommand\thesubsection{\thesection.\arabic{subsection}}
    \def\d{{\textnormal d}}
        \def\D{{\textnormal d}}

        \def\r{{\mathbb R}}

    \newcommand{\iu}{{\textnormal i}}
\newcommand{\cenxicirc}{\stackrel{_{\leadsto}}{\mathbf{\xi}^\circ}}

\counterwithout{equation}{subsection}
\counterwithout{equation}{section}

\begin{document}
\title[Stable L\'evy processes, self-similarity  and the unit ball]{Stable L\'evy processes, self-similarity \\ and the unit ball${}^*$}
\thanks{{\it${}^*$Dedicated to the memory of my co-author and friend, Iranian probabilist Prof. Aliakbar Rahimzadeh Sani, together with whom I studied Markov additive processes  many years ago.}}

\author{Andreas E. Kyprianou}
\address{Department of Mathematical Sciences, \newline
University of Bath, Claverton Down, \newline 
Bath, BA2 7AY UK. }
\email{a.kyprianou@bath.ac.uk}
\urladdr{\url{http://www.maths.bath.ac.uk/~ak257/}}

%\author{Juan Carlos Pardo}
%\address{CIMAT A. C.,\newline
%Calle Jalisco s/n,\newline
%Col. Valenciana,\newline
%A. P. 402, C.P. 36000,\newline
%Guanajuato, Gto.,\newline
%Mexico.}
%\email{jcpardo@cimat.mx}
%\urladdr{\url{http://www.cimat.mx/~jcpardo/}}

\thanks{Research supported by EPSRC grant EP/L002442/1 and Royal Society Advanced Newton Fellowship.}


\begin{abstract}
 Around the 1960s a celebrated collection of papers emerged offering a number of explicit identities for the class of isotropic L\'evy stable processes in one and higher dimensions; these include, for example, the lauded works of \cite{Ray, Widom, Rog} (in one dimension) and  \cite{BGR, G, Port69} (in higher dimension), see also more recently \cite{BMR, Luks}. Amongst other things, these results  nicely exemplify the use of standard {Riesz potential theory} on the unit open ball $\mathbb{B}_d: =\{x\in\mathbb{R}^d: |x| <1\}$, $\mathbb{R}^d\backslash \mathbb{B}_d$ and the sphere $ \mathbb{S}_{d-1}: = \{x\in\mathbb{R}^d: |x| =1\}$ with the, then, modern theory of {\it potential analysis for Markov processes}. 
 
 Following initial observations of \cite{L72}, with the occasional sporadic contributions such as \cite{Kiu, VG, GV}, an alternative understanding of L\'evy stable processes through the {\it  theory of  self-similar Markov processes} has prevailed in the last decade or more. This point of view offers deeper probabilistic insights into some of the aforementioned potential analytical relations; see for example \cite{BY02, BC02, CC06b, CC06a, CKP09, P09, Patie,  BZ, Patie, KPW, KKPW, KW, KP, Deep1, Deep2, Deep3, ACGZ}.

In this review article, we will rediscover many of the aforementioned classical identities in relation to the unit ball by combining elements of these   two theories, which have otherwise been  largely separated by decades in the literature. We present a dialogue that appeals as much as possible to path decompositions. Most notable in this respect is the Lamperti-Kiu decomposition of self-similar Markov processes given in \cite{Kiu, CPR, ACGZ} and the Riesz--Bogdan--\.Zak transformation given in \cite{BZ}.

 Some of the results and proofs we give are known (marked $\heartsuit$), some are  mixed with new results or methods, respectively, (marked $\diamondsuit$) and some are completely new (marked $\clubsuit$).
We assume that the reader has a degree of familiarity with the bare basics of L\'evy processes but, nonetheless, we often include reminders of standard material that can be found in e.g. \cite{bertoin}, \cite{Sato} or \cite{Kypbook}. 
\end{abstract}

\maketitle
\newpage

{\small\tableofcontents}
\newpage

\part{Stable  processes, self-similar Markov processes and MAPs}
\setcounter{section}{0}

In this review article, we give an extensive overhaul of some aspects of the theory of strictly stable L\'evy processes as seen from the point of view of self-similarity. Our presentation  takes account of a sixty-year span of literature. As we walk the reader through a number of classical and recent results, they will note that the statements of results and their proofs are marked with one of three symbols. For statements or proofs which are known, we use the mark $\heartsuit$; for those statements or proofs which are known, but mixed with new results or methods (respectively), we use the mark $\diamondsuit$; for those statements or proofs which are completely new, we use the mark $\clubsuit$.

\section{Introduction}\label{intro}

We can define the family of stable L\'evy  processes as being those $\mathbb{R}$-valued stochastic processes which lie at the intersection of the class of L\'evy processes and the class of self-similar processes. Whilst the former class insists on c\`adl\`ag paths and stationary and independent increments, a process $X = (X_t, t\geq 0)$ with probabilities $\mathbb{P}_x$, $x\in\mathbb{R}^d$,  in the latter class has the property that there exists a {\it stability index} $\alpha$ such that, for $c>0$ and $x\in\mathbb{R}^d\setminus\{0\}$,
\begin{equation}
\text{under } \mathbb{P}_x, \text{ the law of }(cX_{c^{-\alpha}t}, t\geq 0) \text{ is equal to } \mathbb{P}_{cx}.  
\label{scaling property}
\end{equation}
What we call stable L\'evy processes here  are known in the literature as {\it strictly stable L\'evy processes}, but for the sake of brevity we will henceforth refer to them as just {\it stable processes} as there will be no confusion otherwise.

Living in the intersection of self-similar Markov processes and L\'evy processes, it turns out that  stable processes are useful prototypes for exemplifying the theory of  both fields as well as for examining phenomena of processes with path discontinuities and how they differ significantly from e.g. the theory of diffusions; see for example \cite{DK18} for recent results showing discrepancies with Feller's classical boundary classification for diffusions in one-dimension for stochastic differential equations driven by stable processes. 

In a large body of the literature concerning stable processes it  is usual to restrict the study of stable processes to those that are distributionally isotropic. That is to say, those processes for which, for all orthogonal transforms $B:\mathbb{R}^d\to\mathbb{R}^d$ and $x\in\mathbb{R}^d$, 
\[
\text{under } \mathbb{P}_x, \text{ the law of }(BX_t, t\geq 0)\text{ is equal to }\mathbb{P}_{Bx}.
\]
As such, we talk about {\it isotropic stable processes}. The restriction to the isotropic subclass already presents sufficient challenges, whilst allowing for one more mathematical convenience. That said, we can and will drop the assumption of isotropy, but in dimension $d=1$ only.  For dimension $d\geq 2$, we will always work in the isotropic setting. It remains to be seen how rich the development of the literature on stable processes will become in the future without the assumption of isotropy in higher dimensions.

It turns out that stable processes necessarily have index of stability $\alpha$ which lies in the range $(0,2]$. The case $\alpha = 2$, in any dimension, pertains to Brownian motion and therefore necessarily has continuous paths. Somewhat dogmatically we will exclude this parameter choice from our discussion for the simple reason that we want to explore phenomena which can only occur as a consequence of jumps. That said, we pay occasional lip service to the Brownian setting. 

\medskip

Henceforth, $X:= (X_t:t\geq 0)$, with probabilities $\mathbb{P}_x$, $x\in\mathbb{R}^d$,  will denote  a $d$-dimensional isotropic stable process  with index of similarity $\alpha\in(0,2)$. For convenience, we will always write $\mathbb{P}$ in place of $\mathbb{P}_0$.

\subsection{One-dimensional stable processes}

When $d = 1$,  as alluded to above, we will be more adventurous and permit anistropy. It is therefore inevitable that a second parameter is needed which will code the degree of asymmetry. To this end, we introduce  the positivity parameter $\rho:=\mathbb{P}(X_1\geq 0)$. This parameter as well as  $\hat\rho:=1-\rho$ will appear in most identities. A good case in point in this respect is the L\'evy measure of $X$ and another is its characteristic exponent. The former  is written
\begin{equation}
 \Pi(\d x): =  \frac{\Gamma(\alpha+1)}{\pi} \left\{
\frac{ \sin(\pi \alpha \rho) }{ x^{\alpha+1}} \mathbf{1}_{(x > 0)} + \frac{\sin(\pi \alpha \hat\rho)}{ {|x|}^{\alpha+1} }\mathbf{1}_{(x < 0)}
 \right\}\d x,
\qquad  x \in \mathbb{R},
\label{stabledensity3}
\end{equation}
and the latter takes the form 
\begin{align}\label{Psi_alpha_rho_parameterization}
\Psi(\theta) &:=-\frac{1}{t}\log\mathbb{E}[{\rm e}^{{{\rm i}\theta X_t}}]\notag\\
&= |\theta|^\alpha ({\rm e}^{\pi\iu\alpha(\frac{1}{2} -\rho)} \mathbf{1}_{(\theta>0)} + {\rm e}^{-\pi\iu\alpha (\frac{1}{2} - \rho)}\mathbf{1}_{(\theta<0)}), \qquad \theta\in\mathbb{R}, t>0.
\end{align}
For a derivation of this exponent, see Exercise 1.4 of \cite{Kypbook} or Chapter 3 of \cite{Sato}.

It is well known that the transition semigroup of $X$ has a density with respect to Lebesgue measure. That is to say, 
\begin{equation}
{\texttt p}_t(x): = \frac{1}{2\pi}\int_{\mathbb{R}}{\rm e}^{-\iu zx}{\rm e}^{-\Psi(z)t}\d z, \qquad x\in\mathbb{R}, t> 0,
\label{pinverseF}
\end{equation}
exists and satisfies 
\begin{equation}
\mathbb{P}_x(X_t \in\d y) = \texttt{p}_t (y-x)\d y
\label{textttp}
\end{equation}
for all $x,y\in\mathbb{R}$ and $t> 0$.
\medskip

In one dimension, we are predominantly interested in the fluctuations of $X$ when it moves both in an upward and a downward direction. The aforesaid exclusion can  be enforced by requiring throughout that

\bigskip

\begin{center}
{both $\alpha\rho$ and $\alpha\hat\rho$ belong to $(0,1)$.} 
\end{center}
\bigskip

\noindent This excludes both the case stable subordinators, the negative of stable subordinators (when $\rho=1$ or $\hat\rho=1$, respectively) and spectrally   negative and positive processes (when $\alpha\rho = 1$ or $\alpha\hat\rho=1$, respectively).

%Stable processes in dimension one can be recurrent or transient depending on the value of $\alpha$. More precisely, if $\alpha\in(0,1)$ then, just as with dimension $d\geq 2$, $\lim_{t\to\infty}|X_t| = \infty
%$ almost surely. When $\alpha =1$ we can say that $X$ is set-recurrent, but not point-recurrent, in the sense that, $\mathbb{P}_x$-almost surely, from all points of issue $x\in\mathbb{R}\backslash\{0\}$,
%\[
%\limsup_{t\to\infty}|X_t| = \infty \text{ and }\liminf_{t\to\infty}|X_t| =0,
%\]
%where as $\mathbb{P}_x(\tau^{\{0\}}<\infty) = 0$.
%Finally, when $\alpha\in(1,2)$, we have point-recurrence for $X$, meaning that $\mathbb{P}_x(\tau^{\{0\}}<\infty) = 1$, for all $x\in\mathbb{R}$.

\bigskip

Stable processes are one of the few known classes of L\'evy processes which reveal 
a significant portion of the general theory in explicit detail. We will spend a little time here recalling some of these facts. 

As one dimensional L\'evy processes with two-sided jumps,  we can talk about their running maximum process $\overline{X}_t: = \sup_{s\leq t}X_s$, $t\geq 0$, and their running minimum process $\underline{X}_t : = \inf_{s\leq t}X_s$. As is the case with all L\'evy processes, it turns out that the range of $\overline{X}$ agrees with that of a subordinator, say $H$, called the ascending ladder height processes. By subordinator we mean a L\'evy process with increasing paths and we allow the possibility that it is killed at a constant rate and sent to a cemetery state, taken to be $+\infty$. The inclusion of a killing rate depends on whether the underlying L\'evy process drifts to $-\infty$ (resp. $+\infty$), in which case $\overline{X}_\infty<\infty$ (resp. $-\underline{X}_\infty <\infty$)  almost surely or oscillates (in which case $\overline{X}_\infty = -\underline{X}_\infty = \infty$).  When the process drifts to $-\infty$, the killing rate is strictly positive, otherwise it is zero (i.e. no killing).

Roughly speaking the subordinator $H$ can be thought of as the trajectory that would remain if we removed the sections of path, or better said, if we removed the excursions of $X$, which lie between successive increments of $\overline{X}$ and the temporal gaps created by this removal procedure are closed. Similarly the range of $-\underline{X}$ agrees with that of a subordinator, say $\hat{H}$, called the descending ladder height process. Naturally the two processes $H$ and $\hat{H}$ are corollated. For more background on the ladder height processes see  Chapter VI of \cite{bertoin} or Chapter 6 of \cite{Kypbook}.

Suppose that we denote the Laplace exponent of $H$ by $\kappa$. To be precise, 
\[
\kappa(\lambda) = \frac{1}{t}\log\mathbb{E}[{\rm e}^{-\lambda H_t}], \qquad t, \lambda\geq 0,
\]
where, necessarily, the exponent $\kappa$ is a Bernstein function with general form
\begin{equation}
\kappa(\lambda) = q + \delta\lambda + \int_{(0,\infty)} (1- {\rm e}^{-\lambda x})\Upsilon({\rm d}x),\qquad \lambda\geq 0,
\label{bernstein}
\end{equation}
Here, $q\geq 0$ is the killing rate, $\delta\geq 0$ is a linear drift and $\Upsilon$ is the L\'evy measure of  $H$. (See \cite{SSV} for more on Bernstein functions in the context of subordinators.)
The Laplace exponent of $\hat{H}$, which we shall henceforth denote by $\hat\kappa$, is similarly described.
Note that the Laplace exponent of both the ascending and descending ladder height processes can be extended analytically to $\{z\in\mathbb{C}:\re(z)\geq 0 \}$.

It is a remarkable fact that the  characteristic exponent of every L\'evy process  factorises  into two terms, each one exposing the Laplace exponent  of the ascending and descending ladder height processes respectively. 
%More generally, this is also true if we take the characteristic exponent of a L\'evy process when killed and sent to a cemetery state at an independent and exponentially distributed random time with rate $p\geq0$.
That is to say, up to a multiplicative constant, we have 
\begin{equation}
\Psi({z}) = \kappa(-{\rm i}{z})\hat\kappa({\rm i}{z}), \qquad {z} \in\mathbb{R}.
\label{SWHF}
\end{equation}
This equality is what is commonly referred to as the {\it Wiener--Hopf factorization}; see Chapter VI of \cite{bertoin}  or Chapter 6 of \cite{Kypbook}.
%The factorization is unique in class of pairs $(\kappa,\hat\kappa)$, each taking the form of a Laplace exponent of a (killed) subordinator  such that $\kappa(-\i z)$ can be analytically extended to the upper half of the complex plane and $\hat\kappa(\i z)$ can be extended to the lower half of the complex plane.

In the stable case, for $\theta \geq 0$, 
\[
\kappa(\theta) = \theta^{\alpha\rho}\text{ and }\hat\kappa(\theta) = \theta^{\alpha\hat\rho}.
\]
Notably,  $H$  is a stable subordinator with no killing and no drift (and hence, by exchanging the roles of $\rho$ and $\hat\rho$, the same is true of $\hat H$). It is a pure jump subordinator with L\'evy intensity 
\begin{equation}
\Upsilon(\d x) =\frac{\alpha\rho}{\Gamma(1-\alpha\rho)} \frac{1}{x^{1+\alpha\rho}}\d x,
\qquad x>0
\label{ladderjump}
\end{equation}
(again, with the same being true for $\hat H$, albeit with the roles of $\rho$ and $\hat\rho$ reversed).
An explicit understanding of the Wiener--Hopf factorisation is important from the perspective of understanding how one-dimensional stable process cross levels for the first time, the precursor of the problem we will consider here in higher dimensions. Indeed, consider the first passage time $\tau^+_a = \inf\{t>0: X_t> a\}$, for $a>0$;
\[
\text{the overshoot of $X$ above $a$ is defined as $X_{\tau^+_a}-a$}. 
\]
As the reader may already guess, the scaling property of stable processes suggests that, to characterise overshoots above all levels $a$, it is sufficient to characterise  overshoots above level 1. Indeed, for each constant $c>0$, suppose we define ${X}^c_t =cX_{c^{-\alpha}t}$, $t\geq 0$. Then
\begin{align*}
\tau^+_a &= \inf\{t\geq 0: X_t >a\}\\
&= \inf\{a^{\alpha}s>0 : a^{-1}X_{a^{\alpha} s} >1\}\\
&= a^{\alpha}\inf\{s>0: X^{1/a}_{s}>1\}.
\end{align*}
Accordingly, we see that $(X_{\tau^+_a}-a)/a$ is equal in distribution to $X_{\tau^+_1}-1$.

It should also be clear that the overshoot of $(X, \mathbb{P})$ above level 1 agrees precisely with the overshoot of $(H, \mathbb{P})$ above the same level.
As alluded to above, the simple and explicit form of the ascending ladder processes  offers quite a rare opportunity to write down the overshoot distribution of $X$ above $1$.

A classical computation using the compensation formula for the Poisson point process of jumps tells us that for a bounded measurable function $f$ and $a\geq 0$,
\begin{align}
\mathbb{E}[f(H_{\tau^+_1 } -1 )]  &= \mathbb{E}\left[\sum_{t>0} f(H_t-1)\mathbf{1}_{(H_{t-}<1)}\mathbf{1}_{(H_{t-}+\Delta H_{t}>1)}\right]\notag\\
&=\mathbb{E}\left[\int_0^\infty \int_{(0,\infty)}f(H_t+z-1)\mathbf{1}_{(H_{t}<1)}\mathbf{1}_{(H_{t}+z>1)}\Upsilon(\d z)\d t\right]\notag\\
&=\int_{[0,1)} U(\d y)\int_{(1-y,\infty)}\Upsilon (\d z)f(y+z-1),
\label{finddistn}
\end{align}
where $U(\d x) = \int_0^\infty  \mathbb{P}(H_t \in \d x)\d t$, $x\geq 0$, is the potential of $H$. 
The identity in \eqref{finddistn} was first proved in \cite{Kest2, Kest1}, see also \cite{Horo}.
Noting by Fubini's Theorem that 
\begin{equation*}
\int_{[0,\infty)}{\rm e}^{-\beta x}U(\d x) = \int_0^\infty  \mathbb{E}[{\rm e}^{-\beta H_t} ]\d t = \frac{1}{\beta^{\alpha\rho}}, \qquad \beta\geq 0,
\end{equation*}
we easily invert to find that 
\begin{equation}
U(\d x) =\Gamma(\alpha\rho)^{-1}x^{\alpha\rho -1} \d x, 
\qquad x>0.
\label{subordinatorpotential}
\end{equation}
 Together with \eqref{ladderjump}, we can evaluate the distribution of the overshoot over level 1 in \eqref{finddistn}. The scaling property of overshoots now gives us the following result,  which is originally due to \cite{Dynkin1961} and \cite{Lamperti62} and which one may refer to in potential analytic terms as the Poisson kernel on the half-line.
\begin{lemma}[$\heartsuit$]\label{precorollary}For all $u,a>0$,
\[
\mathbb{P}(X_{\tau^+_a} - a\in \D u) = \frac{\sin(\pi\alpha)}{\pi} \left(\frac{u}{a}\right)^{-\alpha}\left(\frac{1}{a+u}\right)\D u.
\]
\end{lemma}
It is not difficult to compute the total mass on $(0,\infty)$ of the distribution above using the beta integral to find that it is equal to unity. Accordingly, the probability that $X$ {\it creeps} over the level $a$ is zero, that is $\mathbb{P}(X_{\tau^+_{a}} =a)=0$.

\bigskip

We are also interested in the potential of the one-dimensional stable process. That is to say, we are interested in the  potential measure  

\begin{equation}
U(x, \d y) = \int_0^\infty\mathbb{P}_x( X_t \in \d y)\d t
=\left(\int_0^\infty \texttt{p}_t(y-x)\d t\right)\d y, \qquad x,y\in\mathbb{R}.
\label{potential1anddd}
\end{equation}
In order to discuss the potential, we first need to recall various notions from the theory of L\'evy processes, looking in particular at the properties of transience and recurrence as well as point-recurrence. 

Thanks to Kolmogorov's zero-one law for events in the tail sigma algebra of the natural filtration $\sigma(X_s, s\leq t)$, $t\geq0$, for each fixed $a>0$, the convergence of the integral $\int_{0}^{\infty }\mathbf{1}_{\left(
|X_t|<a\right) }{\D}t$ occurs with probability zero or one. Moreover, thanks to the scaling property of stable processes, if $\int_{0}^{\infty }\mathbf{1}_{\left(
|X_t|<a\right) }{\D}t<\infty $ almost surely for some $a>0$ then this integral is almost surely convergent   for all $a>0$. In that case we say that $X$ is called {\it transient}. A similar statement holds if the integral is divergent for some, and then all, $a>0$, in which case we say that $X$ is {\it recurrent}.
The point of issue of $X$ is not important here thanks to spatial homogeneity. 
 
It turns out that, more generally, this is the case for all L\'evy processes. This is captured by the following classic analytic dichotomy; see for example  \cite{Kingman64, PortStoneRT}.

\begin{theorem}[$\heartsuit$]\label{portstone}
For a L\'evy process with characteristic exponent
$\Psi$, it is transient if and only if, for some sufficiently
small $\varepsilon>0$,
\begin{equation}
\int_{|{z}|<\varepsilon}\re\left(\frac{1}{\Psi({z})}\right){\D}{z}
<\infty,
\label{transrecintegral}
\end{equation}
and otherwise it is recurrent.
\end{theorem}
\noindent Probabilistic reasoning also leads to the following interpretation of the dichotomy.
\begin{theorem}[$\heartsuit$]\label{only first two bits}
Let  $Y$ be any L\'evy process.
\begin{enumerate}
\item[(i)] We have  transience if and only if
\[
\lim_{t\to\infty} |Y_t| = \infty
\]
almost surely.
\item[(ii)] If  $Y$ is not a compound
Poisson process,
 then we have   recurrence if and only if,   for all $x\in\mathbb{R},$
\begin{equation}
\liminf_{t\to\infty}|Y_t - x|=0
\label{not the only condition}
\end{equation}
almost surely.  %{\color{red} Is this true in high dimension?}
\end{enumerate}
\end{theorem}

Back to the stable setting, on account of the fact that 
\[
\int_{(-\varepsilon, \varepsilon)}\re\left(\frac{1}{\Psi({z})}\right)\D {z} \propto 
\int_{(-\varepsilon, \varepsilon)}\frac{1}{|{z}|^\alpha}\D{z},
\]
it follows from Theorem \ref{portstone}  that $X$ is transient whenever $\alpha \in(0,1)$ and recurrent when $\alpha\in[1,2)$. It is worth remarking here that  a L\'evy   process which is recurrent cannot drift to
$\infty$ or $-\infty$, and therefore must oscillate  and we see this consistently with stable processes. 
On the other hand, %whilst it is clear that a process drifting to $\infty $ or $-\infty $ is transient, 
an oscillating process is not necessarily 
recurrent. A nice  example of this phenomenon is 
provided by the case of  a symmetric stable process
 of index $0<\alpha<1$. % Note that this differs from the case of a one-dimensional linear Brownian motion, where  oscillation and drifting coincide as do the definitions of transience and drifting to $\pm\infty$.

Returning to the issue of  \eqref{potential1anddd}, it is clear that the potential makes no sense for $\alpha\in[1,2)$. That is to say, for each $x\in\mathbb{R}$, $U(x, A)$ assigns infinite mass to each non-empty open set $A$. When $\alpha\in(0,1)$, a general result for L\'evy processes tells us that transience is equivalent to the existence of \eqref{potential1anddd} as a finite measure. In the stable case, we can verify this directly thanks to the following result.

\begin{theorem}[$\heartsuit$]\label{1dfreepotential} Suppose that $\alpha\in(0,1)$. The potential of $X$ is absolutely continuous with respect to Lebesgue measure. Moreover,  $U(x, \d y) = u(y-x)\d y$, $x,y\in\mathbb{R}^d$, where 
\begin{equation}
u(x) = \Gamma(1-\alpha)\left(\frac{\sin(\pi\alpha\rho)}{\pi}\mathbf{1}_{(x>0)}+\frac{\sin(\pi\alpha\hat\rho)}{\pi}\mathbf{1}_{(x<0)}\right)|x|^{\alpha-1}, \qquad x\in\mathbb{R}.
\label{eq:freepotential1d}
\end{equation}

\end{theorem}
\begin{proof}[Proof \emph{($\heartsuit$)}] The proof we give here is classical.
Let us first examine the expression for $u$ on the positive half-line. For positive, bounded measurable $f: [0,\infty)\to[0,\infty)$, which satisfies \[\int_{\mathbb{R}} f(x)|x|^{\alpha - 1}\d x<\infty,\]
we have 
 \begin{align}
\int_0^\infty f(x)U(0, \d x)
&=\int_0^\infty \int_{0}^\infty f(x)\mathbb{P}(X_t\in \d x)\d t \notag\\
&=
 \frac{1}{2\pi}\int_0^{\infty} f(x) \int_{\mathbb{R}}\int_{0}^{\infty}  {\rm e}^{-|z|^\alpha t-\i z x} \d t \, \d z \, \d x\notag \\
 &=
 \frac{1}{2\pi}\int_0^{\infty} f(x) \int_{\mathbb{R}} |z|^{-\alpha} {\rm e}^{-\i z x}  \, \d z \, \d x\notag \\
 &= \frac{ {\rm e}^{-\pi\iu\alpha(\frac{1}{2} -\rho)} }{2\pi} \int_0^{\infty} f(x)x^{\alpha-1} \int_{0}^{\infty}  y^{-\alpha}  {\rm e}^{-\i y} \d y \, \d x\notag\\
 &\hspace{0.5cm}-\frac{ {\rm e}^{\pi\iu\alpha(\frac{1}{2} -\rho)} }{2\pi} \int_0^{\infty} f(x)x^{\alpha-1} \int_{0}^{\infty}  y^{-\alpha} {\rm e}^{\i y}  \d y \, \d x.
 \label{doagain1-e}
 \end{align}
Using that the Mellin transform of ${\rm e}^{\pm\iu y}$, $y\geq0$, at $1-\alpha$ is known to be equal to 
$ \Gamma(1-\alpha){\rm e}^{\pm\iu \pi(1-\alpha)/2}$,
  \begin{align*}
  \int_0^\infty f(x)U(0, \d x)
&=
 \Gamma(1-\alpha) \frac{ {\rm e}^{-\pi\iu(\frac{1}{2} -\alpha\rho)} }{2\pi} \int_0^{\infty} f(x)x^{\alpha-1}  \, \d x\\
 &\hspace{0.5cm}-\Gamma(1-\alpha)\frac{ {\rm e}^{\pi\iu(\frac{1}{2} -\alpha\rho)} }{2\pi} \int_0^{\infty} f(x)x^{\alpha-1}  \, \d x\\
 &=\Gamma(1-\alpha) \frac{\sin(\pi\alpha\rho)}{\pi}\int_0^{\infty} f(x)x^{\alpha-1}  \, \d x,
 \end{align*}
as required. A similar proof when $x<0$ gives the second term in \eqref{eq:freepotential1d}.
 \end{proof}

Write $Y$ for a general L\'evy process with law $  {\rm P}$, when issued from the origin.
Transience and recurrence  in the sense of the ${\rm P}$-almost sure convergence or divergence of $\int_{0}^{\infty }\mathbf{1}_{\left(
|Y_t|<a\right) }{\D}t$  is a notion that pertains to the time spent visiting (open) sets. A finer notion of transience and recurrence can be developed in relation to visiting individual points.

We say that a general L\'evy process  {\it can hit a point} $x\in\mathbb{R}$ if 
\[
\mathrm{P}(Y_t = x \text{ for at least one  }t > 0) > 0.
\]
This  notion is of course well defined for all L\'evy processes. In the case of a stable process, the scaling property means that hitting a point $x$ with positive probability is equivalent to hitting any other point with positive probability. It turns out that this is generally the case for all L\'evy processes, with the exception of compound Poisson processes, which may be troublesome in this respect if their jump distribution has lattice support. 
The following theorem, taken from \cite{Kestenpoint}, applies for the general class of L\'evy processes; see also \cite{Bret}. 
\begin{theorem}[$\heartsuit$]
\label{hitting points}
Suppose that a general L\'evy process is not a compound Poisson process and has characteristic exponent $\Psi$. Then it can hit points if and only if
\begin{equation}
\label{local-time}
\int_{\mathbb{R}} \re\left( \frac {1}{1 + \Psi({z})}\right) \D z < \infty. 
\end{equation}
\end{theorem}
A straightforward comparison of $\re\left(({1 + \Psi({z})})^{-1}\right)$ with $\mathbf{1}_{(|x|<1)} + \mathbf{1}_{(|x|\geq 1)}|{z}|^{-\alpha}
$ shows that \eqref{local-time} holds if and only if $\alpha\in(1,2)$.
Therefore, referring to Theorem \ref{hitting points},  the  process $X$ can hit points almost surely if and only if $\alpha\in(1,2)$.
Coupled with recurrence, it is thus clear that, when $\alpha\in(1,2)$, $\mathbb{P}(X_t = x \text{ for at least one  }t > 0)=1$ for all $x\in\mathbb{R}$, showing point-recurrence.
This leaves the case of $\alpha = 1$ which is recurrent but not point-recurrent.

\subsection{Higher dimensional stable processes} Recall that in dimension $d\geq 2$, we insist that $X$ is isotropic. This carries the consequence that $X$ is a $\mathbb{R}^d$-valued L\'evy process with jump measure satisfying
\begin{align}\label{jumpmeasure}
\Pi(B)&=\frac{2^{\alpha}\Gamma(({d+\alpha})/{2})}{\pi^{d/2}|\Gamma(-{\alpha}/{2})|}\int_{B}\frac{1}{|y|^{\alpha+d}}{\rm d}y\notag\\
& =
2^{\alpha-1}\pi^{-d}\frac{\Gamma((d+\alpha)/2)\Gamma(d/2)}{\big|\Gamma(-\alpha/2)\big|}\int_{\mathbb{S}_{d-1}}r^{d-1}\sigma_1({\rm d}\theta)\int_0^\infty\mathbf{1}_{B}(r\theta)\frac{ 1}{r^{\alpha+d}}\d r,
\end{align}
for $B\in\mathcal{B}(\mathbb{R})$,
where $\sigma_1(\d \theta)$ is the surface measure on $\mathbb{S}_{d-1}$ normalised to have unit mass and the change in the constant in the second equality comes from the Jacobian when changing from Cartesian to generalised polar coordinates (see \cite{blum}).
Equivalently, this means $X$ is a $d$-dimensional  L\'evy process with characteristic exponent $\Psi(\theta) = -\log\mathbb{E}({\rm e}^{{\rm i}\langle\theta, X_1\rangle})$ which satisfies
\[
\Psi(\theta) = |\theta|^\alpha, \qquad \theta\in\mathbb{R}^d.
\]
Stable processes in dimension $d\geq 2$ are transient in the sense that 
\[
\lim_{t\to\infty}|X_t| = \infty
\]
almost surely, from any point of issue.

Just as in the one-dimensional case, a quantity that will be of specific interest is the potential $U(x,\d y)$, which is defined just as in \eqref{potential1anddd} albeit that, now, $x,y\in\mathbb{R}^d$.
The following is classical, found, for example, in \cite{BGR} if not earlier; see also the discussion in Example 3.4 of \cite{BH} or Section 1.1 of \cite{Landkof}.
\begin{theorem}[$\heartsuit$]\label{freepotential}
For dimension $d\geq 2$, the potential of $X$ is absolutely continuous with respect to Lebesgue measure, in which case, remembering  spatial homogeneity, its density satisfies $U(x, \d y) = u(y-x)\d y$, $x,y\in\mathbb{R}^d$, where 
\[
u(z) = 2^{-\alpha}\pi^{-d/2}\frac{\Gamma((d-\alpha)/2)}{\Gamma(\alpha/2)}|z|^{\alpha -d}, \qquad z\in\mathbb{R}^d. 
\]
\end{theorem}
\begin{remark} Combining the above theorem with Theorem \ref{1dfreepotential}, for $\rho =1/2$, one can in fact state Theorem \ref{freepotential}  more generally with the qualification that $d>\alpha$. The reader will also note that the proof given for Theorem \ref{freepotential} below works equally well when $d = 1$, $\alpha\in(0,1)$ and  $\rho =1/2$.
\end{remark}
\begin{proof}[Proof of Theorem \ref{freepotential} \emph($\heartsuit$)]
The proof we give here is also classical and  taken from  p187 of \cite{BH}.
Fix $\alpha\in(0,2)$ and 
suppose that $(S_t, t\geq 0)$ is a stable subordinator with index $\alpha/2$. If we write $(B_t, t\geq 0)$ for a standard $d$-dimensional Brownian motion, then it is known that $X_t: = \sqrt{2}
B_{S_t}$, $t\geq 0$,is a stable process with index $\alpha$. Indeed, its stationary and independent increments and scaling, in the sense of \eqref{scaling property},  are inherited directly from those of $S$ and $B$,  and are easy to verify. Note, moreover, that
\[
\mathbb{E}[{\rm e}^{{\rm i}\langle\theta,  X_t\rangle}] = \mathbb{E}\left[{\rm e}^{-|\theta|^2 S_t}\right] = {\rm e}^{-|\theta|^\alpha t}, \qquad \theta\in\mathbb{R}^d.
\]
Now note that, for bounded and measurable $f:\mathbb{R}^d\to[0,\infty)$,  which satisfies $\int_{\mathbb{R}^d}f(x)|x|^{\alpha-1}\d x<\infty$, 
\begin{align*}
\mathbb{E}\left[\int_0^\infty f(X_t)\d t\right]
&=\mathbb{E}\left[\int_0^\infty f(B_{S_t})\d t\right]\\
&=\int_0^\infty \d s  \int_0^\infty \d t\,\mathbb{P}(S_t\in \d s)\int_{\mathbb{R}} \mathbb{P}(B_s\in \d x) f(\sqrt{2}x)\\
&=\frac{1}{\Gamma(\alpha/2)\pi^{d/2}2^d} \int_\mathbb{R}\d y \int_0^\infty \d s \,{\rm e}^{- |y|^2/4 s} s^{-1+(\alpha-d)/2 }f(y)\\
&= \frac{1}{2^{\alpha}\Gamma(\alpha/2)\pi^{d/2}} \int_\mathbb{R}\d y \, |y|^{(\alpha-d)} \int_0^\infty \d u \, e^{-u} u^{-1+(d-\alpha/2)}f (y)\\
&=\frac{\Gamma((d-\alpha)/2)}{2^{\alpha }\Gamma(\alpha/2)\pi^{d/2}}\int_\mathbb{R}\d y \, |y|^{(\alpha-d)} f(y),
\end{align*}
where we have used the expression for the potential of $S$ as in \eqref{subordinatorpotential} (albeit replacing the index $\alpha\rho$ there by $\alpha$).
This completes the proof. 
\end{proof}

The final thing to mention in this section is the issue of hitting points for stable processes in dimension $d\geq2$.
It is known that if the condition \eqref{local-time} fails, then  points cannot be hit from Lebesgue-almost every point of issue.
Note in higher dimensions that the effect of the Jacobian comes into play when we estimate the integral in \eqref{local-time}. Indeed, one can easily make the comparison with the integral
\[
\int_{|x|>1} \frac{1}{|z|^\alpha}\d z \asymp \int_{1}^\infty\int_{\mathbb{S}_{d-1}} \frac{1}{r ^\alpha}r^{d-1}\sigma_1(\theta)\d z 
 \asymp \int_{1}^\infty r ^{d-\alpha -1}\d r= \infty,
\]
as $d\geq2$. Here, $a\asymp b$ means that $a/b$ is bounded form above and below by a strictly positive constant. Hence, for stable process, points cannot be hit from Lebesgue-almost every point of issue. However, with a little work one can upgrade this to the statement that points cannot be hit from any point of issue. The subtleties of this can be found for, example, in Chapter 8, Section 43 of \cite{Sato}.

\section{Positive self-similar Markov processes and stable processes}

In this section we introduce one of the key mathematical tools that we shall use to analyse stable processes: {\it positive self-similar Markov processes}. We shall often denote this class by {\it pssMp} for convenience. Shortly we will give the definition of these processes and their pathwise characterisation as space-time-changed L\'evy processes through the {\it Lamperti transform}. Thereafter, we spend the rest of the section exploring a number of examples of pssMp which can be constructed through path transformations of stable processes. Each of these examples of pssMp turn out to be intimately connected, through the Lamperti transform, to a different L\'evy process belonging to the (extended) hypergeometric class. 

\subsection{The Lamperti transform}\label{sect:Lamperti}

Let us begin with a definition of the fundamental class of processes that will dominate our analysis.

\begin{definition}\label{pssMpdef}\rm
A $[0,\infty)$-valued Feller   process $Z = (Z_t , t\geq 0)$ % which has paths that are almost surely right-continuous and  quasi-left-continuous
%\footnote{Recall that $Z$ is quasi-left-continuous if it has the following property:   For each  $\mathbb{F}$-stopping time $T$,  if there exists an increasing sequence of  $\mathbb{F}$-stopping times, $\{T_n{:}~n\geq 1\}$, satisfying    $\lim_{n\uparrow\infty}T_n = T$ almost surely, then  $\lim_{n\uparrow\infty}Z_{T_n}= Z_T$ almost surely on $\{T<\infty\}$.}\index{quasi-left-continuity}    
is called a {\it positive self-similar Markov process} if there
exists a constant $\alpha > 0$ such that, for any ${x}> 0$ and  $c >0$,
\begin{equation}\label{13scale}
\mbox{the law of $( c  Z_{ c ^{-\alpha}t}, t\ge0)$ under ${P}_{x}$ is
${P}_{ c  {x}}$,}
\end{equation}
where    ${P}_{x}$ is the law of $Z$ when issued from ${x}$. %\footnote{It is important to note that our definition of a positive self-similar Markov process differs slightly from what one normally finds in the literature. Where we have assumed that it is a strong Markov process with  right-continuous and quasi-left-continuous paths, a more usual assumption would be that it is a Markov process that satisfies the so-called Feller property. The latter assumption implies the former assumption.}
In that case, we refer to $\alpha$ as the {\it index of self-similarity}. (The reader should note that some authors prefer to refer to $1/\alpha$ as the index of self-similarity.)
\end{definition}

In his landmark paper, \cite{L72} showed that there is  a natural bijection between the class of exponentially killed L\'evy processes and positive self-similar Markov processes, up to a naturally defined lifetime, 
\[
\zeta = \inf\{t>0 : Z_t = 0\},
\] 
the first moment $X$ visits the origin. 
 Roughly speaking, this bijection shows that the property of self-similarity is interchangeable with the property of having stationary and independent increments through an appropriate space-time transformation.
 Below, we state this bijection as a theorem. 
  
  Let us first introduce some more notation.  Throughout this section, we shall use $\Xi: = (\Xi_t , t\geq 0)$ % with law $\mathbb{P}$ 
to denote a one-dimensional L\'evy process (not necessarily issued from the origin)
%, $\Xi: = \{\Xi_t: t\geq 0\}$, 
which is killed and sent to the cemetery state $-\infty$ at an independent and exponentially distributed random time, $\mathbf{e} = \inf\{t> 0 :\Xi_t  = -\infty\}$, with rate in $[0,\infty)$. As usual, we understand $\mathbf{e}$ in the {\it broader sense} of an exponential distribution, so that if its rate is $0$, then $\mathbf{e} = \infty$ with probability one, in which case there is no killing.
%As usual we understand the case that $q=0$ to correspond to the absence of killing.
% That is to say, we understand the random variable $\mathbf{e}_0$ to be valued at $+\infty$ with probability one.  

We will be interested in applying a time change to the process $\Xi$ by using its integrated exponential process, $I:  =\{I_t: t\geq 0\}$, where
\begin{equation}
I_t= \int_0^{t} {\rm e}^{\alpha \Xi_s}{\rm d}s, \qquad t\geq 0.
\label{withoutalpha}
\end{equation}
As the process $I$ is  increasing, we may define its limit, $I_\infty:=\lim_{t\uparrow\infty}I_t$,  whenever it exists. We are also interested in the inverse process of $I$: 
\begin{equation}
\varphi(t) = \inf\{s>0 :  I_s> t\},\qquad t\geq 0.
\label{inverse}
\end{equation}
As usual, we work with the convention  $\inf\emptyset =\infty$.

The following theorem introduces the celebrated Lamperti transformation\footnote{As a referee pointed out, different authors use different nomenclature for the Lamperti transformation. For example one may choose to call \eqref{13lampertirep} the Lamperti transform of $\Xi$. We prefer to use a slightly looser use of `Lamperti transformation' to mean the bijection between the class of positive self-similar Markov processes and (killed) L\'evy processes.},
which characterises all positive self-similar Markov processes. It was originally proved in \cite{L72}; see also Chapter 13 of \cite{Kypbook}. We omit the proof here as it is long and a distraction from our main objectives.

\begin{theorem}[The Lamperti transform $\heartsuit$]\label{iso-2} Fix $\alpha>0$.
\begin{itemize}
\item[(i)] If $(Z, P_x)$, $x>0$, is a positive self-similar Markov process with index of self-similarity $\alpha$, then up to absorption at the origin, it can   be represented as follows: %For ${x}>0$,
\begin{equation}
Z_t\mathbf{1}_{(t<\zeta)}  = \exp\{\Xi_{\varphi(t)}\}, \, \qquad t\geq 0,
\label{13lampertirep}
\end{equation}
such that $\Xi_0 = \log x$ and  either
\begin{itemize}
\item[(1)]  $P_x(\zeta%^{({x})}
=\infty) =1$  for all ${x}>0$, in which case, $\Xi$ is a L\'evy process satisfying $\limsup_{t\uparrow\infty}\Xi_t = \infty$,
\item[(2)] $P_x(\zeta <\infty \text{ and }Z_{\zeta-}=0)=1$ for all ${x}>0$, in which case $\Xi$ is a L\'evy process satisfying $\lim_{t\uparrow\infty}\Xi_t = -\infty$, or
\item[(3)] $P_x(\zeta<\infty \text{ and }Z_{\zeta-}>0)=1$ for all ${x}>0$, in which case $\Xi$ is a L\'evy process  killed at an independent and exponentially distributed random time. % with rate some $q>0$.
\end{itemize}
In all cases, we may identify $\zeta =  I_\infty$.

\item[(ii)] Conversely, for each $x>0$, suppose that 
$\Xi$ is a given (killed) L\'evy process, issued from $\log x$. Define
\begin{equation*}
Z_t %\mathbf{1}_{(t<\zeta)}
  = \exp\{\Xi_{\varphi(t)}\} \mathbf{1}_{(t< I_\infty)}, \qquad t\geq 0.
%\label{Lamperti-2}
\end{equation*}
%such that,
%\begin{equation}
%\varphi(t) = \inf\{s>0 :  I_s\geq t\}.
%\label{inverse}
%\end{equation}
%As usual we work with the standard definition $\inf\emptyset =\infty$. 
Then $Z$ defines a positive self-similar Markov process up to its absorption time $
\zeta =  I_\infty
$, which satisfies $Z_0=x$ and 
 has index $\alpha$.

%\begin{theorem}\label{xigiven}
%For $\Xi$ and $q$ given as above, for each $x>0$, the second Lamperti transform applied to $\Xi$,  $Z^{(x)}$, defines a  positive self-similar Markov process.
%\end{theorem}
\end{itemize}

\end{theorem}

It is tempting to immediately think of a stable process as an example of a positive self-similar Markov process, but, with the exception of a stable subordinator (which has been ruled out of this discussion), it fails against the criteria of positivity.  In fact, the case of  a stable subordinator  is precisely the example (and the only example) of a self-similar Markov processes  which was given in  \cite{L72}. It is possible, however,  to construct examples of positive self-similar Markov processes from path transformations of stable processes. 

In all of the examples that follow to the end of this section, we will take  $X$, with probabilities $\mathbb{P}_x$, $x\in\mathbb{R}$, to be a stable process with two-sided jumps.

\subsection{Stable processes killed on entering $(-\infty,0)$}\label{killed*}
This first example was introduced in detail in \cite{CC06b}; see also \cite{KRS}. To some extent, the former of these two references  marks the beginning of the modern treatment of stable processes through the theory of self-similar Markov processes.

Let us define, for ${x}>0$,
\begin{equation}
Z_t = X_t\mathbf{1}_{(\underline{X}_t \geq 0)}, \qquad t\geq 0,
\label{establishss}
\end{equation}
where $X$ is a stable process. It is straightforward to show that the pair $
((X_t,  \underline{X}_t), t\geq 0)
$
is a strong Markov process. Moreover, if we denote its probabilities by $\{\mathbb{P}_{({x}, s)}: \infty >{x}\geq s>-\infty\}$, then, for all $c>0$ and $\infty>{x}\geq s>-\infty$, 
\begin{equation}
\text{the law of }(c(X_{c^{-\alpha } t}, \underline{X}_{c^{-\alpha }t}), t\geq 0)\text{ under } \mathbb{P}_{({x}, s)}\text{ is }\mathbb{P}_{(c{x}, cs)}.
\label{pairscaling}
\end{equation}
See for example Exercise 3.2 in \cite{Kypbook}.
We see that, for ${x},c>0$, under $\mathbb{P}_x = \mathbb{P}_{(x,x)}$, 
\[
cZ_{c^{-\alpha}t} = cX_{c^{-\alpha}t}\mathbf{1}_{( \underline{X}_{c^{-\alpha}t} \geq 0)}, \qquad t\geq 0,
\]
and, thanks to the scaling \eqref{pairscaling}, this is equal in law to  $(Z,\mathbb{P}_{cx})$. With a little more work, it is not difficult to show that $Z$ also inherits the Markov property and the Feller property from $X$. It follows that \eqref{establishss} is a positive self-similar Markov process.
Note in particular, this example  falls into category (3) of Theorem \ref{iso-2} on account of the fact that stable processes pass almost surely into the lower half-line with a jump. Its Lamperti transform should therefore reveal a L\'evy process which is killed at a strictly positive rate. %This is precisely what was shown in \cite{CC06a}.

\begin{theorem}[$\heartsuit$]\label{Psi*thrm}
For the pssMp constructed by killing a stable process on first entry to $(-\infty,0)$, the underlying L\'evy process, $\xi^*,$ that appears through the Lamperti transform has characteristic exponent\footnote{Here and elsewhere, we use the convention that the characteristic exponent of a L\'evy process $Y$ with law ${\rm P}$ is given by $\Psi(z):=-\log{\rm E}[{\rm e}^{{\rm i}z Y_t}] $, $z\in\mathbb{R}$, $t>0$.} given by 
\begin{equation}
\Psi^*({z}) =  \frac{\Gamma(\alpha - \i {z})}{\Gamma(\alpha\hat\rho  -\i {z}  )}\times\frac{\Gamma(1+\i {z})}{\Gamma(1-\alpha\hat\rho +\i {z})}, \qquad {z}\in\mathbb{R}.
\label{psi*exponent}
\end{equation}
\end{theorem}

Since $\Psi^*(0) =\Gamma(\alpha)/(\Gamma(\alpha\hat\rho)\Gamma(1-\alpha\hat\rho))>0$, we conclude that $\xi^*$ is a killed L\'evy process.  Remarkably, Theorem \ref{Psi*thrm} provides an explicit example of a Wiener--Hopf factorisation, with the two relevant factors placed on either side of the multiplication sign in \eqref{psi*exponent}.
Moreover, the process $\xi^*$, often referred to as a {\it Lamperti-stable} process (see e.g. \cite{CPP11}), also has the convenient property that its L\'evy measure is absolutely continuous with respect to Lebesgue measure, and its density takes the explicit form
\begin{equation}
\pi^{*}(x) =  \frac{\Gamma(1+\alpha)}{\Gamma(\alpha\rho)\Gamma(1-\alpha\rho)}
 \frac{{\rm e}^{x}}{({\rm e}^{x}-1)^{1+\alpha}}, \qquad x>0,
 \label{pi*}
\end{equation}

\subsection{Censored stable process}\label{CSPsection} Recall  that $X$ is a stable process with two-sided jumps. Define the
occupation time of $(0,\infty)$ for $X$,
\[ A_t = \int_0^t \mathbf{1}_{(X_s > 0)} \, \d s ,\qquad t\geq 0,
 \]
and let 
\begin{equation}
\gamma(t) = \inf\{ s \ge 0 : A_s > t \},\qquad t\geq 0,
\label{inverseA}
\end{equation} be its right-continuous inverse.
Define a process $(\Zch_t)_{ t \ge 0}$ by setting $\Zch_t = X_{\gamma(t)}$, $t\geq 0$. This is the process
formed by erasing the negative components of the trajectory of $X$ and shunting together the remaining positive sections of path.\footnote{Censored stable processes were introduced in \cite{KPW}. In that paper, there was discussion of another family of path adjusted stable processes which are also called censored stable processes; see \cite{BBC}.}

We now make zero into an absorbing state. Define the stopping time
\begin{equation}
\tau^{\{0\}} = \inf\{ t > 0 : \Zch_t = 0 \} 
\label{censoredstable}
\end{equation}
and the process
\[ Z_t = \Zch_t \mathbf{1}_{(t < \tau^{\{0\}})} , \qquad t \ge 0 , \]
which is absorbed at zero.
We call the process $Z$ the {\it censored stable process}. Our claim is that this process is a positive self-similar Markov process.

We now consider the scaling property. For each $c>0$, define the rescaled process $(\check{Z}^c_t, t \ge 0)$
by $ \check{Z}^c_t = c\check{Z}_{c^{-\alpha} t}$, and, correspondingly, let $\gamma^c$ be defined such that 
\begin{equation}
 \int_0^{\gamma^c(t)} \mathbf{1}_{( X^c_s > 0)}\, \D  s = t,
 \label{changevariable}
 \end{equation}
 where $X^c_t = c X_{c^{-\alpha } t}$, $t\geq 0$.
By changing variable with $u = c^{-\alpha}s$ in (\ref{changevariable}) and   noting that $A_{\gamma(c^{-\alpha}t)} = c^{-\alpha}t$, a short
calculation shows that
\[ 
c^{\alpha} \gamma(c^{-\alpha}t) =  \gamma^c(t) . 
\]
For each $x,c>0$, we have under $\mathbb{P}_x$, 
\[
c\Zch_{c^{-\alpha}t}
  = c X_{\gamma(c^{-\alpha} t)}
=c X_{c^{-\alpha} \gamma^c(t)}
= X^c_{ \gamma^c(t)},\qquad t\geq 0.
\]
The right hand side above is equal in law to the process $(\check{Z}, \mathbb{P}_{cx})$,
which establishes self-similarity of $\check{Z}$. % and hence $Z$.
Note, moreover, that, for all $c>0$, if $T^c_0$ is the time to absorption in $\{0\}$ of $\check{Z}$, then 
\begin{equation}
T^c_0 =% \inf\{t>0 : \check{Z}^c_t = 0\} = 
\inf\{t> 0 : \check{Z}_{c^{-\alpha }t} = 0\} = c^\alpha \inf\{s>0 : \check{Z}_s = 0\} = c^\alpha \tau^{\{0\}}.
\label{stopscale}
\end{equation}
It follows that, for all $x,c>0$, under $\mathbb{P}_x$, 
$c Z_{c^{-\alpha }t}  = c \check{Z}_{c^{-\alpha }t}\mathbf{1}_{(c^{-\alpha }t< \tau^{\{0\}})}
$, $t\geq 0$, which is equal in law to $Z$ under $\mathbb{P}_{cx}$.

As a killed, time-changed Markov process, the censored stable process remains in the class of Markov processes.  It remains to show that $Z$  is Feller. Once again, we claim that the latter is easily verified through  Feller property of $X$. The reader is referred to Chapter 13 of \cite{KPW}, where the notion of the censored stable process in the self-similar Markov setting was first introduced.

We now consider the pssMp $Z$ more closely for different
values of $\alpha\in(0,2)$. Denote by $\stackrel{_\leadsto}{\xi} = \{\stackrel{_\leadsto}{\xi}_t : t\geq 0\}$  the L\'evy process associated to the censored stable process through the Lamperti transform. As mentioned previously, we know that, for $\alpha \in (0, 1]$, the stable process $X$ cannot hit points.  This implies that 
 $\tau^{\{0\}} = \infty$ almost surely, 
and so, in this case, $Z = \Zch$ and $\cenxi$ experiences no killing. 
Moreover, when $\alpha\in(0,1)$, the process $X$ is transient which implies that $Z$ has almost surely finite occupancy of any bounded interval, and hence $\lim_{t\to\infty}\cenxi_t = \infty$.
When $\alpha= 1$, the process $X$ is recurrent which, using similar reasoning to the previous case, implies that $\limsup_{t\to\infty}\cenxi_t = -\liminf_{t\to\infty}\cenxi_t = \infty$.
Meanwhile, for $\alpha \in (1,2)$,
$X$ can hit every point. Hence, we have, in particular, that  
 $\tau^{\{0\}} < \infty$. However, on account of the fact that the stable process must always cross thresholds by a jump and never continuously, the process $X$ must make infinitely many jumps across zero during any arbitrarily small period of time immediately prior to hitting zero. Therefore, for $\alpha\in(1,2)$,  $Z$ approaches zero continuously.

Calculating the characteristic exponent of $\cenxi$ is non-trivial, but was carried out by \cite{KPW}, leading to the following result.

\begin{theorem}[$\heartsuit$]\label{censoredpsithrm} For the pssMp constructed by censoring the stable process in $(-\infty,0)$, the underlying L\'evy process $\cenxi$ that appears through the  Lamperti transform has characteristic exponent given by 
\begin{equation}
 \stackrel{_{\leadsto}}{\Psi}(z)=
   \frac{\Gamma(\alpha\rho - \iu{z})}{\Gamma(-\iu{z})}
    \frac{\Gamma(1 - \alpha\rho + \iu{z})}{\Gamma(1 - \alpha + \iu{z})},\qquad z\in\mathbb{R}.
\label{wigglePsi}
\end{equation}
\end{theorem}

One may now verify directly from the expression in the previous theorem for $\stackrel{_{\leadsto}}{\Psi}$ that $\cenxi$ drifts to $\infty$, oscillates, drifts to $-\infty$, respectively as $\alpha \in(0,1)$, $\alpha = 1$ and $\alpha\in(1,2)$. %In other words, its associated pssMp $Z$ is transient for $\alpha\in (0,1)$ and recurrent for $\alpha\in[1,2)$. Moreover, we also deduce that the process $Z$ drifts to $+\infty$, oscillates or hits 0 continuously, respectively as $\alpha \in(0,1)$, $\alpha = 1$ and $\alpha\in(1,2)$.

It would be tempting here to assume that, as with the exponent $\Psi^*$, the Wiener--Hopf factorisation is clearly visible in \eqref{wigglePsi}. This turns out to be a little more subtle than one might think. \cite{KPW} showed that, when $\alpha\in(0,1]$,  the factorisation, indicated by the multiplication sign below, does indeed take the expected form 
\begin{equation}
\stackrel{_{\leadsto}}{\Psi}(z)=
  \frac{\Gamma(\alpha\rho - \iu{z})}{\Gamma(-\iu{z})}\times
    \frac{\Gamma(1 - \alpha\rho + \iu{z})}{\Gamma(1 - \alpha + \iu{z})},\qquad z\in\mathbb{R}.
    \label{factorisation<1}
\end{equation}
When $\alpha\in(1,2)$, this is not the case. The factorisation in this regime, again indicated by the multiplication sign below, takes the form 
\begin{equation}
\stackrel{_{\leadsto}}{\Psi}(z)=
(1-\alpha+\iu z)  \frac{\Gamma(\alpha\rho - \iu{z})}{\Gamma(1-\iu{z})}\times
(\iu z)    \frac{\Gamma(1 - \alpha\rho + \iu{z})}{\Gamma(2 - \alpha + \iu{z})},\qquad z\in\mathbb{R}.
\label{factorisation>1}
\end{equation}

\subsection{Radial part of an isotropic stable process}\label{radius}

Suppose now we consider  an isotropic $d$-dimensional stable process $X$ with index $\alpha\in (0,2)$.  In particular, we are interested in the process defined by its radial part, i.e.
\[
R_t=|X_t|, \qquad t\ge 0,
\]
where $|\cdot|$ denotes the Euclidian norm.  

Similarly as in the censored stable case,  we make zero into an absorbing state since the process $R$ may be recurrent and  hit zero. Define the stopping time
\begin{equation}
\tau^{\{0\}} = \inf\{ t > 0 : R_t = 0 \} 
\end{equation}
and the process
\[
Z_t=R_t\mathbf{1}_{(t<\tau^{\{0\}})}, \qquad t\ge 0,
\]
which is absorbed at zero whenever $R$ hits 0 for the first time.

It follows from isotropy of $X$  that the process $Z=(Z_t, t\ge 0)$ is  Markovian. Moreover, the process $Z$ inherits the scaling property from $X$. Once again, with the Feller property of $Z$ inherited from the same property of $X$, we have the implication that the radial part of an isotropic stable processes  killed when it hits zero   is a pssMp with index $\alpha$. 

We now consider the process $Z$ more closely for different values of $d$ and $\alpha$, and denote by $\xi=(\xi_t , t\ge 0)$  its associated L\'evy process through the Lamperti transform.
From the discussion at the end of Section \ref{intro}, we know that, for $d\ge \alpha$, the stable process $X$ cannot hit points. This implies that $\tau^{\{0\}}=\infty$ almost surely, and so, in this case, $Z$ and $\xi$ experience no killing. 
Moreover, when $\alpha<d$, the process $X$ is transient implying that $Z$ and $\xi$ drift to $\infty$. When $d=\alpha=1$, the process $X$ is recurrent which implies that the L\'evy process $\xi$ oscillates.
In the remaining case, i.e. $d=1$ and $\alpha\in(1,2)$, the process $X$ is recurrent and can hit every point, in other words, $\tau^{\{0\}}<\infty$ almost surely. Since $X$ must make infinitely many jumps across zero during any arbitrarily   small period of time immediately prior to hitting zero, the process $Z$ approaches zero continuously implying that $\xi$ drifts to $-\infty$.

 The identification of the underlying L\'evy process through the Lamperti transform   was proved in \cite{CPP11}, albeit for the case that $\alpha\leq d$. 
The result is in fact true for $\alpha\in(0,2)$, albeit there being no proof to refer to. This will appear in a forthcoming book  \cite{KPbook}.

\begin{theorem}[$\heartsuit$]\label{radialpsi}
For the pssMp constructed using the radial part of an isotropic $d$-dimensional stable process, the underlying L\'evy process, $\xi$ that appears through the Lamperti has characteristic exponent given by 
\begin{equation}
\Psi(z) = 2^\alpha\frac{\Gamma(\frac{1}{2}(-{\rm i}z +\alpha ))}{\Gamma(-\frac{1}{2}{\rm i}z)}\frac{\Gamma(\frac{1}{2}({\rm i}z +d))}{\Gamma(\frac{1}{2}({\rm i}z +d-\alpha))}, \qquad z\in\mathbb{R}.
\label{a}
\end{equation}
\end{theorem}
\begin{remark}
By setting $z = 0$ in \eqref{a}, we see easily that $\Psi(0) = 0$ and hence $\xi$ is a conservative L\'evy process, i.e. it does not experience killing.
\end{remark}

\section{Stable processes as self-similar Markov processes}

Unlike the previous section, we are interested here in self-similar Markov processes that explore larger Euclidian domains than  the half-line. More precisely, we are interested in the class  of stochastic processes that respect Definition \ref{pssMpdef}, albeit the state-space is taken as $\mathbb{R}$ or, more generally, $\mathbb{R}^d$. Like the case of pssMp, it is possible to describe general {\it self-similar Markov processes},  or {\it ssMp} for short, via a space-time transformation to another family of stochastic processes. Whereas pssMp are connected to L\'evy processes via the Lamperti space-time transformation, ssMp turn out to be connected to a family of stochastic processes whose dynamics can be described by a L\'evy-type process with Markov modulated characteristics or  {\it Markov additive process} ({\it MAP} for short). As with the previous chapter, our interest in ssMp comes about through their relationship with stable processes and their path transformations. 

We first build up the relationship between ssMp, MAPs and stable processes in the one-dimensional setting. Although we won't really apply this theory directly,  it sets the scene to  consider  the relationship in higher dimension. We conclude this section by presenting a remarkable space-time transformation for stable processes, the so-called Riesz--Bogdan--\.Zak transform, that can be easily explained using their representation as self-similar Markov processes. As we shall see later,  the Riesz--Bogdan--\.Zak transform is one of the tools that  allows us to take a new perspective on the classical fluctuation identities in relation to $\mathbb{S}_{d-1}$.

\subsection{Discrete modulation MAPs and the Lamperti--Kiu transform}\label{MAPsection0}
We are interested here in one-dimension, specifically   {\it real self-similar Markov processes} (rssMp). As alluded to above, a rssMp with {\it self-%
similarity index} $\alpha > 0$ is a Feller process, $Z = (Z_t, t\geq 0)$,  on $\mathbb{R}$ such that the origin is an absorbing state, which has the property that  its  probability laws
$P_x$, $x \in \mathbb{R}$, satisfy the
{\it scaling property} that for all $x \in \mathbb{R} \setminus \{0\}$
and $c > 0$,
\begin{equation} \text{the law of }(c Z_{t c^{-\alpha}}, t \ge 0)
  \text{ under } P_x \text{ is } P_{cx} . 
  \label{ssmpdef}
  \end{equation}
%By comparing this definition \ref{pssMpdef} with the one above, it is already clear that pssMp is a rssMp. We consider the former to be a degenerate case of the latter. It turns out that there are other `degenerate' cases in which a rssMp can  change sign at most once. We want to rule these out, and therefore enhance the definition given above by insisting that
%\begin{equation}
%P_x( \exists t > 0: Z_t Z_{t -} < 0 ) = 1
% \text{ for all }x\neq 0.
%\label{additionalassumption}
%\end{equation}
%That is to say, with probability one, the process $Z$ changes
%sign infinitely often.

%Define
%\[ \zeta= \inf\{t \ge 0: Z_t = 0 \},\]
%the time to absorption at the origin of the rssMp, which is also its lifetime. 
In the spirit of the Lamperti transform of the previous chapter, we are able to identify each rssMp with a so-called (discretely modulated) Markov additive process via a transformation of space and
time, known as the  
{\it Lamperti--Kiu representation}. We shall shortly describe this transformation in detail. However, first we must make clear what we mean by a  Markov additive process. 

\begin{definition} Let $E$ be a finite state space such that $|E| = N$. 
A Feller process, $(\xi,J) = ((\xi_t, J_t), t\geq 0)$, with  probabilities  $\mathbf{P}_{x,i}$, $x\in\mathbb{R}$, $i\in E$, and cemetery state $(-\infty,\dagger)$ is called a
{\it Markov additive process} 
if $(J_t, t \ge 0)$ is a continuous-time Markov chain on $E$ with cemetery state $\{\dagger\}$, and the pair $(\xi, J)$ is such  that 
for any $i \in E$, $s,t \ge 0$:
\begin{eqnarray}
\label{e:MAP}
& \text{ given $\{J_t = i, t<\varsigma\}$, 
the pair $(\xi_{t+s}-\xi_{t}, J_{t+s})$ is independent of
$(\xi_u, u\leq s)$,}\notag\\
&\text{ and has the same distribution as $(\xi_{s}-\xi_{0}, J_{s})$
given $\{J_{0} = i\}$,}
\end{eqnarray}
where $\varsigma = \inf\{t>0 : J_t =\dagger\}$.
\end{definition}

\noindent 
For $x\in\mathbb{R}$,  write $\mathbf{P}_{x,i} = \mathbf{P}( \cdot \,\vert\, \xi_0 = x, J_0 = i)$.
%If $\mu$ is a probability distribution on $E$, we write $\mathbf{P}_{x, \mu}    = \sum_{i \in E} \mu_i \mathbf{P}_{x,i}$. 
% \mu_i?
% It will also be convenient to denote by $\mathbf{P}(A)$ the column vector
% with $i$th element $\mathbf{P}_i(A)$,
% and by $\mathbf{P}(A; J_t)$ the matrix whose $(i,j)$th entry
% is $\mathbf{P}_i(A, J_t = j)$.
We adopt a similar convention
for expectations.

The following proposition gives a characterisation of MAPs in terms of a mixture of L\'evy processes, a Markov chain and a family of additional jump distributions.

\begin{proposition}[$\heartsuit$]\label{regimeswitch}
  The pair $(\xi,J)$ is a Markov additive process if and only if, for each $i,j\in E$, 
  there exist a sequence of iid L\'evy processes
  $(\xi^{i,n})_{n \ge 0}$ and  a sequence of iid random variables
  $(\Delta_{i,j}^n)_{n\ge 0}$, independent
  of the chain $J$, such that if $\sigma_0 = 0$
  and $(\sigma_n)_{n \ge 1}$ are the
  jump times of $J$ prior to $\varsigma$, the process $\xi$ has the representation
  \[
   \xi_t = \mathbf{1}_{(n > 0)}( \xi_{\sigma_n -} + \Delta_{J(\sigma_n-), J(\sigma_n)}^n )+ \xi^{J(\sigma_n),n}_{t-\sigma_n},
    \qquad t \in [\sigma_n, \sigma_{n+1}),\, n \ge 0.
     \]
     and $\xi_\varsigma  = -\infty$.
\end{proposition}
%\noindent Note that the above proposition can also be used to verify that the pair $(\xi,J)$ is also a strong Markov process.
 
 MAPs, sometimes called Markov modulated processes or semi-Markov processes, have traditionally found a home in queueing theory, in particular the setting of fluid queues. A good reference for the basic theory of MAPs, in continuous time, including the result above, can be found in e.g. \cite{AsmussenQueue, AA}, with more specialised results found in e.g. \cite{Cinlar,  Cinlar2, Cinlar1, Kaspi}. The literature for discrete-time MAPs is significantly more expansive. A base reference in that case is again \cite{AsmussenQueue}, but also the classical text \cite{Prabhu}. 

\bigskip 

We are now ready to describe the connection between MAP and rssMp which are absorbed at the origin.
%, i.e. with lifetime 
%\[
%\zeta=\inf\{t>0: Z_t =0\}.
%\] 
The next theorem generalises its counterpart for positive self-similar Markov processes, namely Theorem \ref{iso-2} and is due to \cite{CPR} and \cite{KKPW}.

\begin{theorem}[Lamperti--Kiu transform $\heartsuit$]\label{LKtheorem}
  Fix $\alpha>0$. The process $Z$  is a rssMp with index $\alpha$ if and only if  there exists a (killed) MAP, $(\xi, J)$, on  $\mathbb{R}\times\{-1,1\}$ and  
  \[ Z_t =  {\rm e}^{\xi_{\varphi(t)} } 
    J_{\varphi(t)}  , \qquad  0 \le t < I_\varsigma, \]
  where 
  \begin{equation}\label{e:Lamp time change}
  \varphi(t) = \inf \biggl\{ s > 0 : \int_0^s {\rm e}^{\alpha \xi_u}
  \, \d u > t  \biggr\} , \qquad  0 \le t < I_\varsigma,
\end{equation}
where $I_\varsigma = \int_0^\varsigma {\rm e}^{\alpha\xi_s}{\rm d}s$ is the lifetime of $Z$ until absorption at the origin. 
Here, we interpret $\exp\{-\infty\}\times\dagger:=0$ and $\inf\emptyset := \infty$.
\end{theorem}

Intuitively speaking, the relationship of the MAP $(\xi, J)$ to the rssMp $Z$ is that, up to a time change, $J$ dictates the sign of $Z$, whereas $\exp\{\xi\}$ dictates the radial distance of $Z$ from the origin.

\medskip

By comparing Definition \ref{pssMpdef} with the definition in \eqref{ssmpdef}, it is already clear that pssMp is a rssMp. We consider the former to be a degenerate case of the latter. It turns out that there are other `degenerate' cases in which a rssMp can  change sign at most once. 

In the forthcoming discussion, we want to rule out these and other cases in which  $J$ is killed. Said another way, we shall henceforth only consider rssMp which have the property that 
\begin{equation}
P_x( \exists t > 0: Z_t Z_{t -} < 0 ) = 1
 \text{ for all }x\neq 0.
\label{additionalassumption}
\end{equation}
This means that if we define 
\[
\zeta = \inf\{t>0 : Z_t = 0\}, 
\]
then $Z_{\zeta-} =0$ when $\zeta<\infty$.

\subsection{More on discretely modulated MAPs}\label{MAPsection}

The Lamperti--Kiu transform for rssMp can be seen to mirror the Lamperti transform for pssMp even more closely when one considers how mathematically similar MAPs are to L\'evy processes. We spend a little time here dwelling on this fact. This will also be of use shortly when we look at some explicit examples of the Lampert--Kiu transform. We recall that \eqref{additionalassumption} is henceforth in effect.

For each $i \in E$, it will be convenient to define  $\Psi_i$ as the characteristic exponent of a  L\'evy process whose law is  common
  to each of  the processes $\xi_i^n$, $n\geq 1$, that appear in the definition of Proposition \ref{regimeswitch}.  Similarly, for each $i,j \in E$, define $\Delta_{i,j}$ to
  be a random variable having the common law of the $\Delta_{i,j}^n$ variables.

Henceforth, we confine ourselves
to irreducible (and hence ergodic) Markov chains $J$.
Let the state space $E$ be the finite set $\{1, \cdots ,N\}$, for some $N \in \mathbb{N}$.
Denote the transition rate matrix of the chain $J$ by
${\boldsymbol{Q}} = (q_{i,j})_{i,j \in E}$.
For each $i \in E$, the characteristic exponent of the L\'evy process $\xi_i$
will be written $\Psi_i$. % , in the sense that ${\rm e}^{\psi_i(z)} = \LevE({\rm e}^{z \xi_i(1)})$, for all $z \in \mathbb{C}$ for which the right-hand side exists. 
For each pair of $i,j \in E$,
define
the Fourier transform $G_{i,j}(z) = \mathbf{E}({\rm e}^{\i z \Delta_{i,j}})$
 of the jump
distribution $\Delta_{i,j}$. Write ${\boldsymbol{G}}(z)$ for the $N \times N$ matrix
whose $(i,j)$-th element is $G_{i,j}(z)$. %\footnote{Here, we understand the Laplace transform  of a random variable $W$ with law $P$ (and associated expectation operator $E$) to mean $E(\exp\{z W\})$, for all $z\in\mathbf{C}$ such that the expectation is finite. Moreover, we define the associated Laplace exponent as $\log E(\exp\{z W\}).$}
We will adopt the convention that $\Delta_{i,j} = 0$ if
$q_{i,j} = 0$, $i \ne j$, and also set $\Delta_{ii} = 0$ for each $i \in E$.

Thanks to Proposition \ref{regimeswitch},  we can use the components in the previous paragraph to write down an  analogue of the characteristic exponent of a L\'evy process.  Define the matrix-valued function
\begin{equation}\label{e:MAP F}
 {\bf \Psi}(z) = {\rm diag}(- \Psi_1(z), \cdots, -\Psi_N(z))
  + {\boldsymbol{Q}} \circ {\boldsymbol{G}}(z),
\end{equation}
for all $z \in \mathbb{R}$,
where $\circ$ indicates elementwise multiplication, also called
Hadamard multiplication.
It is then known
%\cite[Lemma 2.1]{AK-mg}
that
\begin{equation}
\label{charmxexp}\mathbf{E}_{0,i}\left[ {\rm e}^{\i z \xi_t} ; J_t=j\right] = \bigl({\rm e}^{\boldsymbol{\Psi}(z) t}\bigr)_{i,j} , \qquad i,\,j \in E,  t\geq 0,z \in \mathbb{R}. 
\end{equation}
See for example Section XI.2.2c of \cite{AsmussenQueue}
For this reason, $\boldsymbol{\Psi}$ is called the {\it characteristic matrix exponent} of
the MAP $(\xi, J)$.

\bigskip

As is the case with the characteristic exponent of a L\'evy process, the characteristic matrix exponent $\boldsymbol{\Psi}(z)$  may be extended as an analytic function on  $\mathbb{C}$ to a larger domain than $\mathbb{R}$.  As a matrix, it displays a Perron--Frobenius type decomposition. We have from Section XI.2c of \cite{AsmussenQueue} the following result.

 \begin{proposition}[$\heartsuit$]\label{whatprecedes}
Suppose that $z \in \mathbb{C}$ is such that $\boldsymbol{F}(z): = \boldsymbol{\Psi}(-\i z)$ is defined.
Then, the matrix $\boldsymbol{F}(z)$ has a real simple eigenvalue $\chi(z)$,
which is larger than the real part of all  other eigenvalues.
Furthermore, the
corresponding
right-eigenvector $\boldsymbol{v}(z) = (v_1(z), \cdots, v_N(z))$ has strictly positive entries and may be
 normalised such that
\begin{equation}
  \boldsymbol{\pi}\cdot\boldsymbol{v}(z) = 1. \label{e:h norm}
\end{equation}
\end{proposition}

It will also be important for us to understand how one may establish Esscher-type changes of measure for MAPs.
The following result is also discussed in Chapter XI.2 of \cite{AsmussenQueue}, Chapter IX of \cite{AA} or \cite{JR}.

\begin{proposition}[$\heartsuit$]
\label{p:mg and com}
Let $\mathcal{G}_{t} = \sigma\{(\xi_s, J_s): s\leq t\}$, $t\geq 0$, and
\begin{equation} M_t(x,i) :=  {\rm e}^{\gamma (\xi_t-x) - \chi(\gamma)t}
    \frac{v_{J_t}(\gamma)}{v_{i}(\gamma)} ,
  \qquad t \ge 0, x\in\mathbb{R}, i\in E,
  \label{MAPCOM}\end{equation}
for some $\gamma\in\mathbb{R}$ such that $\chi(\gamma)$  is defined.
Then,
% \begin{enumerate}[(i)]
% \item
  $(M_t$, $t\geq 0)$, is a unit-mean martingale with respect to $({\mathcal G}_t, t\geq 0)$. Moreover, under the change of measure 
  \[
  \left.\frac{{\rm d}\mathbf{P}^\gamma_{x,i}}{{\rm d}\mathbf{P}_{x,i}}\right|_{\mathcal{G}_t} = M_t (x,i),\qquad t\geq 0,
  \]
  the process $(\xi,J)$ remains in the class of MAPs, and its  matrix characteristic exponent given by 
  \begin{equation}
  \boldsymbol{\Psi}_\gamma(z) = \boldsymbol{\Delta}_{\boldsymbol v}(\gamma)^{-1}\boldsymbol{\Psi}(z-\i\gamma)\boldsymbol{\Delta}_{\boldsymbol v}(\gamma) - \chi(\gamma)\mathbf{I}.
  \label{MEsscher}
  \end{equation}
  Here,  $\mathbf{I}$ is the identity matrix and $\boldsymbol{\Delta}_{\boldsymbol v}(\gamma) = {\rm diag}(\boldsymbol{v}(\gamma))$. 
  \end{proposition}

Just as is the case with the Esscher transform for L\'evy processes, a primary effect of the exponential change of measure is to alter the long-term behaviour of the process. This is stipulated by the strong law of large numbers for MAPs (see again Chapter XI.2 of \cite{AsmussenQueue}) and the behaviour of the leading eigenvalue $\chi$ as a function of $\gamma$, for which the proposition below is lifted from Proposition 3.4 of \cite{KKPW}.

\begin{proposition}[$\heartsuit$]\label{convexeval}
Suppose that $\chi$ is defined in some open interval $D$ of $\mathbb{R}$,
then, it 
is smooth and convex on $D$.
\end{proposition}

\noindent Note that, since $\boldsymbol{\Psi} (0)= \boldsymbol{Q}$, it is always the case that $\chi(0) = 0$ and $\boldsymbol{v}(0) = (1,\cdots, 1)$. Hence, for $D$ as in the previous proposition, we must necessarily have $0\in D$, in which case $\chi'(0)$ is well defined and finite. 
When this happens, the strong law of large numbers takes the form of the almost sure limit 
\begin{equation}
\lim_{t\to\infty}\frac{\xi_t}{t} = \chi'(0)
\label{SLLNMAP}
\end{equation}
and we call $\chi'(0)$  the {\it drift} of the MAP.

When, moreover, $\gamma\in D$ is a non-zero root of $\chi$, convexity dictates that $\gamma>0$ and $\chi'(\gamma)>0$ when $\chi'(0)<0$ and $\gamma<0$ and $\chi'(\gamma)<0$ when $\chi'(0)>0$. If $\chi'(0)=0$ then no such root $\gamma$ exists. A natural consequence of the change of measure in Proposition \ref{p:mg and com} is that, under $\mathbf{P}^\gamma_{i,x}$, the MAP $(\xi, J)$ aquires a new drift, which, by inspection, must be equal to $\chi'(\gamma)$.
It follows that, when $\gamma<0$, the drift of $(\xi, J)$ switches from a positive to a negative value and when $\gamma>0$, the drift switches from negative to positive.

\subsection{One-dimensional stable process and its $h$-transform.}\label{stablerssmp}

The most obvious example of a rssMp, which is not a pssMp, is a two-sided jumping stable process killed when it first hits the origin (if at all). The qualification of hitting the origin is  an issue if and only if $\alpha\in(1,2)$, as otherwise the stable process almost surely never hits the origin. Nonetheless we consider both regimes in this section.  We name the underlying process that emerges through the Lamperti--Kiu transform a {\it Lamperti-stable MAP}. For this fundamental example, we can compute the associated characteristic matrix exponent explicitly.  Recall that the state space of the underlying modulating chain in the Lampert-stable MAP is $\{-1,1\}$. Accordingly we henceforth arrange any matrix $\boldsymbol A$ pertaining to this MAP with the ordering
\[
\left(
\begin{array}{cc}
A_{1,1} & A_{1,-1}\\
A_{-1,1} & A_{-1,-1}
\end{array}
\right).
\]
\cite{CPR} and \cite{KKPW} showed the following.

\begin{lemma}[$\heartsuit$]
The characteristic matrix exponent of the Lamperti-stable MAP is given by
\begin{equation}
  \boldsymbol{\Psi}(z) =\left[
  \begin{array}{cc}
    - \dfrac{\Gamma(\alpha-\i z)\Gamma(1+\i z)}
      {\Gamma(\alpha\hat\rho- \i z)\Gamma(1-\alpha\hat\rho+ \i z)}
    & \dfrac{\Gamma(\alpha-\i z)\Gamma(1+\i z)}
      {\Gamma(\alpha\hat\rho)\Gamma(1-\alpha\hat\rho)}
    \\
    &\\
    \dfrac{\Gamma(\alpha-\i z)\Gamma(1+\i  z)}
      {\Gamma(\alpha\rho)\Gamma(1-\alpha\rho)}
    & - \dfrac{\Gamma(\alpha-\i z)\Gamma(1+\i z)}
      {\Gamma(\alpha\rho-\i z)\Gamma(1-\alpha\rho+\i z)}
  \end{array} 
  \right],
  \label{MAPHG}
\end{equation}
for $z\in\mathbb{R}$. Moreover, the relation (\ref{charmxexp}) can be analytically extended in $\mathbb{C}$ so that $\re(\i z)\in (-1,\alpha)$.
\end{lemma}

Without checking the value of $\chi'(0)$,   we are able to deduce the long term behaviour of the Lamperti-stable MAP from the transience/recurrence properties of the stable process.

We know that when $\alpha\in(1,2)$, the stable process is recurrent and $\mathbb{P}_x(\tau^{\{0\}}<\infty) =1$ for all $x\neq 0$. In that case, the Lamperti--Kiu representation dictates that 
\[
\lim_{t\to\infty}\xi_t = -\infty.
\]
When $\alpha\in(0,1)$, we also know that   $\lim_{t\to\infty}|X_t| = \infty$ almost surely, irrespective of the point of issue. Once again, the Lamperti--Kiu transform tells us that
  \[
\lim_{t\to\infty}\xi_t = \infty.
\]
FInally, when $\alpha = 1$, we have that  $\limsup_{t\to\infty}|X_t| = \infty$ and  $\liminf_{t\to\infty}|X_t| = 0$, which tells us that $\xi$ oscillates.

\medskip

There is a second example of a rssMp that we can describe to the same degree of detail as stable processes in terms of the underlying MAP. This comes about by a change of measure, which corresponds to a Doob $h$-transform to the semigroup of a two-sided jumping stable process killed on first hitting the origin if $\alpha\in(1,2)$. As it is instructive for future discussion, we give a proof of the following result which is originally from \cite{CPR}, for $\alpha\in (1,2)$, and \cite{KRS}, for $\alpha\in(1,2)$. Our proof differs slightly from its original setting.

\begin{proposition}[$\heartsuit$]\label{h-ssMp} Suppose that $(X,\mathbb{P}_x)$, $x\in\mathbb{R}$, is a one-dimensional stable process with two-sided jumps. Let $\mathcal{F}_t: = \sigma(X_s, s\leq t)$, $t\geq 0$. The following constitutes a change of measure 
\begin{equation}
\left.\frac{{\rm d}\mathbb{P}^\circ_x}{{\rm d}\mathbb{P}_x}\right|_{\mathcal{F}_t} = \frac{h(X_t)%|X_t|^{\alpha -1}
}{h(x)%|x|^{\alpha -1}
}\mathbf{1}_{(t<\tau^{\{0\}})},\qquad t\geq 0,
\label{updownCOM7}
\end{equation}
in the sense that the right-hand side is a martingale, where 
\begin{equation}
h(x) =\left(\sin(\pi\alpha\hat{\rho})\mathbf{1}_{(x\geq 0)} +
\sin(\pi\alpha{\rho})\mathbf{1}_{(x<0)}\right) |x|^{\alpha -1}
\label{constantsinh}
\end{equation}
and $\tau^{\{0\}} = \inf\{t>0: X_ t= 0\}$. Moreover, $(X, \mathbb{P}^\circ_x)$, $x\in\mathbb{R}\backslash\{0\}$ is a rssMp with matrix exponent given by 
\begin{equation}
\boldsymbol{\Psi}^\circ(z) =
 \left[
  \begin{array}{cc}
    - \dfrac{\Gamma(1-{\rm i}z)\Gamma(\alpha+{\rm i}z)}
      {\Gamma(1-\alpha{\rho}-{\rm i}z)\Gamma(\alpha{\rho}+ {\rm i}z)}
    & \dfrac{\Gamma(1-{\rm i}z)\Gamma(\alpha+{\rm i}z)}
      {\Gamma(\alpha{\rho})\Gamma(1-\alpha{\rho})}
    \\
    &\\
    \dfrac{\Gamma(1-{\rm i}z)\Gamma(\alpha+ {\rm i}z)}
      {\Gamma(\alpha\hat{\rho})\Gamma(1-\alpha\hat{\rho})}
    & - \dfrac{\Gamma(1-{\rm i}z)\Gamma(\alpha+{\rm i}z)}
      {\Gamma(1-\alpha\hat{\rho}-{\rm i}z)\Gamma(\alpha\hat{\rho}+{\rm i}z)}
  \end{array}
  \right],
  \label{Fcirc7}
\end{equation}
  for $\re({\rm i}{z}) \in (-\alpha,1)$.
\end{proposition}

\begin{proof}[Proof \emph{($\diamondsuit$)}]%In order to characterise its underlying MAP, recall that  $\boldsymbol \Psi$ is given by  (\ref{MAPHG}). 
First we verify that the right-hand side of \eqref{updownCOM7} is a martingale. 
We can compute explicitly the eigenvector $\boldsymbol{v}(\gamma)$ for the matrix $\boldsymbol{F}(\gamma): = \boldsymbol{\Psi}(-\i \gamma)$ at the  particular value of $\gamma =(\alpha -1)$. Note that $\gamma \in (-1,\alpha)$. A straightforward computation using the reflection formula for gamma functions shows that, for $\re({\rm i}z)\in (-1,\alpha)$,
\begin{align*}
{\rm det}\boldsymbol\Psi(z)  =&\frac{\Gamma(\alpha-{\rm i}z)^2\Gamma(1+{\rm i}z)^2}{\pi^2}\\
&\times\left\{
  \sin(\pi(\alpha\rho- {\rm i}z))\sin(\pi(\alpha\hat\rho- {\rm i}z)) 
  - \sin(\pi \alpha \rho) \sin(\pi\alpha\hat\rho)\right\},
\end{align*}
which has a root at ${\rm i}z = \alpha-1$. In turn, this implies that $\chi(\alpha-1)=0$. One also easily checks that 
\[
\boldsymbol{v}(\alpha -1) \propto\left[
\begin{array}{c}
\sin(\pi\alpha\hat\rho)\\
\sin(\pi\alpha\rho)
\end{array} \right].
\]
%and, by considering $\boldsymbol\Psi(0) = \boldsymbol{Q}$,
%\begin{equation}
%\boldsymbol{\pi} \propto%\frac{1}{\sin(\pi\alpha\rho) + \sin(\pi\alpha\hat\rho)} 
%\left[
%\begin{array}{c}
%\sin(\pi\alpha\rho)\\
%\sin(\pi\alpha\hat\rho) 
%\end{array}
%\right].
%\label{pi}
%\end{equation}
We claim that  with $\gamma = \alpha -1$, the change of measure (\ref{MAPCOM}) corresponds precisely to (\ref{updownCOM7}) when $(\xi, J)$ is the MAP underlying the stable process. To see this, first note that the time change $\varphi(t)$ is a stopping time and so we consider the change of measure (\ref{MAPCOM}) at this stopping time. In this respect, thanks to the Lamperti-Kiu transform, we use   $\exp\{\xi_{\varphi(t)}\} = |X_t|$, $J_{\varphi(t)} = {\rm sign}(X_t)$ and  ratio of constants, coming from \eqref{constantsinh}, as they appear in the expression for (\ref{updownCOM7}) matches the term $v_{J_{\varphi(t)}}(\gamma)/v_{J_0}(\gamma)$. Theorem III.3.4 of \cite{JS} ensures that we still have a martingale change of measure after the time-change.

Next we  address the claim that $(X, \mathbb{P}^\circ_x)$, $x\in\mathbb{R}\backslash\{0\}$, is a rssMp. This can be done much in the spirit of the computations in Sections \ref{killed*}, \ref{CSPsection} and \ref{radius}, noting that  the stopping time $\tau^{\{0\}}$ scales with the scaling of $X$ in a similar way to \eqref{stopscale}. Indeed, if, for each $c>0$, we let $X^c_t = cX_{c^{-\alpha}t}$, $t\geq 0$, and write $\tau^{\{0\}}_c = \inf\{t>0: X^c_t = 0\}$, then we have 
\[
\tau^{\{0\}}_c = c^{\alpha }\inf\{c^{-\alpha}t>0: c X_{c^{-\alpha} t} = 0\} = c^{\alpha} \tau^{\{0\}}
\]
and, for bounded measurable $f$, $ x\in\mathbb{R}\backslash\{0\}$ and $t\geq 0$,
\begin{align}
\mathbb{E}^\circ_x[f(X^c_s: s\leq t)]&= \mathbb{E}_x\left[\frac{h(X_{c^{-\alpha}t})}{h(x)}f(cX_{c^{-\alpha}s}: s\leq t)\mathbf{1}_{(c^{-\alpha}t<\tau^{\{0\}})}\right]\notag\\
&=\mathbb{E}_x\left[\frac{h(X^c_{t})}{h(cx)}f(X^c_s: s\leq t)\mathbf{1}_{(t<\tau_c^{\{0\}})}\right]\notag\\
&=\mathbb{E}_{cx}\left[\frac{h(X_{t})}{h(cx)}f(X_s: s\leq t)\mathbf{1}_{(t<\tau^{\{0\}})}\right].
\label{useagaintoprovessMp}
\end{align} 
In other words, the law of $(X^c,\mathbb{P}^\circ_x)$ agrees with $(X,\mathbb{P}^\circ_x)$ for $x\in\mathbb{R}\backslash\{0\}$.

Given the identification of the change of measure as an Esscher transform to the underlying MAP, it is now a straightforward to check from (\ref{MEsscher}) that the  MAP associated to the process $(X, \mathbb{P}^\circ_x)$, $x\in\mathbb{R}\backslash\{0\}$, agrees with   $\boldsymbol\Psi^\circ(z)$, for  $\re(\i z)\in(-\alpha,1)$, where we have again used the reflection formula for the gamma function to deal with the terms coming from $\boldsymbol\Delta_{\boldsymbol\upsilon}(\alpha-1)$.
\end{proof}

Intuitively speaking, when $\alpha\in(0,1)$, the change of measure (\ref{updownCOM7}) rewards paths that visit close neighbourhoods of the origin and penalises paths that wander large distances from the origin. Conversely, when $\alpha\in(1,2)$, the change of measure does the opposite. It penalises those paths that approach the origin and rewards those that stray away from the origin. In fact, it has been shown in \cite{KRS}  that, for $\alpha\in(0,1)$, in the appropriate sense, the change of measure is equivalent to conditioning the stable process to continuously absorb at the origin, and when  $\alpha\in(1,2)$, in \cite{CPR, KKPW} it is shown that the change of measure is equivalent to conditioning the stable process to avoid the origin.

\subsection{Self-similar Markov and stable processes in $\mathbb{R}^d$}\label{stable_D_rep}

The notion of a self-similar process (ssMp) in higher dimensions is equally well defined as in the one-dimensional setting, with \eqref{ssmpdef} as the key defining property, albeit that, now, the process is $\mathbb{R}^d$-valued. The identification of all $\mathbb{R}^d$-valued self-similar Markov processes as a space-time change of a Markov Additive Process also carries through, providing we understand the notion of a MAP in the appropriate way in higher dimensions; see \cite{Cinlar, Cinlar2, Cinlar1,  Kaspi} for some classical theoretical groundwork on this class. 
\begin{definition}
  An $\mathbb{R}\times E$ valued Feller process $(\xi,\Theta)=((\xi_t,\Theta_t): t \geq 0)$ with probabilities $\mathbf{P}_{x,\theta}$, $x\in\mathbb{R}$, $\theta\in E$, and cemetery state $(-\infty, \dagger)$ is called a \emph{Markov additive process} (MAP) if $\Theta$ is a Feller process on $E$ with cemetery state $\dagger$ such that, for every bounded measurable function $f:(\mathbb{R}\cup\{-\infty\})\times (E\cup\{\dagger\})\rightarrow \mathbb{R}$ with $f(-\infty,\dagger)=0$, $t,s\geq 0$ and $(x,\theta)\in \mathbb{R}\times E$, on $\{t<\varsigma\}$,
  \[
    \mathbf{E}_{x,\theta}[f(\xi_{t+s}-\xi_t,\Theta_{t+s})|\sigma((\xi_u, \Theta_u), u\leq t)] = \mathbf{E}_{0,\Theta_t}[f(\xi_{s},\Theta_{s})],
  \]
  where $\varsigma = \inf\{t>0 : \Theta_t = \dagger\}$.
\end{definition}
In one dimension we have worked with the case that the role of $\Theta$ is played by Markov chain $J$ on $E=\{-1,1\}$. This choice of $J$ feeds into the positive or negative positioning of a self-similar Markov process through the Lamperti--Kiu transform with $\xi$ helping to describe the radial distance from the origin.  %Essentially this alludes to a {\it polar} or {\it skew product decomposition}. Accordingly, 
In higher dimensions we will still use $\xi$ to help describe a radial distance from the origin and, by taking $E=\mathbb{S}_{d-1} : = \{x\in\mathbb{R}^d: |x| = 1\}$, the process $\Theta$ will help describe the  spatial orientation. 
In general, $\Theta$ (or $J$) is called the {\it modulator} and $\xi$ the {\it ordinator}

The following theorem is the higher dimensional analogue of Theorem \ref{LKtheorem} and is attributed to \cite{Kiu} with additional clarification from \cite{ACGZ}, building on the original work of   \cite{L72}; see also \cite{GV, VG}. As with Theorem \ref{LKtheorem}, we omit its proof. 

\begin{theorem}[Generalised Lamperti--Kiu transform $\heartsuit$]\label{genLK}   Fix $\alpha>0$. The process  $Z$ is a {ssMp} with index $\alpha$ if and only if
there exists a (killed) MAP, $(\xi, \Theta)$ on $\mathbb{R}\times\mathbb{S}_{d-1}$ such that
   \begin{equation}     
   Z_t:=  {\rm e}^{\xi_{\varphi(t)}}\Theta_{\varphi(t)} \qquad t\geq 0,
  , \qquad  t \leq I_\varsigma, 
  \label{d_polarssMprepresentation}
  \end{equation}
  where 
 \[
  \varphi(t) = \inf \biggl\{ s > 0 : \int_0^s {{\rm e}^{\alpha \xi_u}}
  \, \d u > t  \biggr\}, \qquad  t \leq I_\varsigma,
\] 
and $I_\varsigma = \int_0^\varsigma {\rm e}^{\alpha\xi_s}{\rm d}s$ is the lifetime of $Z$ until absorption at the origin.
Here, we interpret $\exp\{-\infty\}\times\dagger:=0$ and $\inf\emptyset := \infty$.
\end{theorem}

Note that, in the representation \eqref{d_polarssMprepresentation}, the time to absorption in the origin,
\[
\zeta = \inf\{t>0 : Z_t = 0\},
\] 
satisfies $\zeta = I_{\varsigma}$.

\medskip

For each $x\in\mathbb{R}^d\backslash\{0\}$, the {\it skew product decomposition} (for $d\geq 2$), is the unique representation 
\begin{equation}
x = (|x|, {\rm Arg}(x)), %\qquad x\in\mathbb{R}^d,
\label{skew}
\end{equation}
where ${\rm Arg}(x) = x/|x|$ is a vector on $\mathbb{S}_{d-1} $, the sphere of unit radius embedded in $d$-dimensional Euclidian space. Conversely any $x\in(0,\infty)\times\mathbb{S}_{d-1}$ taking the form \eqref{skew} belongs to $\mathbb{R}^d$. The representation \eqref{d_polarssMprepresentation} therefore gives  us a $d$-dimensional skew product decomposition of self-similar Markov processes.

Recall that a measure $\mu$ on $\mathbb{R}^d$ is isotropic  if  for $B\in\mathcal{B}(\mathbb{R}^{d})$, $\mu(B)=\mu(U^{-1}B)$  for every orthogonal $d$-dimensional matrix $U$. In this spirit, we can thus define an isotropic ssMp, $Z = (Z_t, t\geq 0)$  to have the property that, for every  orthogonal $d$-dimensional matrix $U$ and $x\in\mathbb{R}^d$, the law of $(U^{-1}Z, {P}_x)$ is equal to that of $(Z,{P}_{U^{-1}x})$. 

In light of the skew product decomposition in \eqref{d_polarssMprepresentation}, it is natural to ask how the property of isotropy on $Z$ interplays with  the underlying MAP $(\xi,\Theta)$. The theorem and the corollary that follows below, are a rewording of discussion found in \cite{ACGZ} with proofs that are not exactly the same as what is alluded to there but capture the same spirit. 

%if this is equivalent to the property that the underlying process $(\xi,\Theta)$ in the generalised Lamperti--Kiu decomposition also has an  {\it isotropic}  property in the sense  that: $((\xi,U^{-1}\Theta), \mathbf{P}_{x,\theta})$ is equal in law to $((\xi,\Theta), \mathbf{P}_{x,U^{-1}\theta})$,   for every matrix orthogonal $d$-dimensional matrix $U$ and $x\in\mathbb{R}^d$, $
%\theta \in\mathbb{S}_{d-1}$. 

\begin{theorem}[$\diamondsuit$]
Suppose that $Z$  is a ssMp, with underlying MAP $(\xi, \Theta)$. Then $Z$ is isotropic if and only if $((\xi,U^{-1}\Theta), \mathbf{P}_{x,\theta})$ is equal in law to $((\xi,\Theta), \mathbf{P}_{x,U^{-1}\theta})$,   for every orthogonal $d$-dimensional matrix $U$ and $x\in\mathbb{R}^d$, $
\theta \in\mathbb{S}_{d-1}$. 
\end{theorem}
\begin{proof}[Proof \emph{($\diamondsuit$)}]Suppose first that $Z$ is an isotropic ssMp.
On the event $\{t<\zeta\}$, since 
\[
\int_0^{\varphi(t)} {\rm e}^{\alpha \xi_u}\d u = t\quad\text{ and hence } \quad\frac{\d \varphi(t)}{\d t} = {\rm e}^{-\alpha \xi_{\varphi(t)}} = |Z_t|^{-\alpha},
\]
we have that 
\begin{equation}
\varphi(t) = \int_0^t|Z_u|^{-\alpha}\d u.
\label{seelater}
\end{equation}
Let us introduce its right continuous inverse, on $\{t<\varsigma\}$, as follows
\begin{equation}
{A}(t) =\inf\left\{s>0:  \int_0^s|Z_u|^{-\alpha}\d u>t\right\}.
\label{A}
\end{equation}
{Hence, we see that,} on $\{t<\varsigma\}$,
\begin{equation}
(\xi_t, \Theta_t) = (\log |Z_{{A}(t)}|,\,{\rm Arg}(Z_{{A}(t)}) ), \qquad t\geq 0,
\label{thetarotinv}
\end{equation}
where the random times ${A}(t)$, {for $t\leq \varsigma$}, are stopping times in the natural filtration of $Z$.

Now suppose that $U$ is any orthogonal $d$-dimensional matrix and let $Z' = U^{-1}Z$.
Since $Z$ is isotropic and since  $|Z'| =|Z| $, and ${\rm Arg}(Z') = U^{-1}{\rm Arg}(Z) $, we see from \eqref{thetarotinv} that,  for $x\in\mathbb{R}$ and  $\theta\in\mathbb{S}_{d-1}$
\begin{align}
((\xi,U^{-1} \Theta), \mathbf{P}_{\log |x|,\theta})  &= ((\log |Z_{{A}(t)}|,\,U^{-1}{\rm Arg}(Z_{{A}(t)}) ), {P}_x)\notag\\
&\stackrel{d}{=} ((\log |Z_{{A}(t)}|,\,{\rm Arg}(Z_{{A}(t)}) ), {P}_{U^{-1}x})\notag\\
&=((\xi, \Theta), \mathbf{P}_{\log |x|,U^{-1}\theta})
\label{bothways}
\end{align}
and the ``only if'' direction is proved.

For the converse statement, suppose that the left-hand side and right-hand side in \eqref{bothways} are equal for all  orthogonal $d$-dimensional matrices $U$, $x\in\mathbb{R}$ and $\theta\in\mathbb{S}_{d-1}$. Again, setting  $Z' = U^{-1}Z$ and letting $A'$ play the role of \eqref{A} but for $Z'$, we have 
\begin{align}
((\log |Z'_{{A}'(t)}|,\,{\rm Arg}(Z'_{{A}'(t)}) ), {P}_x)&\stackrel{d}{=} ((\log |Z_{{A}(t)}|,\,U^{-1}{\rm Arg}(Z_{{A}(t)}) ), {P}_x) \notag\\
&\stackrel{d}{=} ((\xi,U^{-1} \Theta), \mathbf{P}_{\log |x|,\theta}) \notag\\
&\stackrel{d}{=} ((\xi, \Theta), \mathbf{P}_{\log |x|,U^{-1}\theta})\notag\\
&\stackrel{d}{=} ((\log |Z_{{A}(t)}|,\,{\rm Arg}(Z_{{A}(t)}) ), {P}_{U^{-1}x}).
\end{align}
This concludes the ``if'' part of the proof. 
\end{proof}

\begin{corollary}[$\heartsuit$]\label{|Z|} If $Z$ is an isotropic ssMp, then $|Z|$ is equal in law to a pssMp and hence $\xi$ is a L\'evy process.
%since $\varphi(t)\wedge\varsigma$ is a stopping time for its underlying MAP, $(\xi, \Theta)$, 
%scaling and the strong Markov property tells us thatFor an isotropic  ssMp, $Z$, with probabilities ${P}_x$, $x\in\mathbb{R}^d$,  
% $|Z|:=(|Z_t|, t\geq 0)$ is a pssMp.
\end{corollary}
\begin{proof}[Proof \emph{($\diamondsuit$)}] The scaling property of $|Z|$ follows directly from that of $Z$. Moreover we have, for bounded measurable $g:[0,\infty)\to\mathbb{R}$ and $s,t\geq0$, on $\{t<\zeta\}$,
  \begin{align*}
    {E}_{x}[g(|Z_{t+s}|)\,|\,\sigma(Z_u, u\leq t)] &= \mathbf{E}_{\log |Z_t|, {\rm Arg}(Z_t)}[g({\rm e}^{\xi_{\varphi(s)}})]\\
    &\stackrel{d}{=}\mathbf{E}_{\log |Z_t|, \mathbf{1}}[g({\rm e}^{\xi_{\varphi(s)}})]\\
    &={E}_{|Z_t|\mathbf{1}}[g(|Z_s|)], 
  \end{align*}
where ${\texttt 1} = (1,0,\cdots,0)\in\mathbb{R}^d$ is the ``North Pole'' on $\mathbb{S}_{d-1}$. This ensures the Markov property.

To verify the Feller property, we note that, for $x\in\mathbb{R}^d$,  $(Z,{P}_x)$ is equal in law to 
\[
Z^{(x)}: = |x|{\rm e}^{\xi_{  \varphi(|x|^{-\alpha} t)}   }\Theta_{  \varphi(|x|^{-\alpha} t)}   , \qquad t\leq |x|^{\alpha}\int_0^\varsigma {\rm e}^{\alpha\xi_u}\d u,\qquad 
\]
under $\mathbf{P}_{0, {\rm Arg}(x)}$. Hence for all continuous  $g:[0,\infty)\to\mathbb{R}$ vanishing at $\infty$, 
\[
{E}_x[g(|Z_t|)] = \mathbf{E}_{0, {\rm Arg}(x)}\left[g(|x|{\rm e}^{\xi_{  \varphi(|x|^{-\alpha} t)}   })\right].
\]
The conditions of the Feller property can now be easily verified using dominated convergence.
\end{proof}

\medskip

The most prominent example of a $d$-dimensional ssMp that will be of use to us is of course the isotropic stable process in $\mathbb{R}^d$ itself. 
%Recall that isotropic $d$-dimensional stable processes have characteristic exponent given by
%\[
%\Psi(z) = |z|^\alpha, \qquad z\in\mathbb{R}^d
%\]
%and, accordingly, the underlying L\'evy measure satisfies 
%\[
%\Pi(B) = c(\alpha)%2^{\alpha}\pi^{-d/2}\frac{\Gamma((d+\alpha)/2)}{\big|\Gamma(-\alpha/2)\big|}
%\int_{\mathbb{S}_{d-1}}\sigma_1(\d\phi)\int_{(0,\infty)}\mathbf{1}_{B}(r\phi)\frac{{\rm d} r}{r^{\alpha+1}} = c(\alpha)\int_B \frac{1}{|z|^{\alpha+d}}\d z, 
%\]
%for $B\in\mathcal{B}(\mathbb{R}^d)$, where
%\[
%c(\alpha) =2^{\alpha}\pi^{-d/2}\frac{\Gamma((d+\alpha)/2)}{\big|\Gamma(-\alpha/2)\big|}.
%\]
The description of the underlying MAP is somewhat less straightforward to characterise. We know however that the stable process is a pure jump process. In the spirit of a calculation found in Lemma 2 of \cite{BW}, the  theorem below uses the compensation formula for the jumps of the stable process as a way of capturing the jump dynamics of the MAP. This author has also seen similar computations in working documents of Bo Li from Nankai University, PR China and Victor Rivero from CIMAT, Mexico.

We will use the usual notation in the stable setting. That is, $(\xi,\Theta)$ with probabilities $\mathbf{P}_{x,\theta}$, $x\in\mathbb{R}$, $\theta\in\mathbb{S}_{d-1}$, is the MAP underlying the stable process. We will work with the increments $\Delta\xi_t = \xi_t- \xi_{t-}\in\mathbb{R}$, $t\geq 0$. %, but also   $\Delta\Theta_t  = \Theta_t\circ\Theta_{t-}^{\star}\in\mathbb{S}_{d-1}$, $t\geq 0$, where $\circ$ represents the group operation of addition on the unit sphere and, for $\theta\in\mathbb{S}_{d-1}$, we understand $\theta^\star$ to be the unique rotation such that $\theta^{\star}\circ\theta = \theta\circ\theta^\star ={\texttt 1}$. %, the ``North Pole''.

\begin{theorem}[$\diamondsuit$]\label{MAPjumps}  Suppose that $f$ is a positive, bounded measurable function on $[0,\infty)\times\mathbb{R}\times\mathbb{R}\times\mathbb{S}_{d-1}\times\mathbb{S}_{d-1}$ such that $f(\cdot,\cdot,0,\cdot, \cdot) = 0$, then, for all $\theta\in\mathbb{S}_{d-1}$,
\begin{align}
&\mathbf{E}_{0,\theta}\left[\sum_{s>0}f(s,\xi_{s-},\Delta\xi_{s},\Theta_{s-}, \Theta_s)\right]\notag\\
&=\int_0^\infty\int_{\mathbb{R}} \int_{\mathbb{S}_{d-1}}\int_{\mathbb{S}_{d-1}}\int_{\mathbb{R}}V_\theta( \d s, \d x, \d \vartheta)  \sigma_1(\d\phi)  \d y\frac{c(\alpha){\rm e}^{yd}}{|{\rm e}^{y}\phi-   \vartheta|^{\alpha+d}}f(s, x,y,\vartheta, \phi),
\label{MAPdeltatostable}
\end{align}
where 
\[
V_\theta(\d s, 
d x, \d \vartheta) =\mathbf{P}_{0,\theta}(\xi_s \in \d x, \Theta_s\in \d\vartheta)\d s , \qquad x\in\mathbb{R}, \vartheta\in\mathbb{S}_{d-1}, s\geq 0,
\]
is the space-time potential of $(\xi, \Theta)$,
$\sigma_1(\phi)$ is the surface measure on $\mathbb{S}_{d-1}$ normalised to have unit mass and $c(\alpha) =2^{\alpha-1}\pi^{-d}{\Gamma((d+\alpha)/2)\Gamma(d/2)}/{\big|\Gamma(-\alpha/2)\big|}$.
\end{theorem}

\begin{proof}[Proof \emph{($\diamondsuit$)}]
According to the generalised Lamperti--Kiu transformation \eqref{d_polarssMprepresentation}, we have 
\[
\xi_{t}=\log(|X_{{A}(t)}|/|X_{0}|),\qquad \Theta_{t}=\frac{X_{{A}(t)}}{|X_{{A}(t)}|},\qquad t\geq 0,
\] 
where ${A}(t) = \inf\{s>0: \int^{s}_{0}|X_{u}|^{-\alpha}\d u> t\}$; see also \eqref{A}. Let $f$ be given as in the statement of the theorem. Writing the left-hand side of \eqref{MAPdeltatostable} in terms of the stable process, we have  for all $\theta\in\mathbb{S}_{d-1}$,

\begin{align*}
&\mathbf{E}_{0,\theta}\left[\sum_{s>0}f(s,\xi_{s-},\Delta\xi_{s},\Theta_{s-}, \Theta_s)\right] \\
&= \mathbb{E}_{\theta}\left[\sum_{s>0}f\left(\int^{s}_{0}|X_{u}|^{-\alpha}\d u, \log|X_{s-}|, \log\left(|X_{s}|/|X_{s-}|\right), {\rm Arg}(X_{s-})
%\frac{X_{s-}}{|X_{s-}|}
, {\rm Arg}(X_s)
% \frac{X_{s}}{|X_{s}|} 
\right)\right].
\end{align*}

Next note that, for $t\geq 0$, 
\[
\frac{|X_{s}|}{|X_{s-}|}=\left|\frac{X_{s-}}{|X_{s-}|}+\frac{\Delta X_{s}}{|X_{s-}|}\right| = \left|{\rm Arg}(X_{s-})+\frac{\Delta X_{s}}{|X_{s-}|}\right|
\]
 and
\[
{\rm Arg}(X_s) = \frac{X_{s}}{|X_{s}|}=\dfrac{\dfrac{X_{s-}}{|X_{s-}|}+\dfrac{\Delta X_{s}}{|X_{s-}|}}{\left|\dfrac{X_{s-}}{|X_{s-}|}+\dfrac{\Delta X_{s}}{|X_{s-}|}\right|} = \dfrac{{\rm Arg}(X_{s-})+\dfrac{\Delta X_{s}}{|X_{s-}|}}{\left|{\rm Arg}(X_{s-})+\dfrac{\Delta X_{s}}{|X_{s-}|}\right|}.
\]

The compensation formula for the Poisson random measure of jumps of $X$ now tells us that 
\begin{align*}
&\mathbf{E}_{0,\theta}\left[\sum_{s>0}f(s,\xi_{s-},\Delta\xi_{s},\Theta_{s-}, \Theta_s)\right] \notag\\
&=\mathbb{E}_{\theta}\Bigg[
\int^{\infty}_{0}\hspace{-3pt}\d s\int_{ \mathbb{S}_{d-1}}\hspace{-3pt}{\sigma_1(\d\phi)}\int_{0}^\infty\hspace{-3pt}
\d r\frac{c(\alpha)}{r^{1+\alpha}}\notag\\
&\hspace{1cm}f\Bigg(\varphi(s),\log|X_{s-}|, \log\left|{\rm Arg}(X_{s-})+\frac{r\phi}{|X_{s-}|}\right|, {\rm Arg}(X_{s-}),  \frac{{\rm Arg}(X_{s-})+\frac{r\phi}{|X_{s-}|}}{\left|{\rm Arg}(X_{s-})+\frac{r\phi}{|X_{s-}|}\right|}\Bigg)
\Bigg]\notag\\
&=
\mathbf{E}_{0,\theta}\left[\int^{\infty}_{0}\d v\int_{ \mathbb{S}_{d-1}}{\sigma_1(\d\phi)}\int_{0}^\infty %c(\alpha)
\d u\frac{c(\alpha)}{u^{1+\alpha}}f\left(v, \xi_v, \log\left|\Theta_{v}+u\phi\right|, \Theta_v, \frac{\Theta_{v}+u\phi}{\left|\Theta_{v}+u\phi\right|}\right)\right]\notag\\
&=\mathbf{E}_{0,\theta}\left[\int^{\infty}_{0}\d v\int_{ \mathbb{R}^{d}}\d z \frac{\tilde{c}(\alpha)}{|z|^{\alpha+d}}f\left(v,\xi_v, \log\left|\Theta_{v}+z\right|, \Theta_v, \frac{\Theta_{v}+z}{\left|\Theta_{v}+z\right|}\right)\right]\notag\\
&=\mathbf{E}_{0,\theta}\left[\int^{\infty}_{0}\d v\int_{\mathbb{R}^{d}}\d w\frac{\tilde{c}(\alpha)}{|w -  \Theta_{v}|^{\alpha+d}}f\left(v, \xi_v, \log|w|,\Theta_v,  \frac{w}{\left|w\right|}\right)\right]
\end{align*}
where in the second equality, we first make the change of variables $u= r/|X_{s-}|$ and then $v = \int^{s}_{0}|X_{u}|^{-\alpha}\d u$ and in the third equality we convert to Cartesian coordinates with $\tilde{c}(\alpha) = 2^{\alpha-1}\pi^{-d}{\Gamma((d+\alpha)/2)\Gamma(d/2)}/{\big|\Gamma(-\alpha/2)\big|}$.  Converting the right-hand side above back to skew product variables, we thus get
\begin{align}
&\mathbf{E}_{0,\theta}\left[\sum_{s>0}f(s,\xi_{s-},\Delta\xi_{s},\Theta_{s-}, \Theta_s)\right] \notag\\
&=\mathbf{E}_{0,\theta}\left[\int^{\infty}_{0}\d v\int_{\mathbb{S}_{d-1}}\sigma_1(\d\phi)\int_{0}^\infty\d r \frac{c(\alpha)r^{d-1}}{|r\phi-     \Theta_{v}|^{\alpha+d}}f\left(v,  \xi_v,\log r, \Theta_v, \phi\right)\right]\notag\\
&=\mathbf{E}_{0,\theta}\left[\int^{\infty}_{0}\d v\int_{\mathbb{S}_{d-1}}\sigma_1(\d\phi)\int_{\mathbb{R}}\d y \frac{c(\alpha){\rm e}^{yd}}{|{\rm e}^{y}\phi-     \Theta_{v}|^{\alpha+d}}f\left(v, \xi_v, y, \Theta_v, \phi\right)\right]\notag\\
&=\int_0^\infty\int_{\mathbb{R}}\int_{\mathbb{S}_{d-1}}\int_{\mathbb{S}_{d-1}}\int_{\mathbb{R}} V_\theta(\d v,\d x,\d \vartheta ) \sigma_1(\d\phi)\d y 
\frac{c(\alpha){\rm e}^{yd}}{|{\rm e}^{y}\phi-   \vartheta|^{\alpha+d}}f(v, x, y,\vartheta, \phi),
\label{onemorechangeofvariable}
\end{align}
as required.
%
%Next note that, for $\phi\in\mathbb{S}_{d-1}$,
%\[
%|{\rm e}^{y}\phi-\Theta_{v}|^{2}=({\rm e}^{y}-1)^{2}-2{\rm e}^{y}(\phi-\Theta_{v})\cdot\Theta_{v},
%\]
%and hence, continuing the computation in \eqref{onemorechangeofvariable}, making a change of variables $y=\log r$, we have 
%
%\begin{align}
%&\mathbf{E}_{0,\phi}\left[\sum_{s>0}f(s,\Delta\xi_{s},\Theta_{s-}, \Theta_s)\right] \notag\\
%&=\mathbf{E}_{0,\phi}\left[\int^{\infty}_{0}\d v\int_{\mathbb{S}_{d-1}}\sigma_1(\d\phi)\int_{\mathbb{R}}\frac{c(\alpha){\rm e}^{yd}\d y}{|({\rm e}^{y}-1)^{2}-2{\rm e}^{y}(\phi-   \Theta_{v})\cdot\Theta_{v}|^{(\alpha+d)/2}}f(v, y,\Theta_v, \phi)\right]\notag\\
%&=\int_{\mathbb{S}_{d-1}}\int_{\mathbb{S}_{d-1}}\int_{\mathbb{R}}\sigma_1(\d\phi)V_\theta(\d \vartheta)\d y\frac{c(\alpha){\rm e}^{yd}}{|({\rm e}^{y}-1)^{2}-2{\rm e}^{y}(\phi-   \vartheta)\cdot\vartheta|^{(\alpha+d)/2}}f(v, y,\vartheta, \phi)
%%
%%&=\mathbf{E}_{0,\phi}\left[\int^{\infty}_{0}\d v\int_{\mathbb{S}_{d-1}}\sigma_1(\d\vartheta)\int_{\mathbb{R}}\frac{c(\alpha){\rm e}^{yd}\d y}{|({\rm e}^{y}-1)^{2}-2{\rm e}^{y}(\vartheta\circ\Theta_v-     \Theta_{v})\cdot\Theta_{v}|^{(\alpha+d)/2}}f(v, y, \Theta_{v},\vartheta)\right]\notag\\
%%&=\mathbf{E}_{0,\phi}\left[\int^{\infty}_{0}\d v\int_{\mathbb{S}_{d-1}}\sigma_1(\d\vartheta)\int_{\mathbb{R}}\frac{c(\alpha){\rm e}^{yd}\d y}{|{\rm e}^{2y}+1-2{\rm e}^{y}\vartheta\cdot\Theta_v|^{(\alpha+d)/2}}f(v, y, \vartheta)\right],
%\label{comparewithdualcompensation}
%\end{align}
%as required.
\end{proof}

The radial component of an isotropic $d$-dimensional stable process, which can be singled out by Corollary \ref{|Z|}, has already been studied in Theorem \ref{radialpsi}.
%From the latter, we find that the underlying MAP, $\xi$ is a L\'evy process with characteristic exponent
%\begin{equation}
%\Psi(z) = 2^\alpha\frac{\Gamma(\frac{1}{2}(-{\rm i}z +\alpha ))}{\Gamma(-\frac{1}{2}{\rm i}z)}\frac{\Gamma(\frac{1}{2}({\rm i}z +d))}{\Gamma(\frac{1}{2}({\rm i}z +d-\alpha))}, \qquad z\in\mathbb{R}.
%\label{circexponentagain}
%\end{equation}
%
\medskip

The second example of $d$-dimension ssMp takes inspiration from Proposition \ref{h-ssMp}. In the spirit of \eqref{updownCOM7} we define for an isotropic $d$-dimensional stable process, $(X, \mathbb{P}_x)$, $x\in\mathbb{R}^d\backslash\{0\}$, 
\begin{equation}
\left.\frac{{\rm d}\mathbb{P}^\circ_x}{{\rm d}\mathbb{P}_x}\right|_{\mathcal{F}_t} = \frac{|X_t|^{\alpha -d}
}{|x|^{\alpha -d}
}%\mathbf{1}_{(t<\tau^{\{0\}})}
, \qquad t\geq 0,
\label{d-circ}
\end{equation}
where $\mathcal{F}_t = \sigma(X_s, s\leq t)$.

\begin{proposition}[$\diamondsuit$] \label{h-ssMp-d}For $d\geq 2$, \eqref{d-circ} constitutes a change of measure, in the sense that the right-hand side is a martingale, and the resulting process $(X, \mathbb{P}^\circ_x)$, $x\in\mathbb{R}^d\backslash\{0\}$ is a ssMp. Moreover, $(|X|, \mathbb{P}^\circ_x)$, $x\in\mathbb{R}^d\backslash\{0\}$ is a pssMp with underlying L\'evy process that has characteristic exponent
\begin{equation}
\Psi^\circ(z) %= \Psi(z-{\rm i}(\alpha -d)) 
= 2^\alpha\frac{\Gamma(\frac{1}{2}(-{\rm i}z +d ))}{\Gamma(-\frac{1}{2}({\rm i}z+\alpha  -d))}\frac{\Gamma(\frac{1}{2}({\rm i}z +\alpha))}{\Gamma(\frac{1}{2}{\rm i}z )}, \qquad z\in\mathbb{R}.
\end{equation}
\end{proposition}
\begin{proof}[Proof \emph{($\diamondsuit$)}]
Recalling that $\Psi$ in \eqref{a}  is the characteristic exponent of the L\'evy process $\xi$ which underlies the radial component of a stable process in $d$-dimensions, we easily verify that $\Psi(-{\rm i}(\alpha -d)) = 0$. It follows that   
$(
\exp\{(\alpha -d)\xi_{t}\}
, t\geq 0)$ is a martingale. Moreover, under the change of measure induced by this martingale,  $\xi$ remains in the class of L\'evy processes, but now with characteristic exponent 
$\Psi^\circ(z) = \Psi(z-{\rm i}(\alpha -d))$,  $z\in\mathbb{R}$.
Noting that $\varphi(t)$ is a stopping time in the filtration of $(\xi, \Theta)$, recalling again Theorem III.3.4 of \cite{JS}, we see that \eqref{d-circ} also represents the aforesaid change of measure.

Following similar reasoning to the proof of Proposition \ref{h-ssMp}, in particular the calculations centred around \eqref{useagaintoprovessMp}, as well as incorporating the conclusion of Corollary \ref{|Z|},  it is not difficult to verify that both 
$(X, \mathbb{P}^\circ_x)$, $x\in\mathbb{R}^d\backslash\{0\}$, is a ssMp and $(|X|, \mathbb{P}^\circ_x)$, $x\in\mathbb{R}^d\backslash\{0\}$, is a pssMp. 
It follows that $\Psi^\circ$ is the characteristic exponent of the L\'evy process that underlies the Lamperti transform of $(|X|, \mathbb{P}^\circ_x)$, $x\in\mathbb{R}^d\backslash\{0\}$.
\end{proof}

The reader again notes that, for $d\geq 2$, the change of measure \eqref{d-circ} rewards paths that remain close to the origin and penalises those that stray far from the origin. Just as  in \cite{KRS} it can be shown that  $\mathbb{P}^\circ_x$, $x\in\mathbb{R}^d\backslash\{0\}$, again corresponds to the law of conditioning the stable process to continuously absorb a the origin. The origins of Proposition \ref{h-ssMp-d} can already be found in \cite{BZ}.

\subsection{Riesz--Bogdan--\.{Z}ak transform}\label{RBZ1d}

The changes of measure, (\ref{updownCOM7}) in one dimension and \eqref{d-circ} in higher dimension, also play an important role in a remarkable space-time path transformation, the Riesz--Bogdan--\.{Z}ak transform. This transformation was first introduced rigorously in \cite{BZ}, although the computational visibility of this path transformation was already implicitly on display in the work of \cite{Riesz}; see the remarks in \cite{BGR}.  Later on in this text, we will use it to analyse a number of  path functionals of stable processes in dimension $d\geq 2$. Despite the fact that we only use Riesz--Bogdan--\.{Z}ak transform in higher dimension, we also state and prove it in dimension $d=1$ for instructional purposes. The following theorem and proof  are lifted directly from \cite{Deep1}.

\begin{theorem}[The one-dimensional  Riesz--Bogdan--\.{Z}ak transform $\heartsuit$]
\label{th:BZ} Suppose that $X$ is a one-dimensional  stable process with two-sided jumps.
 Define
\[
\eta(t) = \inf\{s>0 : \int_0^s |X_u|^{-2\alpha}{\rm d}u >t\}, \qquad t\geq 0.
\]
Then, for all $x\in\mathbb{R}\backslash\{0\}$, $(-1/{X}_{\eta(t)}, \eta(t)< \tau^{\{0\}})$ under $\mathbb{P}_{x}$
is  equal in law to $(X, \mathbb{P}_{-1/x}^\circ)$, where $\mathbb{P}_x$, $x\in\mathbb{R}\backslash\{0\}$, was defined in \eqref{updownCOM7}.
\end{theorem}

%It is straightforward  to deduce that $(X, \mathbb{P}^\circ_x)$, $x\in\mathbb{R}\backslash\{0\}$, is a rssMp, inheriting the index of self-similarltiy $\alpha$ from $(X, \mathbb{P}_x)$, $x\in \mathbb{R}\backslash\{0\}$. 

\begin{proof}[Proof \emph{($\heartsuit$)}]%[Proof of Theorem \ref{th:BZ}]
First note that, if $X$ is an $(\alpha,\rho)$ stable process, then  $-X$ is a $(\alpha,\hat\rho)$ stable process.  Next, we show that $(-1/{X}_{\eta(t)}, \eta(t)< \tau^{\{0\}})$ is a rssMp with index $\alpha$ by analysing its Lamperti--Kiu decomposition. 

To this end, note that, if $(\hat{\xi}, \hat{J})$ is the MAP that underlies  $\hat{X}: =-X$, then its matrix exponent, say $\hat{\boldsymbol\Psi}(z)$, is equal to (\ref{MAPHG}) with the roles of $\rho$ and $\hat\rho$ interchanged. As $\hat{X}$ is a rssMp, we have 
\[
\hat{X}_t = {\rm e}^{
\hat{\xi}_{\varphi(t)}
}
\hat{J}_{\varphi(t)},\qquad t<\tau^{\{0\}},
\]
where 
\[
\int_0^{\varphi(t)}{\rm e}^{\alpha\hat{\xi}_{s}}{\rm d}s = t. % \qquad t<\tau^{\{0\}}.
\]
Noting that 
\[
\int_0^{\eta(t)} {\rm e}^{-2\alpha \hat{\xi}_{\varphi(u)}}
{\rm d}u = t, \qquad \eta(t)<\tau^{\{0\}},
\]
a straightforward differentiation of the last two integrals shows  that, respectively, 
\[
\frac{{\rm d}\varphi(t)}{{\rm d}t} = {\rm e}^{-\alpha \hat{\xi}_{\varphi(t)}}\text{ and }
\frac{{\rm d}\eta(t)}{{\rm d}t} ={\rm e}^{2\alpha \hat{\xi}_{\varphi\circ\eta(t)}}, \qquad\eta( t)<\tau^{\{0\}}.
\]
The chain rule now tells us that 
\begin{equation}
\frac{{\rm d}(\varphi\circ\eta)(t)}{{\rm d}t} = \left.\frac{{\rm d}\varphi(s)}{{\rm d}s}\right|_{s = \eta(t)}\frac{{\rm d}\eta(t)}{{\rm d}t}  = {\rm e}^{\alpha \hat{\xi}_{\varphi\circ\eta(t)}},
\label{repeatinddim}
\end{equation}
and hence,
\[
\int_0^{\varphi\circ\eta(t)}{\rm e}^{-\alpha \hat{\xi}_u} {\rm d}u = t, \qquad \eta(t)<\tau^{\{0\}}.
\]
The qualification that $\eta(t)<\tau^{\{0\}}$ only matters when $\alpha\in(1,2)$. In that case, the fact that $\mathbb{P}_x(\tau^{\{0\}}<\infty) = 1$ for all $x\in\mathbb{R}$ implies that  $\lim_{t\to\infty}\hat{\xi}_t = -\infty$ almost surely. As a consequence, it follows that $\int_0^\infty{\rm e}^{-\alpha \hat{\xi}_u} {\rm d}u = \infty $ and hence $\lim_{t\to\infty}\varphi\circ\eta(t) = \infty$. That is to say, we have $\lim_{t\to\infty}\eta(t)= \tau^{\{0\}}$.
Noting that $1/\hat{J}_s = \hat{J}_s$, $s\geq 0$, it now follows that 
\[
\frac{1}{\hat{X}_{\eta (t)}} %= \exp\{ -\hat{\xi}(\varphi\circ\eta(t)) -\iu\pi (J(\varphi\circ\eta (t))+1)\} 
=  {\rm e}^{-\hat{\xi}_{\varphi\circ\eta(t)} }\hat{J}_{\varphi\circ\eta (t)} , \qquad t<\tau^{\{0\}},
\]
is the representation of a rssMp whose underlying MAP has matrix exponent given by $\hat{\boldsymbol\Psi}(-z)$, whenever it is well defined. Recalling the definition of $\hat{\boldsymbol\Psi}(z)$, we see that the MAP that underlies $(-1/X_{\eta(t)})_{t\geq 0}$ via the Lamperti--Kiu transform is identically equal in law to the MAP with matrix exponent
given by $\boldsymbol\Psi^\circ$ given in \eqref{Fcirc7}. 
 \end{proof}

%\begin{remark}\rm
%Recall that for a different linear normalisation of time in the definition of the stable process, that is to say when $(X_t)_{t\geq 0}$, is replaced by $X': = (X_{ct})_{t\geq0}$, for some $c>0$, we have that   the matrix $\boldsymbol {\Psi}$ in (\ref{MAPHG}) is replaced by $c\boldsymbol\Psi$. If one works with this different scaling of time   in the context of  Theorem \ref{th:BZ}, then it can be easily verified that the matrix in \eqref{Fcirc7}  is also changed by multiplying it by the constant $c$. In turn, this matrix exponent correctly identifies with the 
%process $(X', \mathbb{P}_{x}^\circ)$, $x\in\mathbb{R}\backslash\{0\}$. 
%\end{remark} 

Finally, we give the $d$-dimensional version of the Riesz--Bogdan--\.{Z}ak transformation is also available for higher dimensional, albeit isotropic, stable processes. Our proof differs from that of \cite{BZ}, appealing to L\'evy systems rather than potentials. 
Define the  transformation $K:\mathbb{R}^d\backslash\{0\}\to \mathbb{R}^d\backslash\{0\}$, by 
\[
Kx = \frac{x}{|x|^2},\qquad x\in\mathbb{R}^d\backslash\{0\}.
\]
This transformation inverts space through the unit sphere $\{x\in\mathbb{R}^d: |x|=1\}$ and accordingly, it is not surprising that $K(Kx) = x$.
To see how the $K$-transform maps $\mathbb{R}^d\backslash\{0\}$ into itself, write $x\in\mathbb{R}^d\backslash\{0\}$ in skew product form $x = (|x|, {\rm Arg}(x))$, and note that 
\[
K x = (|x|^{-1}, {\rm Arg}(x)), \qquad x\in\mathbb{R}^d\backslash\{0\},
\]
showing that the $K$-transform `radially inverts' elements of $\mathbb{R}^d\backslash\{0\}$ through  $\mathbb{S}_{d-1}$. 

%The following theorem says that the $K$-transform of an appropriate time-changed $d$-dimensional isotropic stable process has a straightforward relation back to the original stable process.

\begin{theorem}[$d$-dimensional Riesz--Bogdan--\.{Z}ak Transform, $d\geq 2$ $\heartsuit$]\label{BZ} Suppose that $X$ is a $d$-dimensional isotropic stable process with $d\geq 2$. 
Define
\begin{equation}
\eta(t) = \inf\{s>0 : \int_0^s |X_u|^{-2\alpha}{\rm d}u >t\}, \qquad t\geq 0.
\label{etatimechange}
\end{equation}
Then, for all $x\in\mathbb{R}^d\backslash\{0\}$, $(KX_{\eta(t)}, t\geq 0)$ under $\mathbb{P}_{x}$ is equal in law to $(X, \mathbb{P}_{Kx}^{\circ})$.
\end{theorem}

\begin{proof}[Proof \emph{($\clubsuit$)}]
As with the proof of Theorem \ref{RBZ1d}, it is straightforward to check that $(KX_{\eta (t)}, t\geq 0)$ is a ssMp. Indeed, in skew product form, 
\[
KX_{\eta(t)} = {\rm e}^{-\xi_{\varphi\circ\eta(t)}}\Theta_{ \varphi\circ\eta(t) }, \qquad t\geq0,
\]
and, just as in the computation \eqref{repeatinddim}, one easily verifies again that
\[
\varphi\circ\eta(t) = \inf\left\{s>0: \int_0^s {\rm e}^{-\alpha \xi_u}\d u >t\right\}.
\]
It is thus clear that $(KX_{\eta (t)}, t\geq 0)$ is a ssMp with underlying MAP equal to $(-\xi, \Theta)$. To complete the proof, it therefore suffices to check that $(-\xi,\Theta)$ is also the MAP which underlies the ssMp $(X, \mathbb{P}^\circ_x)$, $x\in\mathbb{R}^d\backslash\{0\}$. %As such, we need a way to identify the aforesaid MAP. Noting that $(X, \mathbb{P})$ has no Gaussian component and that \eqref{d-circ} is a Doob $h$-transform, it suffices to compare the jump intensity of $(KX_{\eta(t)}, t\geq 0)$,   $x\in\mathbb{R}^d\backslash\{0\}$, under $\mathbb{P}_{x}$ to that of $(X, \mathbb{P}_{x}^{\circ})$,  $x\in\mathbb{R}^d\backslash\{0\}$.

To this end, 
we note that $(X, \mathbb{P}^\circ_x)$, $x\in\mathbb{R}^d\backslash\{0\}$, is a pure jump process and hence entirely characterised by its jump rate. To understand why at a heuristic level, note that, as a Feller process, it is 
is in possession of an infinitesimal generator, say $\mathcal{L}^\circ$. Indeed,  standard theory tells us that 
\begin{equation}
\mathcal{L}^\circ f (x) = \lim_{t\downarrow0}\frac{\mathbb{E}^\circ_x[f(X_t)] -f(x)}{t} = 
\lim_{t\downarrow0}\frac{\mathbb{E}_x[|X_t|^{\alpha - d}f(X_t)] -|x|^{\alpha -d}f(x)}{|x|^{\alpha -d}t}, 
\label{easyformg}
\end{equation}
for  twice continuously differentiable and compactly supported functions $f$,
where $x\in\mathbb{R}^d\backslash\{0\}$. That is to say 
\begin{equation}
\mathcal{L}^\circ f(x) = \frac{1}{h(x)}\mathcal{L}(hf)(x),
\label{generatordoobh}
\end{equation}
where $h(x) = |x|^{\alpha - d}$ and $\mathcal{L}$ is the infinitesimal generator of the stable process, which has action
\[
\mathcal{L}f(x) =\texttt{a}\cdot\nabla f(x)+ \int_{\mathbb{R}^d} [f(x+y) -f(x) - \mathbf{1}_{(|y|\leq 1)}y\cdot\nabla f(x)]\Pi(\d y),\qquad x\in\mathbb{R}^d,
\]
for  twice continuously differentiable and compactly supported functions $f$, where $\texttt{a}$ is an appropriately valued vector in $\mathbb{R}^d$.
Straightforward algebra, appealing to the fact that $\mathcal{L}h =0$  shows that, for a twice continuously differentiable and compactly supported functions, $f$, the infinitesimal generator of the conditioned process \eqref{generatordoobh} takes the form
\begin{align}
\mathcal{L}^\circ f(x)= \texttt{a}\cdot\nabla f(x)+ \int_{\mathbb{R}^d} [f(x+y) -f(x) - \mathbf{1}_{(|y|\leq 1)}y\cdot\nabla f(x)]\frac{h(x+y)}{h(x)}\Pi(\d y),
\label{circgenerator}
\end{align}
for $|x|>0,$ where
$\Pi$ is the stable L\'evy measure given in \eqref{jumpmeasure}.
The integral component in  $\mathcal{L}^\circ$ tells us that the instantaneous rate at which jumps arrive for the conditioned process when positioned at $x$ is given by  
\begin{align*}
\Pi^\circ(x, B) &:=\int_B \frac{h(x+y)}{h(x)}\Pi(\d y) \notag\\
%&=c(\alpha)\int_{\mathbb{S}_{d-1}}\sigma_1(\d\phi)\int_{(0,\infty)}\mathbf{1}_{B}(r\phi)\frac{{\rm d} r}{r^{\alpha+1}} \frac{|x + r\phi|^{\alpha - d}}{|x|^{\alpha - d}},
\end{align*}
where $|x|>0$ and $B$ is Borel in $\mathbb{R}^d$.

 As a consequence of jump rates entirely characterising $(X, \mathbb{P}^\circ_x)$, $x\in\mathbb{R}^d\backslash\{0\}$,  if $\mathbf{P}^\circ_{r,\vartheta}$, $r>0$ and $\vartheta\in\mathbb{S}_{d-1}$, is the law of the associated  MAP, we should expect to see a similar calculation to \eqref{onemorechangeofvariable} with the same jump rates for the modulator and ordinator, albeit that there is the opposite  sign in the discontinuity of the ordinator.

To examine the discontinuities of the modulator and ordinator of the MAP under $\mathbb{P}^\circ_{x}$, $x\in\mathbb{R}^d\backslash\{0\}$, suppose that $f$ is a positive, bounded measurable function on $[0,\infty)\times\mathbb{R}\times\mathbb{R}\times\mathbb{S}_{d-1}\times\mathbb{S}_{d-1}$ such that $f(\cdot, \cdot,0, \cdot,\cdot) = 0$.
Write
\begin{align}
\mathbf{E}^\circ_{0,\theta}&\left[\sum_{s>0}f(s,\xi_{s-},\Delta\xi_{s},\Theta_{s-}, \Theta_s)\right]\notag\\
&= \lim_{t\to\infty}\mathbf{E}_{0,\theta}\left[\mathcal{M}_t
\sum_{0<s\leq t}f(s,\xi_{s-},\Delta\xi_{s},\Theta_{s-}, \Theta_s)\right]\label{MCT}
\end{align}
where 
$\mathcal{M}_t  = \exp\{(\alpha-d)\xi_t\}$, $t\geq 0$,is the martingale density corresponding to \eqref{d-circ} and the limit is justified by monotone convergence.
Suppose we write $\Sigma_t$ for the sum term in the final expectation above. The semi-martingale change of variable formula (see for example p86 of \cite{protter}) tells us that 
\[
 \mathcal{M}_t\Sigma_t = \mathcal{M}_0(\theta)\Sigma_0 + \int_0^t \Sigma_{s-}\d \mathcal{M}_s + \int_0^t\mathcal{M}_{s-}\d \Sigma_s +  [\mathcal{M}, \Sigma]_t, \qquad t\geq0,
\]
where $[\mathcal{M}, \Sigma]_t$ is the quadratic co-variation term. 
On account of the fact that  $(\Sigma_t, t\geq0)$, has bounded variation, the latter term takes the form $[\mathcal{M}, \Sigma]_t = \sum_{s\leq t}\Delta M_t \Delta \Sigma_ t$. As a consequence,
\begin{equation}
 \mathcal{M}_t\Sigma_t = \mathcal{M}_0(\theta)\Sigma_0 + \int_0^t \Sigma_{s-}\d \mathcal{M}_s + \int_0^t\mathcal{M}_{s}\d \Sigma_s , \qquad t\geq0.
\label{smg}
\end{equation}
Moreover, after taking expectations and then taking limits as $t\to\infty$ with the help of \eqref{MCT} and monotone convergence, as the first in integral in \eqref{smg} is a martingale and $\Sigma_0 =0$, the only surviving terms give us 
\begin{align*}
&\mathbf{E}^\circ_{0,\theta}\left[\sum_{s>0}f(s,,\xi_{s-},\Delta\xi_{s},\Theta_{s-}, \Theta_{s})\right] \\
&=\mathbf{E}_{0,\theta}\left[\sum_{s>0} {\rm e}^{(\alpha-d)(\xi_{s-}+\Delta\xi_s)}
f(s,,\xi_{s-},\Delta\xi_{s},\Theta_{s-}, \Theta_{s})\right]\\
&=\mathbb{E}_{\theta}\Bigg[
\int^{\infty}_{0}\hspace{-3pt}\d s\int_{ \mathbb{S}_{d-1}}\hspace{-3pt}{\sigma_1(\d\phi)}\int_{0}^\infty\hspace{-3pt}%c(\alpha)
\frac{c(\alpha)\d r}{r^{1+\alpha}} |X_{s-}|^{\alpha - d}
\left|{\rm Arg}(X_{s-}) + \frac{r\phi }{|X_{s-}|^{\alpha - d}}\right|^{\alpha - d}\notag\\
&\hspace{1cm}f\Bigg(\varphi(s), \log|X_{s-}|, \log\left|{\rm Arg}(X_{s-})+\frac{r\phi}{|X_{s-}|}\right|, {\rm Arg}(X_{s-}), \frac{{\rm Arg}(X_{s-})+\frac{r\phi}{|X_{s-}|}}{\left|{\rm Arg}(X_{s-})+\frac{r\phi}{|X_{s-}|}\right|}\Bigg)
\Bigg]\notag\\
\end{align*}
Now picking up from the second equality of \eqref{onemorechangeofvariable} with $f(\cdot, \xi, \Delta, \cdot,\cdot)$ replaced by $\exp((\alpha-d)(\xi+\Delta))f(\cdot,\xi, \Delta,\cdot,\cdot)$, we get

\begin{align*}
&\mathbf{E}^\circ_{0,\theta}\left[\sum_{s>0}f(s,\xi_{s-},\Delta\xi_{s},\Theta_{s-}, \Theta_{s})\right] \notag\\
%%%1
&=
\mathbf{E}_{0,\theta}\Bigg[\int^{\infty}_{0}\hspace{0pt}\d v\int_{ \mathbb{S}_{d-1}}\hspace{0pt}{\sigma_1(\d\phi)}\int_{0}^\infty \d r 
\frac{c(\alpha)}{r^{1+\alpha}}{\rm e}^{(\alpha-d)\xi_v}\left|\Theta_{v}+r\phi\right|^{\alpha-d}\notag\\
&\hspace{5cm}
f\left(v, \xi_v, \log\left|\Theta_{v}+r\phi\right|,\Theta_v,  \frac{\Theta_{v}+r\phi}{\left|\Theta_{v}+r\phi\right|}\right)\Bigg]\notag\\
%%%4
&=\mathbf{E}_{0,\theta}\left[\int^{\infty}_{0}\d v\int_{ \mathbb{R}^{d}} \d z \frac{c(\alpha) }{|z|^{\alpha+d}}{\rm e}^{(\alpha-d)\xi_v}
\left|\Theta_{v}+z\right|^{\alpha-d}
f\left(v,\xi_v, \log\left|\Theta_{v}+z\right|, \Theta_v, \frac{\Theta_{v}+z}{\left|\Theta_{v}+z\right|}\right)\right]\notag\\
%%%6
&=\mathbf{E}_{0,\theta}\left[\int^{\infty}_{0}\d v\int_{\mathbb{S}_{d-1}}\sigma_1(\d\phi)\int_{0}^\infty\d r \frac{c(\alpha)r^{\alpha-1}}{|r\phi-     \Theta_{v}|^{\alpha+d}}
{\rm e}^{(\alpha-d)\xi_v}
f\left(v, \xi_v, \log r, \Theta_v, \phi\right)\right]\notag\\
&=\int_{\mathbb{S}_{d-1}}\int_0^\infty\int_{\mathbb{R}}\int_{\mathbb{S}_{d-1}}\int_{\mathbb{R}}V_\theta(\d v, \d x, \d \vartheta){\rm e}^{(\alpha-d)x}\sigma_1(\d\phi)\d y \frac{c(\alpha){\rm e}^{y\alpha}}{|{\rm e}^{y}\phi-     \vartheta|^{\alpha+d}}f\left(v,x, y, \vartheta, \phi\right)\notag\\
&=\int_{\mathbb{S}_{d-1}}\int_0^\infty\int_{\mathbb{R}}\int_{\mathbb{S}_{d-1}}\int_{\mathbb{R}}V_\theta(\d v, \d x, \d \vartheta){\rm e}^{(\alpha-d)x} \sigma_1(\d\phi)\d w \frac{c(\alpha){\rm e}^{wd}}{|\phi-    {\rm e}^{w} \vartheta|^{\alpha+d}}f\left(v,x,  -w, \vartheta, \phi\right)\notag\\
&=\int_{\mathbb{S}_{d-1}}\int_0^\infty\int_{\mathbb{R}}\int_{\mathbb{S}_{d-1}}\int_{\mathbb{R}}V_\theta(\d v, \d x, \d \vartheta){\rm e}^{(\alpha-d)x}\sigma_1(\d\phi)\d w \frac{c(\alpha){\rm e}^{wd}}{|{\rm e}^{w} \phi-    \vartheta|^{\alpha+d}}f\left(v, x, -w, \vartheta, \phi\right),
\end{align*}
 where we convert to Cartesian coordinates in the second equality and  back skew product variables in the third equality. In the penultimate equality we simply change variables $y = -w$ and in the final equality we note that $|\phi-    {\rm e}^{w} \vartheta|^2 = | {\rm e}^{w} \phi-   \vartheta|^2$ on account of the fact that 
\[
 (\phi-    {\rm e}^{w} \vartheta)\cdot(\phi-    {\rm e}^{w} \vartheta) = 1 - 2{\rm e}^{w} \vartheta\cdot\phi + {\rm e}^{2w} 
=({\rm e}^{w} \phi-    \vartheta)\cdot( {\rm e}^{w}\phi-    \vartheta)
\]

In conclusion, we have 
\begin{align}
&\mathbf{E}^\circ_{0,\theta}\left[\sum_{s>0}f(s,\xi_{s-},\Delta\xi_{s},\Theta_{s-}, \Theta_{s})\right] \notag\\
%%%1
&=\int_0^\infty\int_{\mathbb{R}}\int_{\mathbb{S}_{d-1}}\int_{\mathbb{S}_{d-1}}\int_{\mathbb{R}}V^\circ_\theta(\d v, \d x, \d \vartheta)\sigma_1(\d\phi)\d w \frac{c(\alpha){\rm e}^{wd}}{|{\rm e}^{w} \phi-    \vartheta|^{\alpha+d}}f\left(v, x, -w, \vartheta, \phi\right),
\label{statedependentjump}
\end{align}
where for $s>0$, $x\in\mathbb{R}$, $\vartheta\in\mathbb{S}_{d-1}$, 
\[
V^\circ_\theta(\d s, \d x, \d \vartheta) = \mathbf{P}^\circ_{0,\theta}(\xi_s \in \d x, \Theta_s\in \d\vartheta)\d s =V_\theta(\d s, \d x, \d \vartheta) {\rm e}^{(\alpha -d)x},
\]
is the space-time potential of $(\xi, \Theta)$,

Comparing the right-hand side of \eqref{statedependentjump} above with that of \eqref{MAPdeltatostable}, it now becomes clear that the jump structure of $(\xi,\Theta)$ under $\mathbf{P}^\circ_{x,\theta}$, $x\in\mathbb{R}$, $\theta\in\mathbb{S}_{d-1}$, is precisely that of $(-\xi,\Theta)$ under $\mathbf{P}_{x,\theta}$, $x\in\mathbb{R}$, $\theta\in\mathbb{S}_{d-1}$.

In conclusion, this is now sufficient to deduce that $(X,\mathbb{P}^\circ_{Kx})$, $|x|>0$, is equal in law to $(KX_{\eta(t)}, t\geq 0)$ under $\mathbb{P}_{x}$, as both are self-similar Markov processes with the same underlying MAP.
\end{proof}

Reviewing the proofs of the previous two theorems we also have the below Corollary at no extra cost. 
\begin{corollary}[$\diamondsuit$]\label{circtominus} When $d =1$, the process $(\xi^\circ, J^\circ)$ is equal in law to $(-\xi,J)$ and, when $d\geq 2$, the process $(\xi^\circ, \Theta^\circ)$ is equal in law to $(-\xi, \Theta)$.
\end{corollary}

\part{One dimensional results on (-1,1)}
\setcounter{section}{0}

\section{Wiener--Hopf precursor}

Having developed the relationship between several path functionals of stable processes and the class of pssMp, we shall go to work and show how an explicit understanding of their Lamperti transform leads to a suite of  fluctuation identities for the stable process.  In essence, we will see that all of the identities we are interested in can be rephrased in terms of the  L\'evy processes that underly  the three examples of the Lamperti transform given in Sections \ref{killed*}, \ref{CSPsection} and \ref{radius}. The specific nature of the Wiener--Hopf factorisation for these three classes, together with some associated classical theory for the first passage problem over a fixed level is what gives us access to explicit results. %Throughout this chapter, we keep to our usual notation that $X = (X_t, t\geq 0)$ with probabilities $\mathbb{P}_x$, $x\in\mathbb{R}$ (reserving the special notation $\mathbb{P}$ in place of $\mathbb{P}_0$) is a one-dimensional stable process.

Recall that the   Wiener--Hopf factorisation \eqref{SWHF}, when explicit, gives access to the Laplace exponents of the ascending and descending ladder height processes. In turn,  this also gives us access to the ladder height potentials 
\begin{equation}
U(x) = \int_0^\infty \mathbb{P}(H_t \leq x)\d t\quad \text{ and }\quad \hat{U}( x) = \int_0^\infty \mathbb{P}(\hat{H}_t \leq   x)\d t, \qquad x\geq 0.
\label{ladderpotentials}
\end{equation}
whose respective Laplace transforms are $\kappa^{-1}$ and $\hat{\kappa}^{-1}$, assuming that an inversion is possible. The basic pretext of the general Wiener--Hopf theory that becomes of use to us is the so-called triple law at first passage and its various simpler forms; see Chapter 7 of \cite{Kypbook}, Chapter 5 of \cite{D} or \cite{DK}.

\begin{theorem}[$\heartsuit$]
\label{triple}$\mbox{ }$ 
\begin{itemize}
\item[(i)] Suppose that $Y$ is a (killed) L\'evy process, but not a compound Poisson process, and neither $Y$ nor $-Y$ is a subordinator. Write ${\rm P}$ for its law when issued from the origin. Then, %up to a multiplicative constant,
for each ${a}>0$, we have on $u>0$, $v\geq y$, $y\in[0,{a}]$, $s,t\geq
0$,
\begin{eqnarray}
&& \mathrm{P}(
 Y_{T^+_a}\! -\! {a}\in {\D}u, {a} \!-\! Y_{T^+_a-} \in {\D}v , {a}\!-\! \overline{Y}_{T^+_a
-}\in {\D}y\, ;\, T^+_a<\infty) \notag\\
&&\hspace{5cm} = U( {a}-{\D}y){\hat U}({\D}v-y)\Lambda({\D}u + v),
\label{creep?}
\end{eqnarray}
where $\Lambda$ is the L\'evy measure of $Y$, $\overline{Y}_t = \sup_{s\leq t}Y_s$, $t\geq 0$, and $T^+_a = \inf\{t>0: Y_t>a\}$.

\item[(ii)] In the case that $Y$ is a subordinator we have, for each ${a}>0$,  $u>0$, $y\in[0,{a}]$, $s,t\geq
0$,
\begin{eqnarray}
&& \mathrm{P}(
 Y_{T^+_a}\! -\! {a}\in {\D}u, {a} \!-\! Y_{T^+_a-} \in {\D}v , {a}\!-\! \overline{Y}_{T^+_a
-}\in {\D}y\, ;\, T^+_a<\infty) \notag\\
&&\hspace{6cm} = U( {a}-{\D}y)\Lambda({\D}u + y),
\label{creep?}
\end{eqnarray}
where, again, $\Lambda$ is the L\'evy measure of $Y$.
\end{itemize}
\end{theorem}
The total mass of the right hand side of \eqref{creep?} is not necessarily equal to unity. One must also take account of the probability that the L\'evy process crosses the level $a$ continuously. That is to say, one must also take account of the event of creeping, $\{Y_{\tau^+_a } =a\}$. In the setting that we will consider the above theorem, this is not necessary since L\'evy processes we will work with are derived from the stable process. The property of no creeping for the stable process will translate to the same property for the processes we use.

In the case that the ascending ladder height process is killed, in which case, from \eqref{bernstein}, the killing rate is $\kappa(0)$, we also get a simple formula for the crossing probability; see  for example Proposition VI.17 of \cite{bertoin}. 
\begin{lemma}[$\heartsuit$]\label{Utail} For $a>0$,
\[
\mathbb{P}(\tau^+_a<\infty) = \kappa(0)U(a,\infty).
\]
\end{lemma} 
\section{First exit from an interval}\label{stripexit}

Lemma  \ref{precorollary} deals with the event of first exit of a stable process from the interval $(-\infty, a)$, for fixed $a>0$. A natural problem to consider thereafter is the event of first exit of a stable process from a bounded interval. Thanks to scaling, it suffices to consider   $(-1,1) = \mathbb{S}_1$. To this end, let us write as usual
\[
\tau^+_1 = \inf\{t>0 : X_t >1\}\,\, \text{ and }\,\, \tau^-_{-1} = \inf\{t>0 : X_t<-1\}.
\]
%where $X$ is a stable process.   As with many of the results in this chapter, we must be careful on occasion to distinguish whether or not the process $X$ is spectrally negative. Recall that  $\alpha\rhohat,\alpha\rho<1$ if and only if $X$ has jumps in both directions. Moreover,  $\alpha\rho=1$ corresponds to the case that $X$ is spectrally negative and $\alpha\rhohat=1$ corresponds to the case that $X$ is spectrally positive.  Throughout the computations below, unless specifically discussed, the spectrally one-sided cases will follow directly as a consequence of the methods employed. 

\subsection{Two-sided exit problem}
As a warm-up to the main result in this section, let us start by computing a two-sided exit probability. The following result first appeared in \cite{BGR} in the symmetric setting, followed by \cite{Rog} in the non-symmetric setting. Its proof is based on the method in \cite{KW}.
%In theory one can marginalize the distribution given in the above theorem and produce an identity for $\mathbb{P}_x(\tau^+_a<\tau^-_0)$ as in the corollary below. However it is worth noting that, appealing to the same technique as in Theorem \ref{2sidedtriple}, a simpler proof is at hand. (Note that we do not exclude the case of spectrally negative stable processes.)

\begin{lemma}[$\heartsuit$]\label{twosidedexitnoovershoot}
For $x\in(-1,1)$,
\[
\mathbb{P}_x(\tau^+_1 <\tau^-_{-1}) =2^{1-\alpha} \frac{\Gamma(\alpha)}{\Gamma(\alpha\rho)\Gamma(\alpha\hat\rho)}\int_{-1}^{x} (1+s)^{\alpha\hat\rho-1}(1-s)^{\alpha\rho-1}\D s.
\]
\end{lemma}
\begin{proof}[Proof  \emph{($\heartsuit$)}] Denote by $\mathbf{P}^*$ the law of $\xi^*$ (the L\'evy process associated via the Lamperti transformation to the stable process killed on passing below the origin, see Section \ref{killed*}) and, for $b>0$, let 
\[
\tau^{*,+}_b=\inf\{t>0 : \xi^*_t>b\}.
\] 
Recalling that the range of the stable process killed on exiting $[0,\infty)$ agrees with the range of the exponential of the process $\xi^*$, we have, with the help of Lemma \ref{Utail} and Theorem \ref{Psi*thrm},
\begin{eqnarray*}
\mathbb{P}_x(\tau^+_1 <\tau^-_{-1}) &=& \mathbf{P}^*(\tau^{*, +}_{\log(2/(x+1))}<\infty)\\
&=&\frac{\Gamma(\alpha)}{\Gamma(\alpha\rho)\Gamma(\alpha\hat\rho)}\int_{\log(2/(x+1))}^\infty {\rm e}^{-\alpha\hat\rho y}(1-{\rm e}^{-y})^{\alpha\rho-1} \D y\\
%&=&\frac{\Gamma(\alpha)}{\Gamma(\alpha\rho)\Gamma(\alpha\hat\rho)}\int_0^{(x+1)/2} t^{\alpha\hat\rho-1}(1-t)^{\alpha\rho-1}\D t,
\end{eqnarray*}
%where in the final equality we have applied the change of variable $t = {\rm e}^{-y}$.
and the result follows by making the change of variables 
${\rm e}^{-y} = (s+1)/2$.
\end{proof}

Now we turn to a more general identity around the event of two-sided exit. The reason why we have first proved the above Lemma is that we shall use it to pin down an unknown normalising constant. The following theorem and the method of its proof come from \cite{KW}.

\begin{theorem}[$\heartsuit$]\label{2sidedtriple}  For $x\in(-1,1)$, $u>0$, $y\in[0,1-x]$ and $v\in [y,2]$,
\begin{eqnarray*}
\lefteqn{\mathbb{P}_x(X_{\tau^+_{1} } - {1} \in {\D}u, {1} - X_{\tau^+_{1} - } \in
{\D}v,
{1} - \overline{X}_{\tau^+_{1} - }\in {\D}y\, ; \, \tau^+_1<\tau^-_{-1})}&& \notag\\
&&=\frac{\sin (\pi\alpha\rho)}{\pi}\frac{\Gamma(\alpha+1)}{\Gamma(\alpha\rho)\Gamma(\alpha\hat\rho)}
\frac{(1+x)^{\alpha\hat\rho}({1}-x-y)^{\alpha\rho-1}(v-y)^{\alpha\hat\rho-1}({2}-v)^{\alpha\rho}}{({2}-y)^{\alpha}(u+v)^{\alpha+1}}\D u\, \D v\, \D y.
\end{eqnarray*}

\end{theorem}

\begin{proof}[Proof \emph{($\heartsuit$)}]
The overshoot and undershoot at first passage over the level $1$ for $X$ on the event $\{\tau^+_1<\tau^-_{-1}\}$ are, up to a linear shift transferring $(-1,1)$ to $(0,2)$ (thanks to stationary and independent increments) and logarithmic change of spatial variable, equal to the overshoot and undershoot at first passage over the  level $\log 2$ for $\xi^*$ on the event this first passage occurs before $\xi^*$ is killed.
Note that, 
%As alluded to above, we can extract the desired result by  computing the distribution of the overshoot and undershoots,
%\[
%\left(\xi^{*}_{\tau^{{*},+}_{\log a}} -\log a, \, \log a- \xi^{*}_{\tau^{{*},+}_{\log a} -}, \, \log a - \overline\xi^{*}_{\tau^{{*},+}_{\log a}-}\right),
%\]
%where
%$ \tau^{{*}, +}_{\log a} = \inf\{t> 0 : \xi^{*}_t>\log a\}
%$
%and $\overline\xi^{*}_t = \sup_{s\leq t}\xi^{*}_s$. Indeed, setting $b  =\log a$ and writing $\mathbf{P}^{*}$ for the law of $\xi^{*}$, we have with the help of Theorem \ref{triple}, 
 for $x\in(-1,1)$, $u\geq 0$,  $y\in[0,1-x]$ and $v\in[ y, 2]$, with the help of Theorem \ref{triple} (i), up to a multiplicative constant,
\begin{eqnarray*}
\lefteqn{\mathbb{P}_{1+x}\left( \frac{X_{\tau^+_{{2}}}}{2}-{1}>
{u}/2, \, 
{1}-\frac{X_{\tau^+_{{2}}-}}{2}>v/2\, , {1}-
\frac{\overline{X}_{\tau^+_{{2}}-}}{2}> {y}/2, \, \tau^+_2<\tau^-_0\right)}&&\\
&&=\mathbf{P}^{*} \left(\xi^{*}_{\tau^{{*},+}_{{\log (2/(1+x))}}} -{\log (2/(1+x))} > \log\Big((2+u)/2\Big), \right.\\
&&\hspace{1.2cm}{\log (2/(1+x))}- \xi^{*}_{\tau^{{*},+}_{{\log (2/(1+x))}} -} > -\log \Big((2-v)/2\Big),\\
&&\hspace{1.2cm} \, \left.{\log (2/(1+x))} - \overline{\xi^{*}}_{\tau^{{*},+}_{{\log (2/(1+x))}}-}>-\log \Big((2-y)/2\Big), \tau^{*, +}_{\log(2/(1+x))}<\infty\right)\\
&&=\int_{-\log ((2-y)/2)}^{\log (2/(1+x))}\int_{-\log ((2-v)/2)}^\infty\int_{\log((2+u)/2)}^\infty u^{*}( \log{(2/(1+x))}-r)\\
&&\hspace{6cm}\times\,
{\hat u}^{*}(z-r)\pi^{*}(w + z)\mathbf{1}_{(z\geq r)}{\D}w{\D}z{\D}r,
\end{eqnarray*}
where $\pi^{*}$ is the L\'evy density of $\xi^{*}$ and, moreover,  $u^{*}$ and $\hat{u}^{*}$ are the densities of the renewal measures of the ascending and descending ladder height processes, respectively.
Taking derivatives and noting the relative overshoots and undershoots to the upper boundary are unchanged when we shift the interval $(0,2)$ back to $(-1,1)$, we get 
\begin{eqnarray*}\lefteqn{\mathbb{P}_x( X_{\tau^+_{{1}}}-{1}\in
{\D}{u}, \, 
{1}-X_{\tau^+_{{1}}-}\in {\D }v, \, {1}-
\overline{X}_{\tau^+_{{1}}-}\in {\D}{y},\, \tau^+_1<\tau^-_{-1})}&&\\
&&=u^*\left(\log \Big(\frac{2-y}{1+x}\Big)\right)\hat{u}^{*}\left(\log \Big(\frac{2-y}{2-v}\Big)\right)
%\\&&\hspace{4cm}\times\,
\pi^{*}\left(\log\Big(\frac{2+u}{2-v}\Big)\right)\frac{\D u\D v\D y}{(2-y)(2-v)(2+u)}
\end{eqnarray*}

Given that the Wiener--Hopf factorisation of $\xi^{*}$  has been described in explicit detail in Theorem  \ref{Psi*thrm}, we can now develop the right-hand side above. To this end, recall that the process $\xi^{*}$ belongs to the class of Lamperti-stable processes. Recall that $\pi^*$ was described in \eqref{pi*}. Moreover, as the Laplace transform of $u^*$ and $\hat{u}^*$ are given by $\kappa^{-1}$ and $\hat{\kappa}^{-1}$, which are described in the factorisation \eqref{psi*exponent}, it is straightforward to check that  the ascending and descending ladder height processes have densities given by  
\[
u^{*}(x) = \frac{1}{\Gamma(\alpha\rho)}{\rm e}^{-\alpha\hat\rho x}(1-{\rm e}^{-x})^{\alpha\rho-1}
\]
and 
\[
\hat{u}^{*}(x) =\frac{1}{\Gamma(\alpha\hat\rho)} {\rm e}^{-(1-\alpha\hat\rho)x}(1-{\rm e}^{-x})^{\alpha\hat\rho-1}, 
\]
for $x\geq 0$,
respectively. %Note, in the special case that $X$ (and hence $\xi^*$) is spectrally positive, the descending ladder height process of $\xi^*$ is a pure drift and hence its potential measure is equal to Lebesgue measure restricted to $[0,\infty)$. In this sense we understand $u^*(x)\equiv1$ when $\alpha\rhohat = 1$.
 Putting everything together, straightforward algebra yields the desired result.\end{proof}

\subsection{Resolvent with killing on first exit of $(-1,1)$.}
Let us  consider the potential of the stable process up to exiting the interval $(-1,1)$,
\begin{eqnarray*}
U^{(-1,1)}( x, \D {y})% &=&
% \mathbb{E}_x\left[\int_0^\infty \mathbf{1}_{(X_t \in \D {y}, \, t<\tau^+_1\wedge \tau^-_{-1})}\,\D t\right] \\
&=& \int_0^\infty\mathbb{P}_x(X_t \in \D y, \, t<\tau^+_1\wedge \tau^-_{-1})\,\D t,
\end{eqnarray*}
for $x,y\in(-1,1).$  An explicit identity for its associated density was first given in \cite{BGR} when $X$ is symmetric. Only recently, the non-symmetric case has been given in \cite{PS, KW}.

\begin{theorem}[$\heartsuit$]\label{X (0,1)} 
For $x, y \in(-1,1)$, the measure $U^{(-1,1)}(x, \D y)$ has a density with respect to Lebesgue measure which is almost everywhere equal to   
\begin{equation} u^{(-1,1)}(x,{y}):
  =  \dfrac{2^{1-\alpha}|y-x|^{\alpha-1}}{\Gamma(\alpha\rho)\Gamma(\alpha\hat\rho)}
  \times
  \begin{cases}
    \displaystyle\int_1^{\left|\frac{1-xy}{y-{x}}\right|} 
    (s+1)^{\alpha\rho-1}
    (s-1)^{\alpha\hat\rho-1} \, \d s,
    & {x}\leq   y, \\
    %
%    \dfrac{ (x-y)^{\alpha-1}}{\Gamma(\alpha\rho)\Gamma(\alpha\hat\rho)}
       \displaystyle\int_1^{\left|\frac{1-xy}{y-{x}}\right|}
     (s+1)^{\alpha\hat\rho-1}(s-1)^{\alpha\rho-1}
    \, \d s,
    & x > {y}.
  \end{cases}
  \label{u(-1,1)}
\end{equation}
\end{theorem}
\begin{proof}[Proof \emph{($\heartsuit$)}]
%We exclude from the proof the case of spectral negativity, that is $\alpha\rho<1$. However, restriction can be removed once the identity has been established by applying the identity to the dual in the spectrally negative case. 
%The proof of this result is considerably more straightforward than one might expect and lies with an observation concerning the identity given in Theorem \ref{2sidedtriple}. 
%Let us start by assuming that the process $X$ has both positive and negative jumps or, equivalently, $c_1, c_2>0$. 
As $X$ cannot creep upwards, splitting over the jumps of $X$, we have for $x\in(-1,1)$, $u>0$, $y\in(x,1)$ and $v\in [y,1)$,
\begin{eqnarray}
\lefteqn{
\mathbb{P}_x(X_{\tau^+_{1} } - {1} \in {\D}u,  X_{\tau^+_{1} - } \in
{\D}v,
\overline{X}_{\tau^+_{1} - }\leq y\, ; \, \tau^+_1<\tau^-_{-1})
}&&\notag\\
&&=\mathbb{E}_x\left[\int_0^\infty\int_\mathbb{R}\mathbf{1}_{(X_{t-}\in \D v, \, \overline{X}_{t-}  \leq y, \, t <\tau^+_{-1}\wedge\tau^-_1)}\mathbf{1}_{(X_{t-} + x -1\in\,\D u)} \,N(\D t, \D x)\right],
\label{usecompform}
\end{eqnarray}
where $N$ is a Poisson random measure on $[0,\infty)\times\mathbb{R}$ with intensity $\D t\times \Pi(\D x)$, representing the arrival of jumps in the stable process, and $\Pi$ is the L\'evy measure given by (\ref{stabledensity3}). It follows from the classical compensation formula for Poisson integrals of this type that 
\begin{eqnarray}
\lefteqn{\mathbb{P}_x(X_{\tau^+_{1} } - {1} \in {\D}u,  X_{\tau^+_{1} - } \in
{\D}v,
\overline{X}_{\tau^+_{1} - }\leq y\, ; \, \tau^+_1<\tau^-_{-1})}&&\notag\\
&&=\mathbb{E}_x\left[\int_0^\infty\mathbf{1}_{( X_{t-}\in \D v, \, \overline{X}_{t-}  \leq y, \, t <\tau^+_{-1}\wedge\tau^-_1)} \,\D t\right]\Pi(1-v+\D u)\notag\\
&&=\Gamma(1+\alpha)  \frac{\sin(\pi \alpha \rho)}{\pi}\mathbb{E}_x\left[\int_0^\infty\mathbf{1}_{(  X_{t}\in\, \D v, \, \overline{X}_{t } \leq y, \,t <\tau^+_{-1}\wedge\tau^-_1)} \,\D t\right]\frac{1}{(1-v+u)^{1+\alpha}}\D u\notag\\
&&=\Gamma(1+\alpha)  \frac{\sin(\pi \alpha \rho)}{\pi}U^{(-1,y)}(x, \D v)\frac{1}{(1-v+u)^{1+\alpha}}\D u,
\label{usecompform2}
\end{eqnarray}
where 
\[
U^{(-1,y)}(x, \D v) = \int_0^\infty\mathbb{P}_x(X_t \in \D v, \, t<\tau^+_y\wedge \tau^-_{-1})\,\D t.
\]

From Theorem \ref{2sidedtriple}, we also have that, for $u>0$ and $v\in(-1,1)$ and $y\in[v\vee x,1)$, 
\begin{align*}
&\mathbb{P}_x(X_{\tau^+_{1} } - {1} \in {\D}u,  X_{\tau^+_{1} - } \in
{\D}v,
\overline{X}_{\tau^+_{1} - }\leq y\, ; \, \tau^+_1<\tau^-_{-1})
\\
&=\frac{\sin (\pi\alpha\rho)}{\pi}\frac{\Gamma(\alpha+1)}{\Gamma(\alpha\rho)\Gamma(\alpha\hat\rho)}
\left\{
\int_{v\vee x}^{y}
\frac{(1+x)^{\alpha\hat\rho}({z}-x)^{\alpha\rho-1}(z-v)^{\alpha\hat\rho-1}({1}+v)^{\alpha\rho}}{({1}+z)^{\alpha}(1-v+u)^{\alpha+1}}\D z\,\right\} \D u\, \D v.
\end{align*}
The consequence of this last observation is that, for $0\leq v\vee x\leq y$,  $U^{(-1,y)}(x, \D v)$ is absolutely continuous with respect to Lebesgue measure and, comparing with \eqref{usecompform2}, its density is given by 
\begin{equation}
u^{(-1,y)}(x, v) = \frac{(1+x)^{\alpha\hat\rho}(1+v)^{\alpha\rho}}{\Gamma(\alpha\rho)\Gamma(\alpha\hat\rho)}\left\{
\int_{v\vee x}^{y}
\frac{({z}-x)^{\alpha\rho-1}(z-v)^{\alpha\hat\rho-1}}{(1+z)^{\alpha}}\, %
\D z\right\}.
\label{evaluate integral}
\end{equation}

To evaluate the integral in (\ref{evaluate integral}) we must consider two cases according to the value of $x$ in relation to $v$. To this end, we first suppose that $x\leq v$. We have 
\begin{eqnarray*}
\lefteqn{(1+x)^{\alpha\hat\rho}(1+v)^{\alpha\rho}\int_v^{y}
\frac{({z}-x)^{\alpha\rho-1}(z-v)^{\alpha\hat\rho-1}}{(1+z)^{\alpha}}\, %
\D z} &&\\
&&=  (v-x)^{\alpha-2}\int_v^y \left[\frac{(1+v)(z-x)}{(1+z)(v-x)}\right]^{\alpha\rho-1}
\left[\frac{(1+x)(z-v)}{(1+z)(v-x)}\right]^{\alpha\hat\rho-1} \frac{(1+x)(1+v)}{(1+z)^2}\D z\\
&&=(v-x)^{\alpha-1}\int_0^{\frac{(1+x)(y-v)}{(1+y)(v-x)}} (s+1)^{\alpha\rho-1}s^{\alpha\hat\rho -1}\D s,
\end{eqnarray*}
where in the final equality we have changed  variables using  $$s ={(1+x)(z-v)}/{(1+z)(v-x)}.$$ To deal with the case $x>v$, one proceeds as above except that the lower delimiter on the integral in (\ref{evaluate integral}) is equal to $x$, we multiply and dive through by $(x-v)^{\alpha -2}$ and one makes the change of variable $s = {(1+v)(z-x)}/{(1+z)(x-v)}$. This gives us an expression for $u^{(-1,y)}$, which, in turn, through translation, thanks to stationary and independent increments, and scaling, can be transformed into the potential density of $u^{(-1,1)}(x,\d y)$. The details are straightforward and left to the reader.
\end{proof}

A useful corollary of this result is the simple point-set exit probability below which proves to be quite useful in the next section. Its first appearance is in \cite{PS}.
\begin{corollary}[$\heartsuit$]\label{pointset}
For $\alpha\in(1,2)$ and  $x,y\in(-1,1)$,
\begin{align}
\mathbb{P}_x&(\tau^{\{y\}}< \tau^+_1\wedge \tau^-_{-1}) \notag\\
&= (\alpha-1) \frac{|y-x|^{\alpha-1}}{(1-y^2)^{\alpha-1}}
\times
\begin{cases}
      \displaystyle\int_1^{\left|\frac{1-xy}{x-y}\right|}
    (s+1)^{\alpha\rho-1}
    (s-1)^{\alpha\hat\rho-1} \, {\rm d} s,
    & {x}\leq   y, \\
         \displaystyle\int_1^{\left|\frac{1-xy}{x-y}\right|}
    (s-1)^{\alpha\rho-1} (s+1)^{\alpha\hat\rho-1}
    \, {\rm d} s,
    & x > {y}.
  \end{cases}
  \label{pointsetprob}
\end{align}
\end{corollary}
\begin{proof}[Proof \emph{($\heartsuit$)}]
We appeal to a standard technique and note that, for $x,y\in (-1,1)$
\[
 u^{(-1,1)}(x,y)  = \mathbb{P}_x(\tau^{\{y\}}<\tau^+_1\wedge\tau^-_{-1}) u^{(-1,1)}(y,y),
\]
where we may use L'H\^opital's rule to compute $
u^{(-1,1)}(y,y) = \lim_{x\uparrow y}u^{(-1,1)}(x,y).
$
The details are straightforward and left to the reader.
 \end{proof}

 \section{First entrance into a bounded interval}\label{stripentrance}
 In Section \ref{stripexit} we looked at the law of the stable process as it first exits $(-1,1)$. In this section, we shall look at the law of the stable process as it first {\it enters} $(-1,1)$. Accordingly, we introduce the first hitting time of the interval
$(-1,1)$,
\[ \tau^{(-1,1)} = \inf\{ t > 0 : X_t \in (-1,1) \} .
\]
%The next result provides the law of $X_{\tau^{(-1,1)}}$, however, because of the issue of creeping when $X$ is spectrally one-sided, it is necessary to consider the cases of one-sided and two-sided jumps separately when $\alpha\in (1,2)$. 
\subsection{First entry point in  $(-1,1)$} 
The following result was first proved in the symmetric case in \cite{BGR} and in the non-symmetric case in \cite{KPW, PS}. The proof we give is that of the first of  these three references.

\begin{theorem}[$\heartsuit$]\label{interval hitting} Suppose $0<\alpha\rhohat,\alpha\rho<1$.
Let $x > 1$. Then, when $\alpha \in (0,1]$,
\begin{eqnarray}
  \lefteqn{\p_{x}\bigl(X_{\tau^{(-1,1)}} \in \d y, \, \tau^{(-1,1)} < \infty\bigr)}&&\notag\\
  &&= \frac{\sin(\pi\alpha\rhohat)}{\pi}
    (1+x)^{\alpha\rho}(1+y)^{-\alpha\rho}
    (x-1)^{\alpha\rhohat}
    (1-y)^{-\alpha\rhohat}
    (x-y)^{-1}\d y,
    \label{striplowalpha}
\end{eqnarray}
for $y \in (-1,1)$. When $\alpha \in (1,2)$,
\begin{eqnarray}
 \lefteqn{\hspace{-4pt}\p_x(X_{\tau^{(-1,1)}} \in \d y)}&&\notag\\
 &&\hspace{-8pt}= \frac{\sin(\pi\alpha\rhohat)}{\pi}
   (1+y)^{-\alpha\rho}
    (1-y)^{-\alpha\rhohat}
    \Bigg((x-1)^{\alpha\rhohat}
     (1+x)^{\alpha\rho}(x-y)^{-1}
 \notag\\
  &&\hspace{3.5cm} - (\alpha-1)
        \int_1^{x} (t-1)^{\alpha\rhohat-1} (t+1)^{\alpha\rho-1}\, \d t\Bigg)\,\d y,
\end{eqnarray}
for $y \in (-1,1)$. 
\label{stripbigalpha}
\end{theorem}

\begin{proof}[Proof \emph{($\heartsuit$)}]
Just as with the proof of    Theorem \ref{2sidedtriple}, the proof here relies on reformulating the problem at hand in terms of an underlying positive self-similar Markov process. In this case, we will appeal to the censored stable process defined in Section \ref{CSPsection}. In particular, this means that we will prove Theorem \ref{interval hitting} by first proving an analogous result for the interval $(0,a)$.

Define $\tau^{(0,a)} = \inf\{t>0: X_t\in (0,a)\}$. Thanks to stationary and independent increments, it suffices to consider the distribution of $X_{\tau^{(0,a)}}$. The key observation that drives the proof is that, when $X_0 = x>a>0$,  on $\{\tau^{(0,a)}<\infty\}$, 
\[
X_{\tau^{(0,a)}} \equiv x\exp\{\stackrel{_{\leadsto}}{\xi}_{\stackrel{_{\leadsto}}{\tau}^-_{\log (a/x)}}\},
\]
where  $\stackrel{_{\leadsto}}{\xi}$ is the L\'evy process described in Theorem \ref{censoredpsithrm} and
\[
\stackrel{_{\leadsto}}{\tau}^-_{\log a} = \inf\{t>0 : \,\stackrel{_{\leadsto}}{\xi}_t < \log (a/x)\}.
\]  Note, moreover, that $\{\tau^{(0,a)}<\infty\}$ and $\{\stackrel{_{\leadsto}}{\tau}^-_{\log (a/x)}<\infty\}$ are almost surely the same event. If we denote the law of $\stackrel{_{\leadsto}}{\xi}$ by $\stackrel{_{\leadsto}}{\mathbf{P}}$, then, for $\alpha\in(0,2)$ and $y\in (0,a)$, 
\begin{eqnarray*}
\lefteqn{\p_x\bigl(X_{\tau^{(0,a)}} \leq y, \, \tau^{(0,a)} < \infty\bigr)}&&\\
&&=\stackrel{_{\leadsto}}{\mathbf{P}}\left(\log(a/x) -\stackrel{_{\leadsto}}{\mathbf{\xi}}_{\stackrel{_{\leadsto}}{\tau}^-_{\log (a/x)}} \geq \log (a/y),\, \stackrel{_{\leadsto}}{\tau}^-_{\log (a/x)}<\infty \right),
\end{eqnarray*}
and hence
\begin{eqnarray}
\label{logdiff}
\lefteqn{\hspace{-0.7cm}\p_x\bigl(X_{\tau^{(0,a)}} \in \D y, \, \tau^{(0,a)} < \infty\bigr)}&&\notag\\
&&\hspace{-0.7cm}= \frac{1}{y}\left.\frac{\D }{\D z}\stackrel{_{\leadsto}}{\mathbf{P}}\left( \log(a/x) -\stackrel{_{\leadsto}}{\mathbf{\xi}}_{\stackrel{_{\leadsto}}{\tau}^-_{\log (a/x)}}  \leq  z ,\, \stackrel{_{\leadsto}}{\tau}_{\log (a/x)}<\infty \right)\D y\right|_{z= \log (a/y)}
\end{eqnarray}
Note that the dual of the process $\stackrel{_{\leadsto}}{\xi}$ has characteristic exponent given by Theorem \ref{censoredpsithrm}. 
With Theorem \ref{triple} in mind, we can check that in the case $\alpha\in(0,1]$, the factorisation \eqref{factorisation<1}  gives us a  potential density and L\'evy density of the descending ladder process that take the form
\begin{equation}
 \frac{1}{\Gamma(\alpha\rhohat)}
      (1-e^{-x})^{\alpha\rhohat-1} e^{-(1-\alpha)x},\qquad  x>0,
      \label{potentialdescend}
      \end{equation}
      and 
      \[
      - \frac{1}{\Gamma(-\alpha\rhohat)}
      e^{\alpha x} (e^x-1)^{-(\alpha\rhohat+1)}, \qquad x>0,
\]
respectively.

In the case that $\alpha\in(1,2)$, the Wiener--Hopf  factorisation is given by \eqref{factorisation>1} and one again easily checks that  potential density and L\'evy density of the descending ladder process that take the form
\[
\frac{\Gamma(2-\alpha)}{\Gamma(1-\alpha\rho)}
    + \frac{1-\alpha\rhohat}{\Gamma(\alpha\rhohat)}
    \int_x^\infty e^{\alpha\rho z} (e^z - 1)^{\alpha\hat\rho-2}\, \d z ,
\qquad     x > 0, 
    \]
     and 
     \[
      \frac{e^{(\alpha - 1)x}
      (e^x-1)^{-(\alpha\rhohat+1)}}{\Gamma(1-\alpha\rhohat)}
      \bigl( \alpha - 1 + (1-\alpha\rho)e^x \bigr) ,
  \qquad    x > 0,
\]
respectively.

We may now appeal to the two  parts of Theorem \ref{triple} (i) to develop the  right hand side of (\ref{logdiff})  by considering the first passage problem of the ascending ladder process of $-\stackrel{_{\leadsto}}{\xi}$ over the threshold $\log(x/a)$. After a straightforward computation, the identity (\ref{striplowalpha}) emerges for $\alpha\in(0,1]$ once we use stationary and independent increments to shift the interval $(0,2)$ to $(-1,1)$. The case $\alpha\in (1,2)$ requires  the evaluation of an  extra term. More precisely, from the second part of Theorem \ref{triple} (ii), we get
\begin{eqnarray}\label{eq20}
\lefteqn{\p_x(X_{\tau^{(0,a)}} \in \d y)} &&\notag\\
&&= \frac{\sin(\pi\alpha\rhohat)}{\pi}
   y^{-\alpha\rho}
    (a-y)^{-\alpha\rhohat} 
    \Bigg((x-a)^{\alpha\rhohat}
     x^{\alpha\rho-1} y(x-y)^{-1}
 \notag\\
  &&\hspace{1.8cm} {} - a^{\alpha-1}(\alpha\rho-1)
    \int_0^{1-\frac{a}{x}} t^{\alpha\rhohat-1} (1-t)^{1-\alpha}\, \d t\Bigg)\,\d y.
\end{eqnarray}
By the substitution $t=(s-1)/(s+1)$, we deduce
\[
\begin{split}
    \int_0^{1-\frac{a}{x}}& t^{\alpha\rhohat-1} (1-t)^{1-\alpha}\, \d t\\
   & =2^{1-\alpha}\left(
    \int_1^{\frac{2x}{a}-1} (s-1)^{\alpha\rhohat-1} (s+1)^{\alpha\rho -1}\, \d s \right.\\
    &\left.\hspace{2cm}-\int_1^{\frac{2x}{a}-1} (s-1)^{\alpha\rhohat} (s+1)^{\alpha\rho-2}\, \d t\right).    
    \end{split}
\]
Now evaluating the second  term on the right hand side above via integration by parts and substituting back into (\ref{eq20}) yields the required law, again, once we use stationary and independent increments to shift the interval $(0,2)$ to $(-1,1)$.
%%%%%%%%%%%%%%%%%%
%%%%%%%%%%%%%%%%%%
%%THIS IS THE LITTLE ALGEBRA
%\begin{eqnarray*}
%\lefteqn{\p_x\bigl(X_{\tau^{(0,a)}} \in \D y, \, \tau^{(0,a)} < \infty\bigr)}&&\\
%&&=\frac{1}{y}\frac{\sin\pi\alpha\hat\rho}{\pi}
%{\rm e}^{-(1-\alpha +\alpha\hat\rho)(\log(a/y)+\log(x/a))}
%\left(\frac{1- {\rm e}^{-\log(x/a)}}{{\rm e}^{-\log(x/a)} - {\rm e}^{-(\log(a/y)+\log(x/y))}}\right)^{\alpha\hat\rho}(1-{\rm e}^{-(\log(a/y)+\log(x/y))})^{-1}\D y.
%\end{eqnarray*}
\end{proof}

 We also have the following straightforward corollary from \cite{KPW} that gives the probability process never hits   the interval $(-1,1)$, in the case when $\alpha\in (0,1)$. The result can be deduced from  Theorem \ref{interval hitting} by integrating out $y$ in expression (\ref{striplowalpha}), however, we present a more straightforward proof based on Lemma \ref{Utail}. 
 
\begin{corollary}[$\heartsuit$]\label{interval hitting prob}
When $\alpha \in (0,1)$ and $0<\alpha\rhohat,\alpha\rho<1$, for $x > 1$,
\begin{equation}
\p_x(\tau^{(-1,1)} = \infty)
% =\p_{1+x}(\tau^{(0,2)} = \infty)
%\p_x( \tau^{(0,a)} = \infty)
  =2^{1-\alpha} 
  %\frac{\Gamma(1-\alpha\rho)}{\Gamma(\alpha\rhohat)\Gamma(1-\alpha)}
 % \int_0^{\frac{x-1}{x+1}} t^{\alpha\rhohat - 1} (1-t)^{-\alpha} \, \d t .
 \frac{\Gamma(1-\alpha\rho)}{\Gamma(\alpha\rhohat)\Gamma(1-\alpha)}
 \int_1^x (s-1)^{\alpha\hat\rho -1}(s+1)^{\alpha\rho-1}\d s.
 \label{readerexercise}
  \end{equation}
  For $x<-1$, $\p_x(\tau^{(-1,1)} = \infty) = \p_{-x}(\tau^{(-1,1)} = \infty)|_{\rho\leftrightarrow\hat\rho}$, where $\rho\leftrightarrow\hat\rho$ means the roles of $\rho$ and $\hat\rho$ have been interchanged.
\end{corollary}
\begin{proof}[Proof \emph{($\heartsuit$)}]Appealing to Lemma \ref{Utail} and recalling \eqref{potentialdescend} we have, for $x>a$, that  
\begin{eqnarray*}
\p_x( \tau^{(0,a)} = \infty)& =&\stackrel{_{\leadsto}}{\mathbf{P}}_{\log (x/a)}({\stackrel{_{\leadsto}}{\tau}^-_{0}} =\infty)  
\\
& =&1 - \frac{\Gamma(1 - \alpha\rho )}{\Gamma(1 - \alpha )}\int_{\log (x/a)}^\infty \frac{1}{\Gamma(\alpha\hat\rho)}{\rm e}^{-(1-\alpha) y} (1-{\rm e}^{-y})^{\alpha\hat\rho-1}\d y \\
&=&
1- \frac{\Gamma(1 - \alpha\rho )}{\Gamma(\alpha\hat\rho)\Gamma(1 - \alpha )}\int_{(x-a)/x}^1 t^{\alpha\hat\rho -1}(1-t)^{-\alpha}\d t,
  \end{eqnarray*}
  where in the last equality we have performed the change of variable $t = 1-{\rm e}^{-y}$. The desired probability  for $x>1$ now follows as a straightforward consequence of the beta integral and using stationary and independent increments to shift the interval $(0,2)$ to $(-1,1)$ as in the proof of Theorem \ref{interval hitting}. 
  We arrive at 
  \[ 
\p_x(\tau^{(-1,1)} = \infty)
% =\p_{1+x}(\tau^{(0,2)} = \infty)
%\p_x( \tau^{(0,a)} = \infty)
  = \frac{\Gamma(1-\alpha\rho)}{\Gamma(\alpha\rhohat)\Gamma(1-\alpha)}
  \int_0^{\frac{x-1}{x+1}} t^{\alpha\rhohat - 1} (1-t)^{-\alpha} \, \d t 
   \]
   and the statement of the theorem for this range of $x$ follows by performing a change of variable $s =(t-1)/(t+1)$.
  The probability for $x<-1$ follows by anti-symmetry.
\end{proof}

\subsection{Resolvent with killing on first entry to $(-1,1)$}

Next, we are interested in the potential 

\[
U^{(-1,1)^{\rm c}}(x, {\rm d}y )  = \int_0^{\infty} \mathbb{P}_x(X_t \in \,\d y, \, t< \tau^{(-1,1)}) \,\D t , \qquad x,y\in(-1,1)^{\rm c}.
\]
%The theorem below  gives us an identity for the above potential in the case that $x,y\in(-1,1)^{\rm c}$ and for all values of $\alpha$.% The cases when $x>a$, $y<0$ and when $\alpha\in(1,2)$ are missing. After the proof, we remark on some of the difficulties in providing a more complete treatment. 
The theorem below was first presented only very recently in an incomplete form in \cite{KW} and a complete form in  \cite{PS}. Our proof is different to both of these references. Moreover, in our proof, we see for the first time the use of the Riesz--Bogdan--\.Zak transform.

\begin{theorem}[$\heartsuit$]\label{potential int}
For $y>x>1$, the measure $U^{(-1,1)^{\rm c}}(x, \D y)$  has a density given by 
\begin{align*}%\label{potentialoutof}
u&^{(-1,1)^{\rm c}}(x,y) \notag \\
 &=\dfrac{2^{1-\alpha}}{\Gamma(\alpha\rho)\Gamma(\alpha\hat\rho)}
\Bigg(|y-x|^{\alpha-1}
\int_1^{\left|\frac{1-xy}{y-{x}}\right|}
     (s+1)^{\alpha\rho-1}(s-1)^{\alpha\hat\rho-1}
    \, \d s 
 \notag  \\
   &\hspace{1cm} -(\alpha-1)_+\int_1^{x}
    (s+1)^{\alpha\rho-1}
    (s-1)^{\alpha\hat\rho-1} \, {\rm d} s
    \int_1^{y}
     (s+1)^{\alpha\hat\rho-1}(s-1)^{\alpha\rho-1}
    \, \d s
    \Bigg),
    \end{align*}
where $(\alpha-1)^+ = \max\{0, \alpha-1\}$.
Moreover, if $x>y>1$ then  
\[
u^{(-1,1)^{\rm c}}(x,y) = u^{(-1,1)^{\rm c}}(y,x)|_{\rho\leftrightarrow\hat\rho}.
\]
% where the $\rho\leftrightarrow\hat\rho$ is understood to mean that the roles of $\rho$ and $\hat\rho$ are interchanged. 
 If $x>1$, $y<-1$, then 
\begin{align*}
u&^{(-1,1)^{\rm c}}(x,y)  \\
&=\dfrac{\sin(\alpha\rhohat)}{\sin(\alpha\rho)}\dfrac{2^{1-\alpha}}{\Gamma(\alpha\rho)\Gamma(\alpha\rhohat)}
\Bigg(|y-x|^{\alpha-1}\int_1^{\left|\frac{1-xy}{y-{x}}\right|}
     (s+1)^{\alpha\rho-1}(s-1)^{\alpha\hat\rho-1}
    \, \d s \notag\\
&\hspace{1cm}-(\alpha-1)^+\int_1^{x}
    (s+1)^{\alpha\rho-1}
    (s-1)^{\alpha\hat\rho-1} \, {\rm d} s
    \int_1^{|y|}
     (s+1)^{\alpha\rho-1}(s-1)^{\alpha\hat\rho-1}
    \, \d s\Bigg).
\end{align*}
 Finally, if $x<-1$, then $u^{(-1,1)^{\rm c}}(x,y) = u^{(-1,1)^{\rm c}}(-x,-y)|_{\rho\leftrightarrow\hat\rho}$.
\end{theorem}
\begin{proof}[Proof \emph{($\diamondsuit$)}]
Let us write 
\[
U_\circ^{(-1,1)}(x,\d y) =\int_0^\infty  \mathbb{P}_x^\circ (X_t \in \d y, \, t<\tau^+_1\wedge \tau^-_{-1})\d t, \qquad |x|,|y|<1,
\]
where we recall that the process $(X,\mathbb{P}^\circ_x)$, $x\in\mathbb{R}\backslash\{0\}$, is the result of the change of measure \eqref{updownCOM7} appearing in the Riesz--Bogdan--\.Zak transform, Theorem \ref{th:BZ}.
Let us preemptively assume that $U^{(-1,1)}_\circ(x,\d y)$ has a density with respect to Lebesgue measure, written $u^{(-1,1)}_\circ(x,y)$, $|x|,|y|<1$.

On the one hand, we have, for  $|x|,|y|<1$,
\begin{equation}
u^{(-1,1)}_\circ(x,y) = \frac{h(y)}{h(x)}u^{(-1,1)}_{\{0\}}(x,y), 
\label{H1}
\end{equation}
where $h$ was given in \eqref{constantsinh} and $ u^{(-1,1)}_{\{0\}}(x,y)$ is the assumed density of 
\[
U^{(-1,1)}_{\{0\}}(x, \d y) = \int_0^\infty \mathbb{P}_x(X_t \in \d  y, \, t< \tau^{\{0\}} \wedge \tau^+_1\wedge \tau^-_{-1})\d t.
\]
By path counting and the Strong Markov Property, we have that
\begin{equation}
U^{(-1,1)}(x,\d y) = U^{(-1,1)}_{\{0\}}(x, \d y)  + \mathbb{P}_x(\tau^{\{0\}} <\tau^+_1\wedge \tau^-_{-1})U^{(-1,1)}(0,\d y),
\label{H1.1}
\end{equation}
for $|x|,|y|<1$, where we interpret the second term on the right-hand side as zero if $\alpha\in(0,1]$. Note that the existence of the density in Theorem \ref{X (0,1)} together with the above equality ensures that the densities $u^{(-1,1)}_\circ$ and $u^{(-1,1)}_{\{0\}}$ both exist. Combining \eqref{H1} and \eqref{H1.1}, we thus have 
\begin{equation}
u^{(-1,1)}_\circ(x,y)  =\frac{h(y)}{h(x)}\left(u^{(-1,1)}(x,y) - \mathbb{P}_x(\tau^{\{0\}} <\tau^+_1\wedge \tau^-_{-1})u^{(-1,1)}(0,y)\right).
\label{H1.2}
\end{equation}

On the other hand, the Riesz--Bogdan--\.Zak transform ensures that, for bounded measurable $f$, 
\begin{align}
\int_{(-1,1)}  f(y )u^{(-1,1)}_\circ(x, y)\d y &=\mathbb{E}_{-Kx} \left[\int_0^\infty   f(-KX_{s} ) |X_s|^{-2\alpha}\mathbf{1}_{
(s<\tau^{(-1,1) }
)
}
\d s\right]\notag\\
&=\int_{(-1,1)^{\rm c}} f(-Kz)|z|^{-2\alpha} u^{(-1,1)^{\rm c}}(-Kx, z)\d z\notag\\
&=\int_{(-1,1)} f(y)|y|^{2\alpha} u^{(-1,1)^{\rm c}}(-Kx, -Ky)|y|^{-2}\d y,
\label{preputtogether}
\end{align}
where the density $u^{(-1,1)^{\rm c}}$ is ensured by the density in the integral on the left-hand side above and we have used the easily proved fact that ${\rm d}(Ky) = -  y^{-2}\d y$, where $Ky = y/|y|^2=1/y$. Putting \eqref{preputtogether} and 
\eqref{H1.2} together, noting that $K(Kx) = Kx$ and $|Kx - Ky| =|x-y| /|x||y|$, we conclude that, for $|x|, |y|>1$, 
\begin{align}
&u^{(-1,1)^{\rm c}}(-x, -y)\notag\\
&=u^{(-1,1)^{\rm c}}(x, y)|_{\rho\leftrightarrow \hat\rho}\notag\\
&=|y|^{2\alpha-2}%u^{(-1,1)}_\circ(Kx,-K y)
\frac{h(Ky)}{h(Kx)}\left(u^{(-1,1)}(Kx,Ky) - \mathbb{P}_{Kx}(\tau^{\{0\}} <\tau^+_1\wedge \tau^-_{-1})u^{(-1,1)}(0,Ky)\right).
\label{feeditin}\end{align}
If we now take, for example,  $y>x>1$,  with the help of \eqref{pointsetprob} and  \eqref{u(-1,1)} we can develop \eqref{feeditin} and get
\begin{align*}
&u^{(-1,1)^{\rm c}}(x, y)\\
&=\dfrac{2^{1-\alpha}}{\Gamma(\alpha\rho)\Gamma(\alpha\hat\rho)}
\Bigg(|y-x|^{\alpha-1}
\int_1^{\left|\frac{1-xy}{y-{x}}\right|}
     (s+1)^{\alpha\rho-1}(s-1)^{\alpha\hat\rho-1}
    \, \d s 
   \\
   &\hspace{1cm} -(\alpha-1)_+\int_1^{x}
    (s+1)^{\alpha\rho-1}
    (s-1)^{\alpha\hat\rho-1} \, {\rm d} s
    \int_1^{y}
     (s+1)^{\alpha\hat\rho-1}(s-1)^{\alpha\rho-1}
    \, \d s
    \Bigg),
\end{align*}
where we have used again that $|Kx - Ky| =|x-y| /|x||y|$, in particular that $|1-KxKy| = |1-xy|/|x||y|$.
With some additional minor computations,  the remaining cases follow similarly. The details are left to the reader.
\end{proof}

\begin{remark}\rm
The conditioning with the help of the Strong Markov Property in \eqref{H1.1},  which can be seen as counting paths of the stable process according to when they first exit the interval $(-1,1)$, is a technique that we will see several times in this text for computing potentials. It is a technique that is commonly used in much of the potential analytic literature concerned with first passage problems of stable processes. \cite{Ray} referred to this technique as producing  {\it D\'esir\'e Andr\'e} type equations. This is a somewhat confusing name for something which is, in modern times,  otherwise  associated with a reflection principle for the paths of Brownian motion or random walks. Nonetheless, the commonality to both uses of the name `D\'esir\'e Andr\'e equations' boils down to straightforward {\it path counting}. This is the terminology we prefer to use here. 
\end{remark}

\begin{remark}\rm As an exercise, and to appreciate the spirit in which we present this review article, the reader is now encouraged to return to Corollaries \ref{interval hitting prob} and \ref{pointset} and to affirm the robustness of the use of the Riesz--Bogdan--\.Zak transform by showing the equivalence of the identies \eqref{readerexercise} from that of \eqref{pointsetprob}.
\end{remark}

\section{First hitting of the boundary of the interval $(-1,1)$}

In the previous section, we looked at the law of the time to first hitting the origin for a one-dimensional stable processes when $\alpha\in(1,2)$. Let us define the hitting times 
\[
\tau^{\{b\}} = \inf\{t> 0 : X_t = b\},
\]
for $b\in \mathbb{R}$, and consider the two point hitting problem of evaluating $\mathbb{P}_x(\tau^{\{-1\}}<\tau^{\{1\}})$ for $a,b,x\in\mathbb{R}$. Naturally for this problem to make sense, we need to assume, as in the previous section, that $\alpha\in(1,2)$. The two point hitting problem is a classical problem for Brownian motion. However, for the case of a stable process, on account of the fact that it  may wander either side of the points $a$ and $b$ before hitting  one of them, the situation is significantly different. 
One nice consequence of the two point hitting problem is that it turns out that it gives us easy access to the potential of the stable process up to first hitting of a point.

\subsection{Two point hitting problem}
It turns out that censoring the stable process is a useful way to analyse this problem. Indeed if we write $\cenxi$ for the L\'evy process which drives the Lamperti transformation of the censored stable process (cf. Section \ref{CSPsection}) and denote its probabilities by   $\stackrel{_\leadsto}{\mathbf{P}}_x$, $x\in\mathbb{R}$, then by spatial homogeneity,
\begin{equation}
\mathbb{P}_{x}(\tau^{\{1\}}<\tau^{\{-1\}})=
\mathbb{P}_{1+x}(\tau^{\{2\}}<\tau^{\{0\}}) =  \stackrel{_\leadsto}{\mathbf{P}}_{\log (1+x)}(\stackrel{_\leadsto}{\tau}^{\{\log 2\}}<\infty),
\label{hitone}
\end{equation}
where
\[
\stackrel{_\leadsto}{\tau}^{\{\log 2\}} = \inf\{t>0: \, \cenxi_t =\log 2 \}.
\]
Thus the two-point hitting problem for the stable process is reduced to a single-point hitting problem for the  L\'evy process associated to the censored stable process via the Lamperti transformation. Moreover, the general theory of L\'evy processes for which single points are not polar gives us direction here. Indeed, it is known that the potential $\int_0^\infty \mathbf{P}_x(\cenxi_t \, \in \d y)\d t$ has a density, which, thanks to stationary and independent increments, depends on $x-y$, say $\stackrel{_\leadsto}{u}(x-y)$, and fuels the formula
%Accordingly the right-hand side of \eqref{hitone}  can be written 
\begin{equation}
%\mathbb{P}_x(\tau^{\{b\}}<\tau^{\{0\}}) 
 \stackrel{_\leadsto}{\mathbf{P}}_{\log (1+x)}(\stackrel{_\leadsto}{\tau}^{\{\log 2\}}<\infty)=  \frac{\stackrel{_\leadsto}{u}(-\log((1+x)/2))}{\stackrel{_\leadsto}{u}(0)}.
\label{workoutexplicitly}
\end{equation}
See for example Corollary II.18  of \cite{bertoin}.
We can derive an explicit identity for the potential density $\stackrel{_\leadsto}{u}$, and thus feed it into the right-hand side of \eqref{hitone}, by inverting its Laplace transform. We have 
\begin{equation}
\int_\mathbb{R}{\rm e}^{z x}\stackrel{_\leadsto}{u}(x){\rm d}x = 
%&\int_0^\infty \stackrel{_\leadsto}{\mathbf{E}}[{\rm e}^{z \stackrel{_\leadsto}{\xi}_t}]{\rm d}t\notag\\
\frac{1}{\stackrel{_\leadsto}{\Psi}(z)}=
%& = &\frac{1}{
%\stackrel{_\leadsto}{\Psi}(-\i z)  
%}\notag\\
 \frac{\Gamma(-{z})}{\Gamma(\alpha\rho -{z})}
    \frac{\Gamma(1 - \alpha + {z})}{\Gamma(1 - \alpha\rho + {z})}, 
    \label{invertthisplease}
\end{equation}
for ${\rm Re}(z)\in(0, \alpha-1)$. More generally, $\stackrel{_\leadsto}{\Psi}(-\i z)$ is well defined as a Laplace exponent for ${\rm Re}(z)\in (\alpha\rho-1,\alpha\rho)$, having roots at $0$ and $\alpha-1$. As $-\stackrel{_\leadsto}{\Psi}(-\i z)$ is convex for real $z$, recalling from the discussion following Theorem \ref{censoredpsithrm} that $\stackrel{_\leadsto}{\mathbf{E}}[\stackrel{_\leadsto}{\xi}_1]<0$, we can deduce that
$
{\rm Re}(\Psi(-{\rm i}z)) >0
$
for 
${\rm Re}(z)\in(0, \alpha-1)$.

%It turns out that the potential density $\stackrel{_\leadsto}{u}$ can be explicitly identified by inverting \eqref{invertthisplease}.
\begin{theorem}[$\clubsuit$]\label{censoreddensity}
For $x>0$ we have
\begin{align*}
\stackrel{_\leadsto}{u}(x)=-&\frac{1}{\pi}\Gamma(1-\alpha)\frac{\sin(\pi \alpha \rho)}{\pi} \left[ 1-(1-{\rm e}^{-x})^{\alpha-1}\right]\\
&-\frac{1}{\pi}
\Gamma(1-\alpha)  \frac{\sin(\pi \alpha \hat \rho)}{\pi} {\rm e}^{-(\alpha -1)x}
\end{align*}
and for $x<0$ 
\begin{align*}
\stackrel{_\leadsto}{u}(x)=&-
\frac{1}{\pi}\Gamma(1-\alpha)  \frac{\sin(\pi \alpha \rho)}{\pi}\\
&-\frac{1}{\pi}\Gamma(1-\alpha)\frac{\sin(\pi \alpha \hat \rho)}{\pi} \left[1- (1-{\rm e}^{x})^{\alpha-1}\right]{\rm e}^{-(\alpha -1)x}. 
\end{align*}
In particular, 
\[
\stackrel{_\leadsto}{u}(0) =-\frac{1}{\pi} \Gamma(1-\alpha) \left( \frac{\sin(\pi \alpha \rho)}{\pi}+ \frac{\sin(\pi \alpha \hat \rho)}{\pi} \right).
\]
\end{theorem}
\begin{proof}[Proof \emph{($\clubsuit$)}]
A classical asymptotic result for the gamma function tells us that 
\begin{equation}\label{App1_gamma_asymptotics1}
\frac{\Gamma(z+a)}{\Gamma(z)}=z^a(1+o(1)),
\end{equation}
as $|z|\to +\infty$, uniformly in any sector $|{\textnormal{Arg}}(z)|<\pi-\epsilon$. Accordingly, we have that 
\begin{equation}\label{pot_density_proof_2}
\frac{1}{\stackrel{_\leadsto}{\Psi}(-\i z)}=z^{-\alpha} (1+o(1)), \;\;\; \im(z) \to \infty,
\end{equation}
which is valid uniformly  in any sector $|{\textnormal{Arg}}(z)|<\pi-\epsilon$. This and the fact that there are no poles along the vertical line $c+\i \r$, for $c \in (0,\alpha-1)$, allows us to invert \eqref{invertthisplease} via the integral
\begin{equation}\label{pot_density_proof_3}
\stackrel{_\leadsto}{u}(x)= \frac{1}{2\pi \i} \int_{c+\i \r} \frac{1}{\stackrel{_\leadsto}{\Psi}(-\i z)}{\rm e}^{-zx} \d z.
\end{equation}
We can proceed to give a concrete value to the above integral by appealing to a standard contour integration argument in connection with Cauchy's residue theory.

The function ${1}/{\stackrel{_\leadsto}{\Psi}(-\i z)}$ has simple poles at points 
$$
\{0, 1, 2, \dots \} \cup \{ \alpha-1, \alpha-2, \alpha-3, \dots\}. 
$$
%Set $c_n=n+1/2$ and $n\ge 1$. 
Suppose that $\gamma_R$ is the contour described in Fig. \ref{Sect6.6}. That is $\gamma_R = \{c+{\rm i}x: |x|\leq R\}\cup \{c+R{\rm e}^{{\rm i}\theta}: \theta\in(-\pi/2,\pi/2)\}$,
where we recall $c\in(0,\alpha-1)$.
 \begin{figure}[h!]
\begin{center}
\includegraphics[width = 5cm]{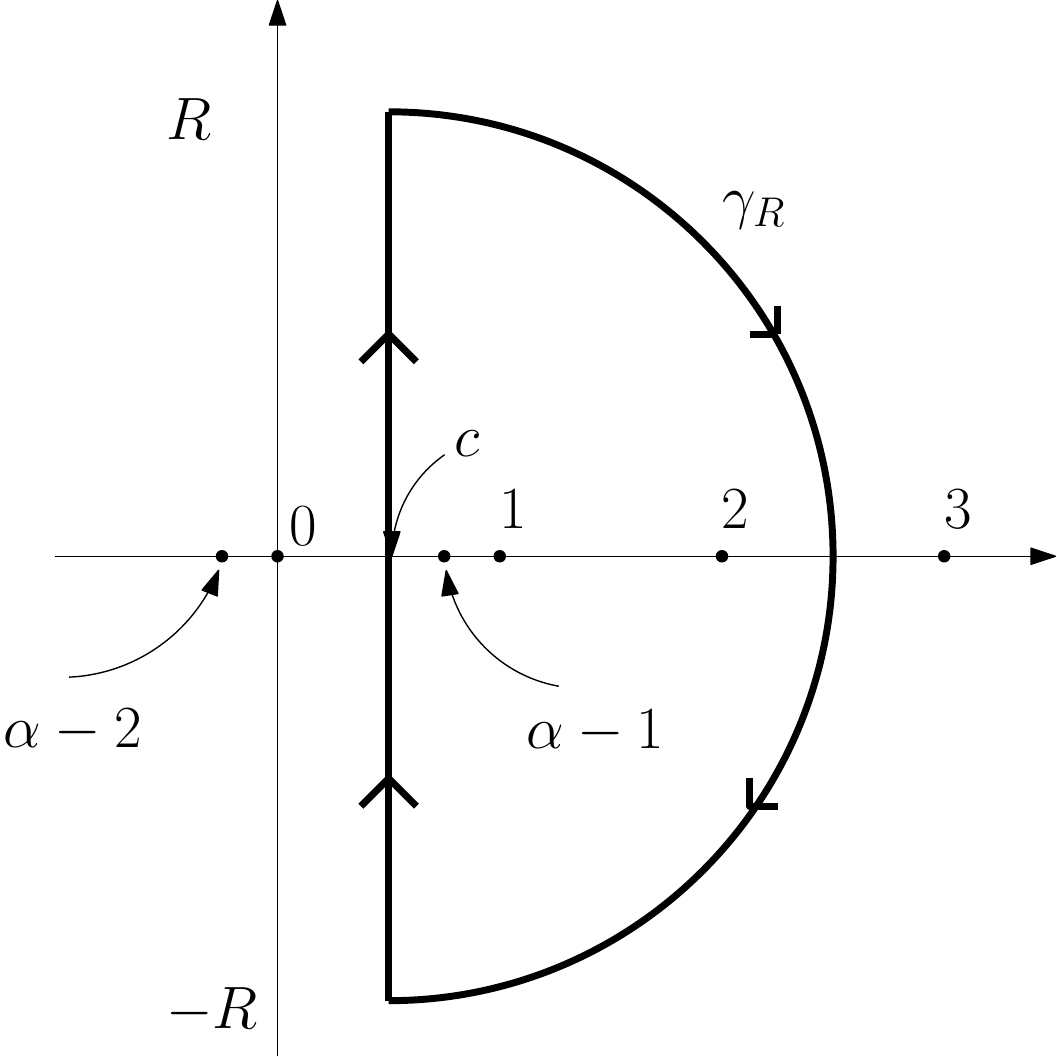}
\end{center}
\caption{The contour $\gamma_R$.}
\label{Sect6.6}
\end{figure}

Residue calculus gives us 
\begin{align}
& \frac{1}{2\pi \i} 
\int_{c+\i x: |x|\leq R} \frac{1}{\stackrel{_\leadsto}{\Psi}(-\i z)}{\rm e}^{-zx} \d z \notag\\
&\hspace{1cm}= -\frac{1}{2\pi \i} 
 \int_{c+R{\rm e}^{{\rm i}\theta}: \theta\in(-\pi/2, \pi/2)} \frac{1}{\stackrel{_\leadsto}{\Psi}(-\i z)}{\rm e}^{-zx} \d z\notag\\
 & \hspace{2cm}- {\rm Res}({1}/{\stackrel{_\leadsto}{\Psi}(-\i z)} :  z  = \alpha -1){\rm e}^{-(\alpha - 1)x}\notag\\
 &\hspace{3cm}-\sum_{1\leq k\leq \lfloor R \rfloor}{\rm Res}({1}/{\stackrel{_\leadsto}{\Psi}(-\i z)}:  z  = k){\rm e}^{-kx} .
 \label{takeRtoinf}
\end{align}

Now fix $x\geq 0$. The uniform estimate \eqref{pot_density_proof_2}, the positivity of $x$ and the length of the arc $\{c+R{\rm e}^{{\rm i}\theta}: \theta\in(-\pi/2, \pi/2)\}$ having length $\pi R$ allows us to estimate
\[
\left|\int_{c+R{\rm e}^{{\rm i}\theta}: \theta\in(-\pi/2, \pi/2)} \frac{1}{\stackrel{_\leadsto}{\Psi}(-\i z)}{\rm e}^{-zx} \d z\right|\leq CR^{-(\alpha-1)}
\]
for some constant $C>0$, 
and hence 
\[
\lim_{R\to\infty}\int_{c+R{\rm e}^{{\rm i}\theta}: \theta\in(-\pi/2, \pi/2)} \frac{1}{\stackrel{_\leadsto}{\Psi}(-\i z)}{\rm e}^{-zx} \d z =0.
\]
%Shifting the contour of integration $c+\i \r \to c_n+ \i \r$ in \eqref{pot_density_proof_2}
%and taking into account the residues of $f$ we obtain
%\begin{eqnarray}\label{pot_density_proof_3.1}
%\stackrel{_\leadsto}{u}(x)&=&-{\textnormal{Res}}({1}/{\stackrel{_\leadsto}{\Psi}(-\i z)} : z=\alpha-1){\rm e}^{-(\alpha-1)x}\notag\\
%&&-\sum\limits_{k=1}^n 
%{\textnormal{Res}}({1}/{\stackrel{_\leadsto}{\Psi}(-\i z)} :  z=k){\rm e}^{-k x}\notag\\
%&&+\frac{1}{2\pi \i} \int_{c_n+\i \r} \frac{1}{\stackrel{_\leadsto}{\Psi}(-\i z)} {\rm e}^{-zx} \d z
%\end{eqnarray}
%The reflection formula and Stirling's asymptotic formula  for the gamma function (see the Appendix) show that there exists a constant $C>0$ such that
%$$
%\left|\frac{1}{\stackrel{_\leadsto}{\Psi}(-\i z)}\right|<C |z|^{-\alpha}, \;\;\; z\in c_n+\i \r.
%$$
%Thus we can estimate the integral in \eqref{pot_density_proof_2} as follows
%$$
%\Big | \int_{c_n+\i \r}\frac{1}{\stackrel{_\leadsto}{\Psi}(-\i z)}{\rm e}^{-zx} \d z \Big |< C {\rm e}^{-c_n x}\int_{\r} |c_n+\i y|^{-\alpha} \d y. 
%$$
%It is clear that the right-hand side converges to zero as $c_n \to \infty$ as long as $x>0$.
 Together with \eqref{pot_density_proof_3}, we can use this convergence and take limits as $R\to\infty$ in \eqref{takeRtoinf} to conclude that
\begin{eqnarray*}%\label{pot_density_proof_4}
\stackrel{_\leadsto}{u}(x)&=-&{\textnormal{Res}}({1}/{\stackrel{_\leadsto}{\Psi}(-\i z)} : z=\alpha-1){\rm e}^{-(\alpha-1)x}\notag\\
&&-\sum\limits_{k=1}^{\infty} 
{\textnormal{Res}}({1}/{\stackrel{_\leadsto}{\Psi}(-\i z)} :  z=k){\rm e}^{-k x}.
\end{eqnarray*}
To compute the residues, we make straightforward use of the fact that 
${\rm Res}(\Gamma(z): z = -n) = (-1)^n/n!$, for $n\geq 0$. Hence, with the help of the binomial series identity, we finally obtain
\begin{align*}%\label{pot_density_proof_4}
\stackrel{_\leadsto}{u}(x)&=-\frac{1}{\pi} \sin(\pi \alpha \hat \rho)\Gamma(1-\alpha)  {\rm e}^{-(\alpha-1)x} +  
\frac{1}{\pi} \sin(\pi \alpha \rho)\sum\limits_{k=1}^{\infty}\frac{\Gamma(1-\alpha+k)}{k!}  {\rm e}^{-k x}
\\ \nonumber
&=-\frac{1}{\pi} \sin(\pi \alpha \hat \rho)\Gamma(1-\alpha)  {\rm e}^{-(\alpha-1)x} \\
& \hspace{2cm}+
\frac{1}{\pi} \sin(\pi \alpha \rho)\Gamma(1-\alpha) \left[(1-{\rm e}^{-x})^{\alpha-1}-1 \right],
\end{align*}
which is valid for $x>0$. 

The proof in the case $x<0$ is identical, except that we need to shift the arc in the contour $\gamma_R$ to extend into the negative part of the complex plane. The details are left to the reader, however, it is heuristically clear that one will end up with 
\begin{align*}
\stackrel{_\leadsto}{u}(x)&
=-{\textnormal{Res}}({1}/{\stackrel{_\leadsto}{\Psi}(-\i z)} : z=0)
-\sum\limits_{k=2}^{\infty} 
{\textnormal{Res}}({1}/{\stackrel{_\leadsto}{\Psi}(-\i z)} :  z=\alpha - k){\rm e}^{-(\alpha - k) x}\\
&=
-\frac{1}{\pi}\sin(\pi\alpha\rho)\Gamma(1-\alpha)
+\frac{1}{\pi}\sin(\pi\alpha\hat\rho){\rm e}^{-(\alpha - 1) x}\sum_{j =1}
\frac{\Gamma(1-\alpha+j)}{j!}{\rm e}^{jx },
\end{align*}
which agrees with the expression given in the statement of the theorem, where, again, we use the binomial series expansion.
\end{proof}

The consequence of being able to identify the above potential density explicitly is that we can now give an explicit identity for the two point hitting problem. Without loss of generality, we can always reduce a general choice of $x$  to the case  $x>-1$. Indeed,
if $x<-1$, then 
\[
\mathbb{P}_{x}(\tau^{\{1\}}<\tau^{\{-1\}}) = \hat{\mathbb{P}}_{1-x}(\tau^{\{-1\}}<\tau^{\{1\}}) = 1- \hat{\mathbb{P}}_{1-x}(\tau^{\{1\}}<\tau^{\{-1\}}),
\]
where $(X,\hat{\mathbb{P}})$ is equal in law to  $-X$.
The following result is originally due to \cite{G}, however we give a completely new proof here. 
\begin{theorem}[$\heartsuit$]\label{twopoint}
Suppose that $x>-1$, then 
\begin{align*}
&\mathbb{P}_x(\tau^{\{1\}}<\tau^{\{-1\}})\\
&=  \left\{
\begin{array}{ll}
\dfrac{2^{\alpha-1}\sin(\pi\rho\alpha) - |x-1|^{\alpha-1} \sin(\pi\hat\rho\alpha) +(x+1)^{\alpha-1}\sin(\pi\hat\rho\alpha)}{2^{\alpha -1}(\sin(\pi\rho\alpha)+ \sin(\pi\hat\rho\alpha))} & x>1\\
&\\
\dfrac{2^{\alpha-1}\sin(\pi\rho\alpha) - |x-1|^{\alpha-1} \sin(\pi\rho\alpha) +(x+1)^{\alpha-1}\sin(\pi\hat\rho\alpha)}{2^{\alpha -1}(\sin(\pi\rho\alpha)+ \sin(\pi\hat\rho\alpha))}
&-1< x<1.
\end{array}
\right.
\end{align*}
%where $I(b) = \mathbf{1}_{(x>b)}\sin(\pi\hat\rho\alpha) + \mathbf{1}_{(x<b)}\sin(\pi\rho\alpha)$.
\end{theorem}
\begin{proof}[Proof \emph{($\clubsuit$)}]
 When $0<z<b$, note that $-\log(z/b)>0$. We therefore use  the first of the two expressions for $\stackrel{_\leadsto}{u}(z)$ in Theorem \ref{censoreddensity} for the identity (\ref{workoutexplicitly}). We have 
\begin{eqnarray}
\lefteqn{\mathbb{P}_z(\tau^{\{b\}}<\tau^{\{0\}})}\notag\\
&&=\frac{\sin(\pi \alpha \rho) \left[ 1-(1-z/b)^{\alpha-1}\right]+
\sin(\pi \alpha \hat \rho)  (z/b)^{\alpha -1}}{(\sin(\pi\rho\alpha)+ \sin(\pi\hat\rho\alpha))}\notag\\
 &&=  \frac{b^{\alpha-1}\sin(\pi\rho\alpha) - (b-z)^{\alpha-1}  \sin(\pi\rho\alpha) +z^{\alpha-1}\sin(\pi\hat\rho\alpha)}{b^{\alpha -1}(\sin(\pi\rho\alpha)+ \sin(\pi\hat\rho\alpha))},
 \label{z<b}
\end{eqnarray}
as required. When $z>b>0$, we have $-\log (z/b)<0$ and we use the second of the two expressions in Theorem \ref{censoreddensity} for the identity (\ref{workoutexplicitly}). In that case, we have 
\begin{eqnarray}
\lefteqn{\mathbb{P}_z(\tau^{\{b\}}<\tau^{\{0\}})}\notag\\
&&= \frac{\sin(\pi \alpha \hat \rho) \left[1- (1-b/z)^{\alpha-1}\right](z/b)^{\alpha -1}+
\sin(\pi \alpha \rho)}{(\sin(\pi\rho\alpha)+ \sin(\pi\hat\rho\alpha))}\notag\\
&&= \frac{b^{\alpha-1}\sin(\pi\rho\alpha) - (z-b)^{\alpha-1} \sin(\pi\hat\rho\alpha)  +z^{\alpha-1}\sin(\pi\hat\rho\alpha)}{b^{\alpha -1}(\sin(\pi\rho\alpha)+ \sin(\pi\hat\rho\alpha))},
\label{z>b}
\end{eqnarray}
and the proof is complete once we set $z=1+x$ and $b =2$.
\end{proof}

\subsection{Resolvent with killing at the origin}
There is one quite nice conclusion we can draw from the two-point hitting problem in relation to the case of hitting a single point. 
Suppose we write $\mathbb{P}^{\{0\}}$ for the law of $X$ killed on hitting the origin and pre-emptively write
\[
u_{\{0\}}(x,b)\d b = \int_0^\infty \mathbb{P}_x(X_t \in \d b, \, t< \tau^{\{0\}})\d t = \int_0^\infty \mathbb{P}^{\{0\}}_x(X_t \in \d b)\d t .
\]
Without loss of generality, assume that $x,b>0$. We have 
%Appealing again to Corollary II.18  of \cite{bertoin}, we note that 
\[
\mathbb{P}^{\{0\}}_x(\tau^{\{b\}}<\infty) =\mathbb{P}_z(\tau^{\{b\}}<\tau^{\{0\}})= \frac{u_{\{0\}}(x, b)}{u_{\{0\}}(b,b)},
\]
where in the second equality we have appealed to a classical identity for potential densities, see for example Chapter V of \cite{BG}. The left-hand side above has been computed in \eqref{z<b} and \eqref{z>b} giving us
\[
u_{\{0\}}(x, b) = c_{\{0\}}\left(b^{\alpha-1}\sin(\pi\rho\alpha) - |x-b|^{\alpha-1} s(b-x)  +x^{\alpha-1}\sin(\pi\hat\rho\alpha)\right)
\]
for some constant $c_{\{0\}}\in(0,\infty)$, where
$
s(x) = \sin(\pi\rho\alpha)\mathbf{1}_{(x>0)} +  \sin(\pi\hat\rho\alpha)\mathbf{1}_{(x<0)} .
$

It seems difficult to  pin down the constant $c_{\{0\}}\in(0,\infty)$, which could in principle depend on $b$, on account of the fact that  $u_{\{0\}}(x, b)\d b$ is an example of a potential measure of a transient process, namely $(X_{t\wedge\tau^{\{0\}}}, t\geq 0)$, that has infinite total mass. Indeed, note that 
\begin{equation}
\int_\mathbb{R}u_{\{0\}}(x,b)\d b =  \int_0^\infty \mathbb{P}_x(X_t \in \d b, \, t< \tau^{\{0\}})\d t  = \mathbb{E}_x[\tau^{\{0\}}],
\label{leftvsright}
\end{equation}
where the expectation on the right-hand side is known to be infinite; see for example \cite{KKPW}.

That said, if we put together some of the ingredients we have examined in the preceding computations in the right way, we can deduce the following precise result which does not seem to be known in the existing literature.

\begin{theorem}[$\clubsuit$]\label{resolvenbefore0} The potential with killing at the origin is absolutely continuous such that, for $x, y$ in $\mathbb{R}$ and distinct from the origin,
\[
u^{\{0\}}(x, y) =-\frac{1}{\pi^2}\Gamma(1-\alpha)
\left(|y|^{\alpha-1}s(y)
 - |y-x|^{\alpha-1} s(y-x) +|x|^{\alpha-1}s(-x)
 \right),
\]
 where
$
s(x) = \sin(\pi\alpha\rho)\mathbf{1}_{(x\geq 0)} +  \sin(\pi\alpha\hat\rho)\mathbf{1}_{(x<0)} .
$
\end{theorem}
\begin{proof}[Proof \emph{($\clubsuit$)}]
Suppose that $f$ is a bounded measurable function in $\mathbb{R}$.
We are interested in the potential measure $U^{\{0\}}(x, \d y)$ which satisfies
\[
\int_{\mathbb{R}} f(y) U^{\{0\}}(x, \d y) = \mathbb{E}_x\left[\int_0^\infty f(X_t)\mathbf{1}_{(t<\tau^{\{0\}})}\, \d t\right].
% = \mathbb{E}_x\left[\int_0^\infty f(Z_t)\\d t\right],
\]
In particular, recalling the change of measure \eqref{updownCOM7}, we have that 
\[
\int_{\mathbb{R}} f(y)\frac{h(y)}{h(x)} U^{\{0\}}(x, \d y) = \mathbb{E}^\circ_x\left[\int_0^\infty f(X_t)\,\d t\right]
\]

Let us momentarily focus our attention on the setting that $x,y>0$ in $U^{\{0\}}(x, \d y)$. In that case, we can write 
\[
\int_{[0,\infty)} f(y)\frac{h(y)}{h(x)} U^{\{0\}}(x, \d y) = \mathbb{E}^\circ_x\left[\int_0^\infty f(X_t)\mathbf{1}_{(X_t>0)}\, \d t\right]\\
=\mathbb{E}^\circ_x\left[\int_0^\infty f(\stackrel{_{\leadsto}}{Z^\circ_t}) \, \d t\right],
\]
where $\stackrel{_{\leadsto}}{Z^\circ}=(\stackrel{_{\leadsto}}{Z^\circ_t}, t\geq 0)$ is the pssMp which is derived by censoring away the negative sections of path of $(X,\mathbb{P}^\circ_x)$, $x\in\mathbb{R}$ in the spirit of what we have already seen for stable processes, cf. Section \ref{CSPsection}. Suppose that we denote the L\'evy process that underlies $Z$ by $\cenxicirc$, with probabilities $\stackrel{_{\leadsto}}{\mathbf{P}^\circ_{x}}$, $x\in\mathbb{R}$. Taking account of the time change in the Lamperti transform \eqref{13lampertirep}, we thus have on the one hand that 
\begin{align}
\int_{[0,\infty)} f(y) \frac{h(y)}{h(x)}U^{\{0\}}(x, \d y) 
&=\stackrel{_{\leadsto}}{\mathbf{P}^\circ}_{\hspace{-0.1cm}\log x}\left[\int_0^\infty f({\rm e}^{\cenxicirc_t}) {\rm e}^{\alpha\cenxicirc_t}\d t\right]\notag\\
&= \int_{\mathbb{R}}f({\rm e}^z ){\rm e}^{\alpha z} \stackrel{_{\leadsto}}{u}{ }^{\hspace{-0.1cm}\circ}((\log x) -z)\,\d z\notag\\
&=\int_{\mathbb{R}}f(y )y^{\alpha - 1}\stackrel{_{\leadsto}}{u}{ }^{\hspace{-0.1cm}\circ}(\log (x/ y))\, \d y,
\label{U0ucirc}
\end{align}
Where we have pre-emptively assumed that $\cenxicirc$ has a potential density, which we have denoted by $\stackrel{_{\leadsto}}{u}{ }^{\hspace{-0.1cm}\circ}$. This is  a reasonable assumption for the following reasons. 

Taking account of the change of measure \eqref{updownCOM7}, noting that the time change pertaining to the censoring of $(X,\mathbb{P}^\circ_x)$, $x\neq 0$, results in sampling this process at a sequence of stopping times, we can also see that \eqref{updownCOM7} described the change of measure between the censored process $\stackrel{_{\leadsto}}{Z^\circ}$  and the censored stable process discussed in \eqref{CSPsection}. In effect, this is tantamount to a Doob $h$-transform between the two positive-valued processes with $h$ function taking the form $h(x) = x^{\alpha-1}$, $x\geq 0$. In terms of the underlying L\'evy processes $\cenxi$ (for the censored stable process) and $\cenxicirc$ (for $\stackrel{_{\leadsto}}{Z^\circ}$), this Doob $h$-transform acts as an Esscher transform. In particular, we have 
\[
\frac{\d\!\stackrel{_{\leadsto}}{\mathbf{P}^\circ}   }
{\d\!\stackrel{_{\leadsto}}{\mathbf{P}}  }
\Bigg|_{\sigma(\cenxi_s: s\leq t)} = {\rm e}^{(\alpha-1)\cenxi_{t}}, \qquad t\geq0.
\]
It is thus straightforward to show that 
\begin{equation}
\stackrel{_{\leadsto}}{u}{ }^{\hspace{-0.1cm}\circ}(x) = \stackrel{_{\leadsto}}{u}(x) {\rm e}^{(\alpha -1)x}, \qquad x\in\mathbb{R}.
\label{nastyu}
\end{equation}
Consolidating \eqref{U0ucirc} and \eqref{nastyu}, noting in particular from \eqref{updownCOM7} that $h(x) = s(-x)|x|^{\alpha-1}$, where $s(x) = \sin(\pi\alpha\rho)\mathbf{1}_{(x\geq 0)} + \sin(\pi\alpha\rhohat)\mathbf{1}_{(x<0)}$, we thus conclude that 
\begin{equation}
U^{\{0\}}(x, \d y)  = y^{\alpha - 1}\stackrel{_{\leadsto}}{u}(\log (y/x))\,\d y, \qquad x,y>0,
\label{positives}
\end{equation}
where we recall that $\stackrel{_{\leadsto}}{u}$ has been computed explicitly in Theorem \ref{censoreddensity}.

Bringing across the specific form of $\stackrel{_{\leadsto}}{u}$ from Theorem \ref{censoreddensity}, we can now read off that, for $x,y>0$, $U^{\{0\}}(x, \d y)$ is absolutely continuous with density, $u^{\{0\}}(x,y)$, taking the form 
\begin{align*}
&u^{\{0\}}(x,y)\\
&=-\frac{1}{\pi^2}\Gamma(1-\alpha)
\left\{
\begin{array}{ll}
\sin(\pi \alpha \rho) y^{\alpha-1}-\sin(\pi \alpha \rho) (y-x)^{\alpha-1}+  \sin(\pi \alpha \hat \rho) x^{\alpha -1}, & y>x,\\
&\\
\sin(\pi \alpha \rho)y^{\alpha-1}+\sin(\pi \alpha \hat \rho)x^{\alpha-1}- \sin(\pi \alpha \hat \rho)(x-y)^{\alpha-1},&y<x,
\end{array}
\right.
\end{align*}
which is consistent with the statement of the theorem.

Note in particular that, for $y>0$,
\[
u^{\{0\}}(y,y): = \lim_{x\to y} u^{\{0\}}(x,y)= -\frac{1}{\pi^2}\Gamma(1-\alpha)y^{\alpha-1} (\sin(\pi \alpha \rho) +\sin(\pi \alpha \hat \rho) ).
\]
We can use this limit to deal with the case that $x<0<y$. Indeed, the strong Markov property gives us
\[
U^{\{0\}}(x, \d y) =\mathbb{P}_x(\tau^{\{y\}}<\tau^{\{0\}})u^{\{0\}}(y, y)\d y, \qquad y>0.
\]
Hence recalling the expression for $\mathbb{P}_x(\tau^{\{y\}}<\tau^{\{0\}})$ in Theorem \ref{twopoint}, we recover the required identity in the regime that $x<0<y$.

By working instead with $-X$ (or equivalently censoring out the positive parts of the path of $(X,\mathbb{P}^\circ_x)$, $x\in\mathbb{R}$) we easily conclude that the same identities hold when $x,y<0$ and $x>0>y$ simply by interchanging the roles of  $\rho$ and $\hat\rho$.
\end{proof}

\part{Higher dimensional results on $\mathbb{B}_d$}

\setcounter{section}{0}

Before handling any of the promised exit problems, let us start in this and the next section by considering two remarkably simple but effective transformations which invert space through a given sphere in $\mathbb{R}^d$. We will be particularly   interested in   how these spatial transformations can be used to  manipulate integrals  that take the form 
\begin{equation}
U\mu(x) = \int_{D}|x-y|^{\alpha -d}\mu(\d y), \qquad x\in \mathbb{R}^d,
\label{genericpotential}
\end{equation}
(traditionally known as Riesz potentials), 
where  $\mu$ is a finite measure on $D\subseteq\mathbb{R}^d$ and $\alpha\in(0,2).$

\section{Sphere inversions}\label{SI}

In this section, we present  two fundamental geometrical inversions of Euclidian space that are prevalent throughout classical Newtonian and Riesz potential analysis. See for example the original work of \cite{Riesz} or the classic texts of \cite{PortStonebook, Landkof, BH}.
\subsection{Inversion of a sphere through another sphere}
Fix a point $b\in\mathbb{R}^d$ and a value $r>0$. A homeomorphism of $\mathbb{R}^d\backslash\{b\}$  defined by 
\begin{equation}
x^* = b + \frac{r^2}{|x-b|^2}(x-b),
\label{sphereinversion}
\end{equation}
is called an {\it inversion through the sphere} $\mathbb{S}_{d-1}(b,r): = \{x\in\mathbb{R}^d: |x-b| = r\}$. (Note that we reserve the special notation $\mathbb{S}_{d-1} $ to mean $ \mathbb{S}_{d-1}(0,1)$.) Amongst the many properties of this inversion, which we shall now discuss, the most important is that the exterior of $\mathbb{S}_{d-1}(b,r)$ maps to its interior and vice versa; see Fig. \ref{sphereinversionfig}.
\begin{figure}[h!]
\begin{center}
\includegraphics[width = 4cm]{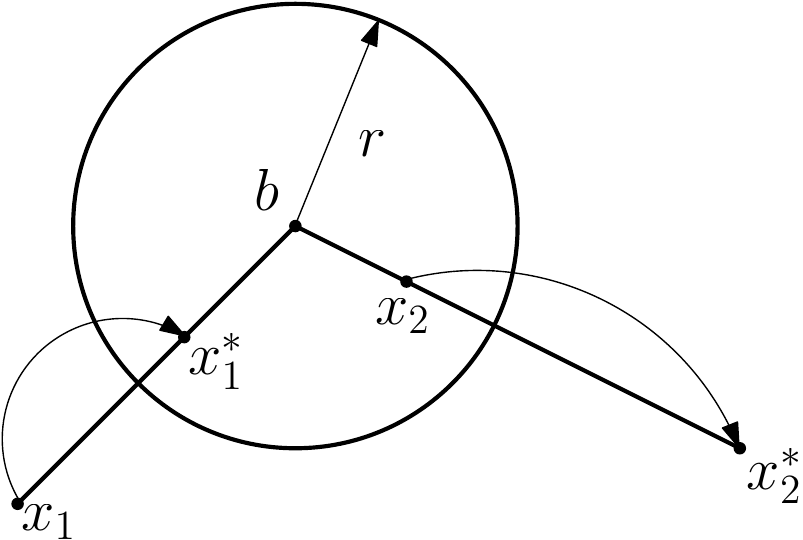}
\end{center}
\caption{Inversion relative to the sphere $\mathbb{S}_{d-1}(b,r)$.}
\label{sphereinversionfig}
\end{figure}

Straightforward algebra also tells us that
\begin{equation}
\label{spheretosphere}
r^2 = |x^*- b||x-b|,
\end{equation}
which, 
in turn, also gives us that  $(x^*)^* = x$, for $x\in\mathbb{R}^d\backslash\{b\}$ and in particular, 
\begin{equation}
\label{selfinverse}
x = b + \frac{r^2}{|x^*-b|^2} (x^*-b).
\end{equation}
Moreover, straightforward algebra using \eqref{sphereinversion} and \eqref{selfinverse} gives us, for $x,y\in\mathbb{R}\backslash\{b\}$,
\begin{equation}
|x^*-y^*| = \sqrt{(x^*-y^*)\cdot(x^*-y^*)}=
 \frac{r^2|x-y|}{|x-b||y-b|}.
\label{difference}
\end{equation}

Another very important fact about inversion through the sphere $\mathbb{S}_{d-1}(b,r)$ is that a sphere which does not pass through or encircle $b$ will always map to another sphere. To see why, suppose that we consider the image of any sphere $\mathbb{S}_{d-1}(c, R)$, for $c\in\mathbb{R}^d$ and $R>0$, for which $|c-b|>R$, and denote its image under inversion through $\mathbb{S}_{d-1}(b,r)$ by $\mathbb{S}^*_d(c,R)$. We can write  
$
\mathbb{S}_{d-1}(c, R) = \{x\in\mathbb{R}^d: |(x-b)-(c-b)|^2 = R^2\},
$
which can otherwise be written as $x\in\mathbb{R}^d$ such that 
\[
|x-b|^2- 2(x-b)\cdot(c-b) + |c-b|^2 = R^2.
\]
 \begin{figure}[h!]
\begin{center}
\includegraphics[width = 4cm]{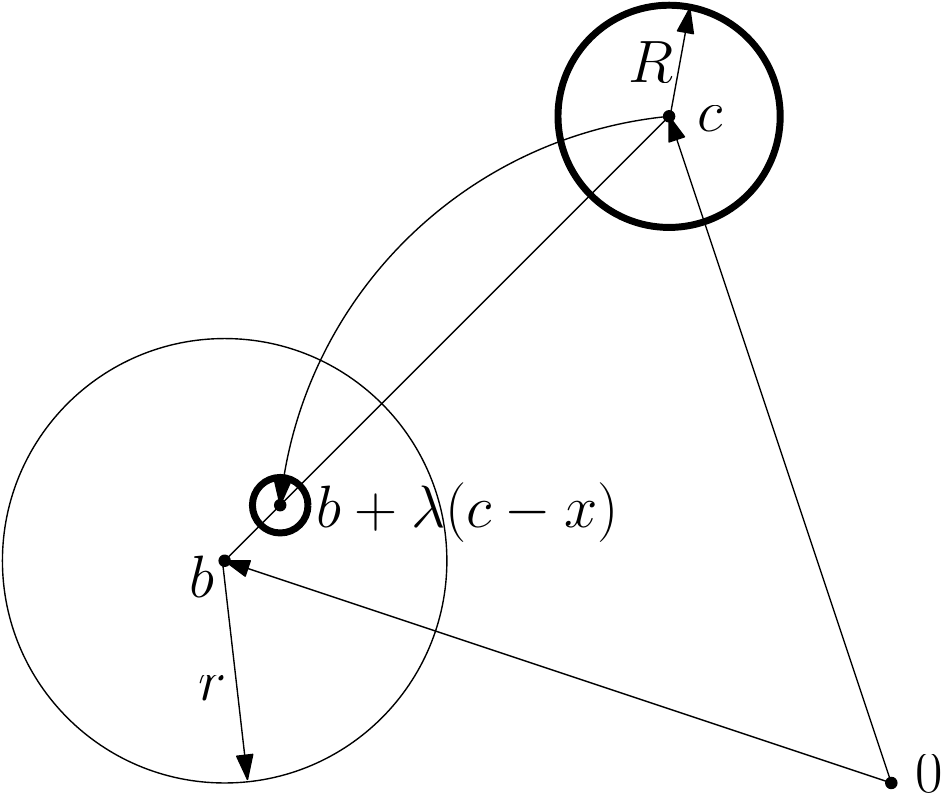}
\end{center}
\caption{The sphere $\mathbb{S}_{d-1}(c,R)$ maps to the sphere $\mathbb{S}^*_d(c,R)$ under inversion through $\mathbb{S}_{d-1}(b,r)$.}
\label{outsideinfig}
\end{figure}

 From \eqref{spheretosphere} and \eqref{selfinverse}, after a little algebra, for  $x\in \mathbb{S}_{d-1}(c, R)$, 
\[
|x^*-b|^{2}- \lambda(x^*-b)\cdot(c-b) + \lambda^2|c-b|^2 = \eta^2, %\frac{r^4 (3|c-b|^2 + R^2)}{(|c-b|^2- R^2)^2} ,
\]
where $\lambda = r^2/(|c-b|^2- R^2)$ and $\eta^2 = r^4 R^2/(|c-b|^2- R^2)^2$. That is to say, 
$
\mathbb{S}^*_d(c, R)=\{x^*\in\mathbb{R}^d: |(x^*-(b+\lambda(c-b))|^2 = \eta^2\}
$
 so that $\mathbb{S}^*_d(c, R)$ is mapped to another sphere. 
 
   \begin{figure}[h!]
\begin{center}
\includegraphics[width = 5cm]{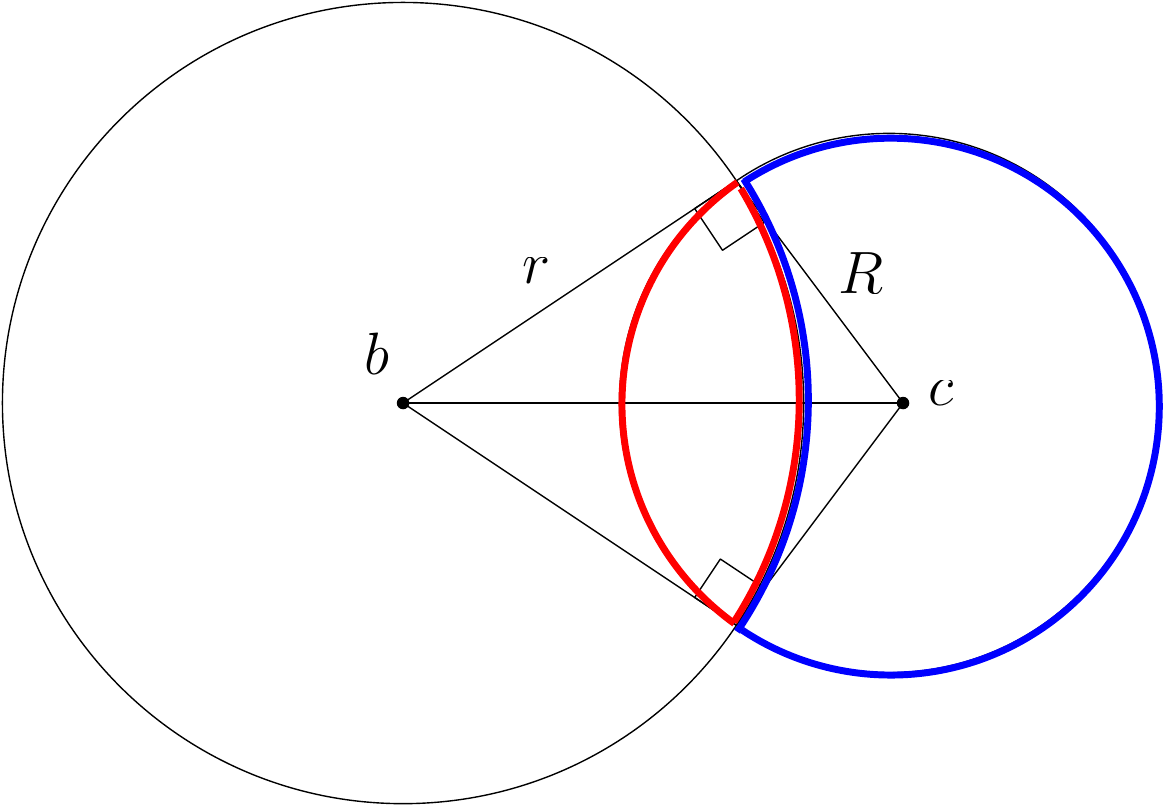}
\end{center}
\caption{The sphere $\mathbb{S}_{d-1}(c,R)$ maps to itself under inversion through $\mathbb{S}_{d-1}(b,r)$ provided the former is orthogonal to the latter, which is equivalent to $r^2+ R^2 = |c-b|^2$. In particular, the area contained in the blue segment (the crescent containing $c$) is mapped to the area in the red segment (the remainder of the ball centred at $c$) and vice versa.}
\label{spheretoitselffig}
\end{figure}
 We note in particular that $\mathbb{S}^*_d(c,R)=\mathbb{S}_{d-1}(c,R)$ if and only if $\lambda = 1$, in other words, $r^2+ R^2 = |c-b|^2$. This is equivalent to requiring that the spheres $\mathbb{S}_{d-1}(c,R)$ and $\mathbb{S}_{d-1}(b,r)$ are orthogonal, and therefore necessarily overlapping. What is additionally interesting about this choice of $\mathbb{S}_{d-1}(c,R)$ is that its interior maps to its interior and its exterior to its exterior.

\subsection{Inversions through another sphere with reflection}\label{SIR}
 A variant of the transformation \eqref{sphereinversion} takes the form 
\begin{equation}
x^\diamond =  b - \frac{r^2}{|x-b|^2}(x-b), 
\label{diamond}
\end{equation}
for a fixed $b\in\mathbb{R}^d$ and $x\in\mathbb{R}^d\backslash\{b\}$,
which similarly has the self-inverse property \eqref{selfinverse}. It is also quite straightforward to show that 
\begin{equation}
\label{spheretospherediamond}
r^2 = |x^\diamond- b||x-b|
\end{equation}
and 
\begin{equation}
|x^\diamond-y^\diamond| =
 \frac{r^2|x-y|}{|x-b||y-b|}
\label{differencediamond}
\end{equation}
 still hold in the spirit of \eqref{spheretosphere} and \eqref{difference}, respectively.

 \begin{figure}[h!]
\begin{center}
\includegraphics[width = 5cm]{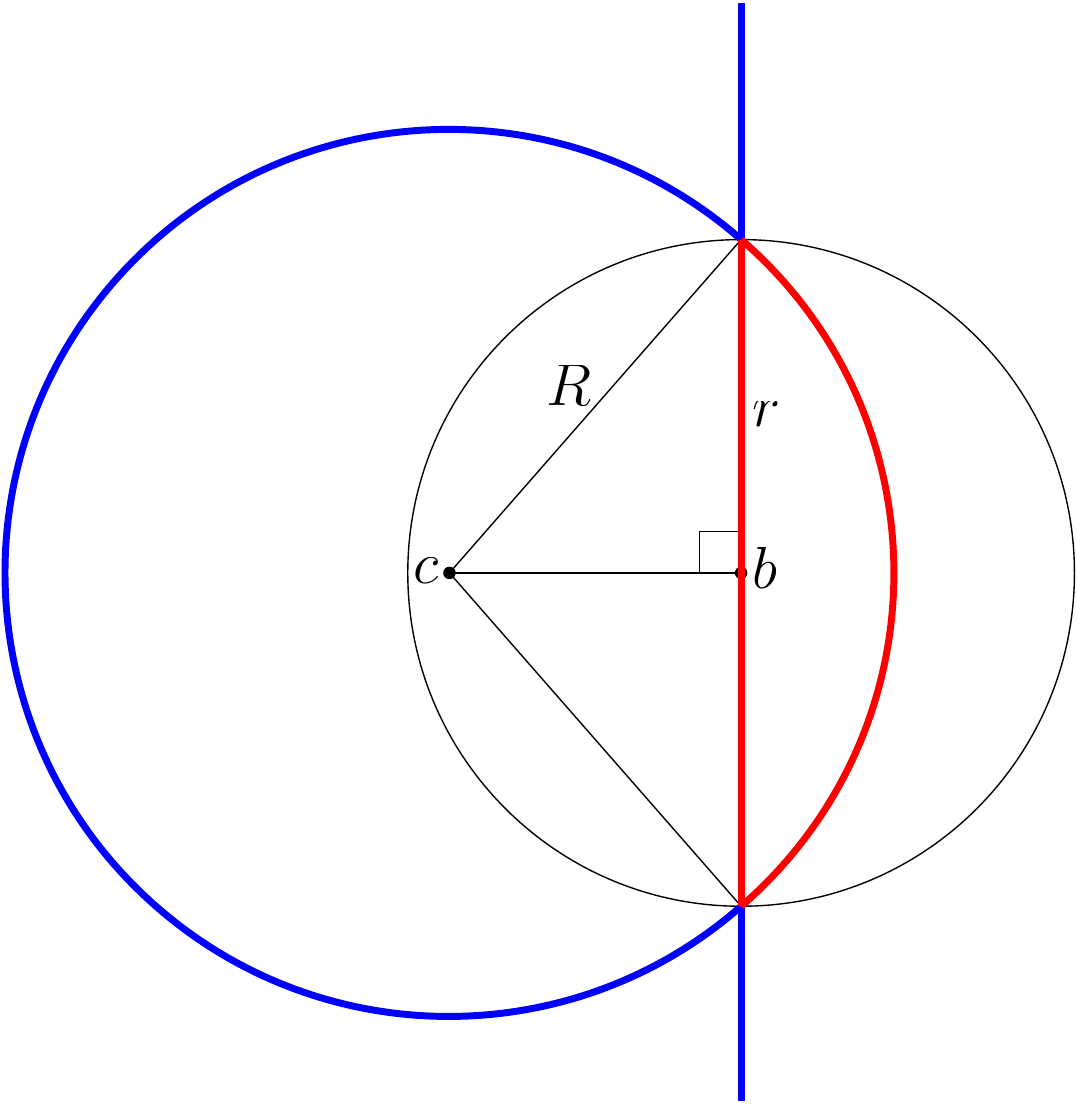}
\end{center}
\caption{The sphere $\mathbb{S}_{d-1}(c,R)$ maps to itself through $\mathbb{S}_{d-1}(b,r)$ via \eqref{diamond}  providing $|c-b|^2 + r^2 = R^2$. However, this time, the exterior of the sphere $\mathbb{S}_{d-1}(c,R)$ maps to the interior of the sphere $\mathbb{S}_{d-1}(c,R)$ and vice versa. For example, the region in the exterior of $\mathbb{S}_{d-1}(c,R)$ contained by blue boundary  (that is the half-space to the left of the vertical line excluding the intersection with the larger ball) maps to the portion of the interior of $\mathbb{S}_{d-1}(c,R)$ contained by the red boundary (that is the intersection of the larger ball with the half-space to the right of the vertical line).}
\label{closespheres}
\end{figure}
Intuitively speaking $x^\diamond$ performs the same sphere inversion as $x^*$, albeit with the additional feature that there is pointwise reflection about $b$. As such, any sphere $\mathbb{S}_{d-1}(c, R)$ will map to another sphere, say $\mathbb{S}_{d-1}^\diamond(c,R)$ so long as $|c-b|<R$. We are again interested in  choices of $c$ and $R$ such that  $\mathbb{S}_{d-1}^\diamond(c,R)=\mathbb{S}_{d-1}(c,R)$. This turns out to be possible so long as $R^2 =r^2+ |c-b|^2$. Moreover, in that case, the interior of $\mathbb{S}_{d-1}(c,R)$ maps to its exterior and its exterior to its interior; see Fig. \ref{closespheres}.

To see how this is possible, we need to prove a new identity for $x^\diamond$. We claim that 
\begin{equation}
|x^\diamond - c|^2 - R^2 = \frac{|x^\diamond-b|^2}{r}(R^2 - |x-c|^2), \qquad x\in\mathbb{R}^d
\label{newidentity}
\end{equation}
Indeed, 
recalling that $|x^\diamond - b||x-b| = r^2$, we can write %$|c-d|^2 + r^2 = R^2$,
\[
x = b + \frac{(x-b)}{|x-b|}|x-b|\quad \text{ and }\quad x^\diamond =  b- |x^\diamond-b|\frac{(x-b)}{|x-b|}.
\]
Hence, as $|b-c|^2 + r^2 = R^2$ and using again that $|x^\diamond - b||x-b| = r^2$, we have
\begin{align}
|x^\diamond - c|^2 - R^2
&=|(x^\diamond -b)+(b- c)|^2 - R^2\notag\\
 &=|x^\diamond - b|^2   -2|x^\diamond-b|\frac{(x-b)\cdot(b-c)}{|x-b|}-r^2\notag\\
 &=\frac{|x^\diamond - b|^2}{r^2}\left(r^2  -2(x-b)\cdot(b-c)- |x-b|^2\right)  \notag\\
 &=\frac{|x^\diamond - b|^2}{r^2}\left(R^2 - |b-c|^2 -2|x-b|\frac{(x-b)}{|x-b|}\cdot(b-c)- |x-b|^2\right)\notag\\
 &=\frac{|x^\diamond - b|^2}{r^2}\left(R^2 - |x-c|^2\right),
\label{quadraticdistance}
\end{align}
which proves \eqref{newidentity}.

It is now immediately apparent that $|x^\diamond - c|^2< R^2$ if and only if $|x-c|^2> R^2$, and $|x^\diamond - c|^2= R^2$ if and only if $|x-c|^2=R^2$ which is to say that $\mathbb{S}_{d-1}^\diamond(c,R)=\mathbb{S}_{d-1}(c,R)$ and that the interior of $\mathbb{S}_{d-1}(c,R)$ maps to its exterior and its exterior maps to its interior as claimed.

%Let us now return to give the promised understanding of how the sphere inversion described above affects potentials of the kind in \eqref{genericpotential}. Although we will not use the following lemma directly, it exemplifies the nature of the computations that we will seek to perform in establishing the forthcoming exit distributions. 
%
%
%\begin{lemma}\label{KT} Suppose that $\mu$ is a finite measure on $D\subseteq\mathbb{R}^d$ and, for a given $b\in\mathbb{R}^d$ and $r>0$, using the inversion through the sphere $\mathbb{S}_{d-1}(b,r)$ given by \eqref{sphereinversion}, we define, for each $\mu$-measurable $A$,
%\[
%\mu^*(A^*) = \int_A r^{d-\alpha}|z-b|^{\alpha -d}\mu(\d z),
%\]
%where $A^* = \{x^*: x\in A\}$ is the image of the set $A$ through the inversion and we assume that $\mu(\{b\}) = 0$. Then $\mu^*$ satisfies the so-called Kelvin transform:
%\[
%U\mu^* (x^*) = {r^{\alpha- d}|x-b|^{d-\alpha}}U\mu(x), \qquad x\in\mathbb{R}^d.
%\]
%\end{lemma}
%\begin{proof}Appealing to \eqref{difference} and the fact that $(z^*)^*= z$, we have 
%\begin{align*}
%U\mu^*(x^*) &= \int_{D^*}|x^*- z^*|^{\alpha - d}\mu^*(\d z^*)\\
%&=r^{\alpha- d}\int_{D}\frac{|x- z|^{\alpha - d}}{|x-b|^{\alpha - d} |z-b|^{\alpha-d}}|z-b|^{\alpha -d}\mu(\d z)\\
%&=r^{\alpha- d}|x-b|^{d-\alpha}\int_{D}|x- z|^{\alpha - d}\mu(\d z)\\
%&=r^{\alpha- d}|x-b|^{d-\alpha}U\mu(x)
%\end{align*}
%as required.
%\end{proof}

%With the Kelvin transform now in hand, we are ready to consider the first and most obvious exit problem for an isotropic stable process.

\section{First hitting of the unit sphere}\label{spherehit}

Let us turn to the first of two natural first passage problems that we will consider in this chapter. This concerns the distribution of the position of $X$ on first hitting of the sphere $\mathbb{S}_{d-1}= \{x\in\mathbb{R}^d: |x| = 1\}$. To this end, let us introduce the notation 
\[
\tau^\odot = \inf\{t>0: |X_t| = 1\}.
\]

\begin{theorem}[$\heartsuit$]\label{hitasphere} Define the function 
\[
h^{\odot}(x, y)=\frac{\Gamma\left(\frac{\alpha+d}{2}-1\right)\Gamma\left(\frac{\alpha}{2}\right)}{\Gamma \left(\frac{d}{2}\right)\Gamma(\alpha-1)}
\frac{||x|^2 - 1|^{\alpha - 1}}{|x-y|^{\alpha + d-2}}
\]
for $|x|\neq 1$, $|y| = 1$. 
Then, if $\alpha \in(1,2)$, 
\[
\mathbb{P}_x(X_{\tau^\odot}\in \d y) = h^{\odot}(x, y)\sigma_1(\d y) \mathbf{1}_{(|x|\neq 1)} + \delta_{x}(\d y)\mathbf{1}_{(|x| = 1)}, \qquad |y| =1,
\]
where $\sigma_1(\d y)$ is the surface measure on $\mathbb{S}_{d-1}$, normalised to have unit total mass.
Otherwise, if $\alpha \in(0,1]$, $\mathbb{P}_x(\tau^\odot=\infty)=1$, for all $|x|\neq 1$.
\end{theorem}
This theorem  is due to \cite{Port69} and we largely follow his steps to its proof, with some adaptations to our updated point of view. 
Before proving it, we need to address a number of preliminary results first.  

\subsection{Probability of ever hitting $\mathbb{S}_{d-1}$}

As alluded to above, the next result is due to \cite{Port69}, albeit that we have opted to write the computed probability in terms of the hypergeometric function, ${}_2F_1$, rather than Port's original choice, the Legendre function of the first kind. This is a consequence of our choice to re-prove his result using the perspective of a stable processes as a self-similar Markov process.

\begin{theorem}[$\diamondsuit$]\label{everhit}
For $|x|>0$, if $\alpha\in(1,2)$, then 
\begin{align*}
&\mathbb{P}_x(\tau^\odot<\infty) \\
&= \frac{\Gamma\left(\frac{\alpha+d}{2}-1\right)\Gamma\left(\frac{\alpha}{2}\right)}{\Gamma \left(\frac{d}{2}\right)\Gamma(\alpha-1)}
\left\{
\begin{array}{rl}
{_2F_1((d-\alpha)/2,1 - \alpha/2, d/2;|x|^2)}
& 1>|x|\\
&\\
|x|^{\alpha-d}{_2F_1((d-\alpha)/2,1 - \alpha/2, d/2;|x|^{-2})}&1\leq |x|.
\end{array}
\right.
\end{align*}
Otherwise, if $\alpha\in(0,1]$, then $\mathbb{P}_x(\tau^\odot=\infty) = 1$ for all $|x|\neq 1$.
\end{theorem}

\begin{proof}[Proof \emph{($\clubsuit$)}]
From Secion \ref{radius}, we know that $|X|$ is a positive self-similar Markov process. Denote the underlying L\'evy processes associated through the Lamperti transform by $\xi$ with probabilities $\mathbf{P}_x$, $x\in\mathbb{R}$.
\[
\mathbb{P}_x(\tau^\odot<\infty) = \mathbf{P}_{\log|x|}(\tau^{\{0\}}<\infty) = \mathbf{P}_0(\tau^{\{\log (1/|x|)\}}<\infty),
\]
where $\tau^{\{z\}} = \inf\{t>0 : \xi_ t = z\}$, $z\in\mathbb{R}$. 
From this observation, we note the ability of $X$ to hit the sphere $\mathbb{S}_{d-1}$ with positive probability, boils down to the ability of $\xi$ to hit points with positive probability.
In this respect, classical theory of L\'evy processes comes to our rescue again and II.Theorem 16 of \cite{bertoin} tells us that a necessary and sufficient condition boils down to the integrability of $(1+\Psi(z))^{-1}$, where $\Psi$ is the  characteristic exponent of $\xi$,   given by Theorem \ref{radialpsi}. 

Another important asymptotic relation for the gamma function that can be derived from Stirling's formula states that, asymptotically as $y\to\infty$, for $x \in {\mathbb R}$, 
\begin{equation}\label{App1_gamma_asymptotics}
|\Gamma(x+\i y)|=\sqrt{2\pi} 
e^{-\frac{\pi}{2} |y|} |y|^{x-\frac{1}{2}} (1+o(1)),
\end{equation}
 uniformly in any finite interval $-\infty<a\le x\le b<\infty$.  Hence, appealing to  \eqref{App1_gamma_asymptotics},
\begin{equation}
\frac{1}{\Psi(z)} = \frac{\Gamma(-\frac{1}{2}{\rm i}z)}{\Gamma(\frac{1}{2}(-{\rm i}z +\alpha ))}\frac{\Gamma(\frac{1}{2}({\rm i}z +d-\alpha))}{\Gamma(\frac{1}{2}({\rm i}z +d))} \sim |z|^{-\alpha} 
\label{psi-1}
\end{equation}
uniformly on $\mathbb{R}$ as $|z|\to\infty$. We thus conclude that $(1+\Psi(z))^{-1}$ is integrable and each sphere $\mathbb{S}_{d-1}$ can be reached with positive probability from any $x$ with $|x|\neq 1$ if and only if $\alpha\in(1,2)$.
Moreover, when $\alpha\in(1,2)$, Corollary II.18 of \cite{bertoin} again  gives us the identity
\begin{equation}\label{problog}
\mathbb{P}_x(\tau^\odot<\infty)  = \frac{u_\xi(\log(1/|x|))}{u_\xi(0)},
\end{equation}
where, up to a multiplicative constant,  the potential density $u_\xi$ can be computed via a Laplace inversion in the spirit of the computations completed in the proof of Theorem \ref{censoreddensity}.

To this end, 
we note that the  Laplace exponent of $\xi$, $\Psi(-{\rm i}z)$, 
%\begin{equation}
%{\Psi(-{\rm i}z)} = \frac{\Gamma(\frac{1}{2}(-z +\alpha ))}{\Gamma(-\frac{1}{2}z)}\frac{\Gamma(\frac{1}{2}(z +d))}{\Gamma(\frac{1}{2}(z +d-\alpha))}  
%\label{psi-2}
%\end{equation}
is well defined  for ${\rm Re}(z)\in(-d,\alpha)$ with roots at $0$ and $\alpha -d$. The transience of $X$ for $d\geq 2$ ensures that $\mathbf{E}[\xi_1]>0$ and hence, as $-\Psi(-{\rm i}z)$ is convex for real $z$, we easily deduce that  ${\rm Re}(\Psi(-{\rm i}z))>0$ for ${\rm Re}(z)\in (\alpha -d, 0)$. In particular, it  follows that the Laplace transform of $u_\xi$ is well defined for  ${\rm Re}(z)\in(\alpha-d,0)$ as
\[
\int_{\r}{\rm e}^{z x} u_\xi(x)  
\d x=\int_{0}^\infty {\rm e}^{-\Psi(-{\rm i}z)t} 
\d t= \frac{1}{\Psi(-{\rm i}z)}.
\] 
As a consequence we can compute $u_\xi$ as a Laplace inversion in the form 
\[
u_\xi(x) = \frac{1}{2\pi{\rm i}}\int_{c+{\rm i}\r}\frac{{\rm e}^{-z x}}{{\Psi(-{\rm i}z)}}  
\d z, \qquad x\in\mathbb{R} ,
\]
providing $c\in (\alpha-d,0)$.

As we have seen in the computation of \eqref{pot_density_proof_3}, this integral can be computed using  relatively straightforward residue calculus. Indeed, from \eqref{psi-1} we note that $1/\Psi(-{\rm i}z)$ has simple poles at $\{2n, n\geq 0\}$, and $\{ -2n-(d-
\alpha):n\geq 0\}$.

We can construct a  contour integral, $\gamma_{R} =\{c+{\rm  i}x : |x|\leq R\}\cup\{c + R{\rm e}^{{\rm i}\theta}: \theta\in(\pi/2,3\pi/2)\}$, where $c\in (\alpha-d,0)$; see Fig. \ref{ch9contour}.
\begin{figure}[h!]
\begin{center}
\includegraphics[width = 5cm]{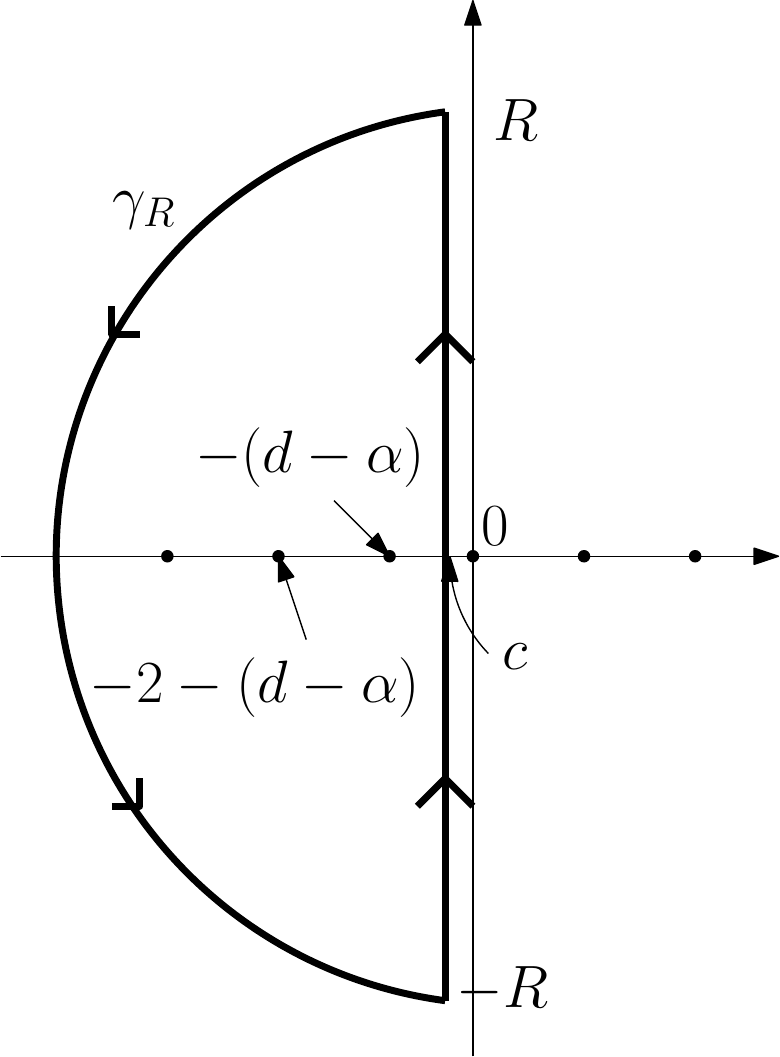}
\end{center}
\caption{The contour integral $\gamma_R$.}
\label{ch9contour}
\end{figure}

Residue calculus now gives us 
\begin{align}
&\frac{1}{2\pi{\rm i}}\int_{c-{\rm i}R}^{c+{\rm i}R}\frac{{\rm e}^{-z x}}{{\Psi(-{\rm i}z)}} \d z\notag\\
&\hspace{1cm}= 
-\frac{1}{2\pi{\rm i}}\int_{c + R{\rm e}^{{\rm i}\theta}: \theta\in(\pi/2,3\pi/2)}\frac{{\rm e}^{-z x}}{{\Psi(-{\rm i}z)}} \d z\notag\\
&\hspace{2cm}
%- {\rm Res}\left(\frac{{\rm e}^{-zx}}{\Psi(-{\rm i}z)}; z = 0\right) 
+\sum_{1\leq n\leq  \lfloor R \rfloor} {\rm Res}\left(\frac{{\rm e}^{-zx}}{\Psi(-{\rm i}z)}; z = -2n-(d-\alpha)\right).
\label{takethelimitinR}
\end{align}
Now fix $x\leq 0$. Appealing again to the uniform estimate \eqref{App1_gamma_asymptotics},  the assumption $x\leq0$ and the fact that the arc length of  $\{ c+ R{\rm e}^{{\rm i}\theta}: \theta\in(\pi/2,3\pi/2)\}$ is $\pi R$, we have 
 \[
\left| \int_{c+ R{\rm e}^{{\rm i}\theta}: \theta\in(\pi/2,3\pi/2)} \frac{{\rm e}^{-xz}}{\Psi(-{\rm i}z)}{\rm d}z \right|\leq C R^{-(\alpha-1)} \rightarrow0
 \]
 as $R\rightarrow\infty$ for some constant $C>0$. By taking limits  in \eqref{takethelimitinR}, we now have
\begin{align}
u_\xi(x) =
%- {\rm Res}\left(\frac{{\rm e}^{-zx}}{\Psi(-{\rm i}z)}; z = 0\right) 
\sum_{n\geq 1} {\rm Res}\left(\frac{{\rm e}^{-zx}}{\Psi(-{\rm i}z)}; z =  -2n-(d-\alpha)\right).
\label{residuesum}
\end{align}
Using the residues of the gamma function, that is,  for $n= 0,1,\cdots$, ${\rm Res}(\Gamma(z);z= -n) = (-1)^{n}/n!$, we have in  \eqref{residuesum}, for $x\leq 0$,
\begin{align}
u_\xi(x)& =\sum_0^\infty(-1)^{n+1}\frac{\Gamma(n+(d-\alpha)/2)}{\Gamma(-n + \alpha/2 )\Gamma(n + d/2)}\frac{{\rm e}^{2nx}}{n!}\notag\\
&={\rm e}^{x(d-\alpha)} 
\frac{\Gamma((d-\alpha)/2)}{\Gamma(\alpha/2)\Gamma(d/2)}\sum_0^\infty
\frac{((d-\alpha)/2)_n(1 - \alpha/2 )_n}{(d/2)_n}\frac{{\rm e}^{2nx}}{n!}\notag\\
&={\rm e}^{x(d-\alpha)} \frac{\Gamma((d-\alpha)/2)}{\Gamma(\alpha/2)\Gamma(d/2)} {_2F_1((d-\alpha)/2,1 - \alpha/2, d/2; {\rm e}^{2x} )}
\label{feedbackintolog}
\end{align}
where $(a)_n = \Gamma(n+a)/\Gamma(a)$ is the Pochhammer symbol and we have used the relation 
\[
\Gamma(-n+x) = (-1)^{n-1}\frac{\Gamma(-x)\Gamma(1+x)}{\Gamma(n+1-x)}, \qquad n\in\mathbb{N}
\]
 and the recursion formula for gamma functions.

Note in particular, this tells us that 
\begin{align*}
u_\xi(0) &= \frac{\Gamma((d-\alpha)/2)}{\Gamma(\alpha/2)\Gamma(d/2)} {_2F_1((d-\alpha)/2,1 - \alpha/2, d/2; 1 )} \\
&= \frac{\Gamma((d-\alpha)/2)}{\Gamma(\alpha/2)\Gamma(d/2)} 
\frac{\Gamma(d/2)\Gamma(\alpha-1)}{\Gamma(\alpha/2)\Gamma((\alpha+d)/2 -1)}
\end{align*}
Feeding \eqref{feedbackintolog} back in \eqref{problog}, we have thus  established an identity for $\mathbb{P}_x(\tau^\odot<\infty)$ when $1\leq |x|$. To deal with the case $1>|x|$, we can appeal to the Riesz--Bogdan--\.Zak transform to help us.
To this end we note that, for $|x|<1$,
\[
\mathbb{P}_{x/|x|^2}(\tau^\odot<\infty)
 = \mathbb{E}_{x}\left[
 \frac{|X_{\tau^\odot}|^{\alpha-d} }{|x|^{\alpha-d}}
 \mathbf{1}_{(\tau^\odot<\infty)}
 \right]=
 \frac{1 }{|x|^{\alpha-d}}\mathbb{P}_{x}(\tau^\odot<\infty)
\]
and hence, 
\[
\mathbb{P}_{x}(\tau^\odot<\infty) = \frac{\Gamma\left(\frac{\alpha+d}{2}-1\right)\Gamma\left(\frac{\alpha}{2}\right)}{\Gamma \left(\frac{d}{2}\right)\Gamma(\alpha-1)} {_2F_1((d-\alpha)/2,1 - \alpha/2, d/2; |x|^{2} )},
\]
thus completing the proof. 
%To generalise to a formula for $\mathbb{P}_{x}(\tau^\odot<\infty) $, where $|x|>a$, we can appeal to a straightforward scaling argument. We leave the details to the reader. 
\end{proof}

%we appeal to a semi-circular or rectangular contour integral, which runs along the real axis and arcs into the upper complex plain to ensure that, again with the help of \eqref{psi-1} and the observation   that $\exp\{-{\rm i}xz\}\rightarrow0$ as ${\rm Im}(z)\rightarrow\infty$, only the integral along the real axis contributes to the total value of the integral along the contour. We now have 
%\begin{align*}
%u_\xi(x)& = 4\pi{\rm e}^{(d-\alpha)x}\sum_0^\infty(-1)^{n+1}\frac{\Gamma(n+(d-\alpha)/2)}{\Gamma(-n + \alpha/2 )\Gamma(n + d/2)}\frac{{\rm e}^{2nx}}{n!}\\
%&=4\pi{\rm e}^{(d-\alpha)x}\, {_2F_1((d-\alpha)/2,1 - \alpha/2, d/2; {\rm e}^{2x} )}.
%\end{align*}
%
%Note in particular that 
%\[
%u_\xi(0) = 4\pi\, {_2F_1(d-\alpha)/2,1 - \alpha/2, d/2; 1)} =4\pi\frac{\Gamma (d/2)\Gamma(\alpha-1)}{\Gamma(\alpha/2)\Gamma({(d+\alpha)}/{2} -1)}
%\]
%The proof is now completed by plugging the expression for $u_\xi$ into \eqref{problog}.

\subsection{Some Riesz potentials on $\mathbb{S}_{d-1}$}

We continue our analysis, working towards the proof of Theorem \ref{hitasphere}, keeping close to the steps in \cite{Port69}. The proof of the next theorem is nonetheless new.

\begin{theorem}[$\heartsuit$]\label{Riesz} Suppose $\alpha\in (1,2)$.
For all $x\in\mathbb{R}^d$, 
\begin{equation}
\mathbb{P}_x(\tau^\odot<\infty) =\frac{\Gamma\left(\frac{\alpha+d}{2}-1\right)\Gamma\left(\frac{\alpha}{2}\right)}{\Gamma \left(\frac{d}{2}\right)\Gamma(\alpha-1)} \int_{\mathbb{S}_{d-1}}|z-x|^{\alpha -d}\sigma_1(\d z).
\label{everywhere}
\end{equation}
In particular, for $y\in\mathbb{S}_{d-1}$,
\begin{equation}
 \int_{\mathbb{S}_{d-1}}|z-y|^{\alpha - d}\sigma_1(\d z) =\frac{\Gamma \left(\frac{d}{2}\right)\Gamma(\alpha-1)}{\Gamma\left(\frac{\alpha+d}{2}-1\right)\Gamma\left(\frac{\alpha}{2}\right)}.
\label{findtheconstant}
\end{equation}
\end{theorem}
\begin{proof}[Proof \emph{($\clubsuit$)}]We start by recalling from \eqref{d-circ} and Proposition \ref{h-ssMp-d} that, for all $z\in\mathbb{R}^d$, 
\[
|X_t-z|^{\alpha - d}, \qquad t\geq 0,
\]
is a martingale.  Indeed, note that for $x\neq 0$, $z\in\mathbb{R}^d$,
\[
\mathbb{E}_x[|X_t-z|^{\alpha - d}] = \mathbb{E}_{x-z}[|X_t|^{\alpha -d}] = |x-z|^{\alpha -d},
\]
where we have used that $(|X_t|^{\alpha-d}, t\geq 0)$ is a martingale.
As a consequence, we note that, for $0\leq s\leq t$, 
\begin{align*}
\mathbb{E}_x\left[\left.\int_{\mathbb{S}_{d-1}}|z-X_{t\wedge \tau^\odot}|^{\alpha -d}\sigma_1(\d z) \right| {\mathcal{F}_s}\right]
&= \int_{\mathbb{S}_{d-1}}\mathbb{E}_x\left[\left.|z-X_{t\wedge \tau^\odot}|^{\alpha -d}  \right|{\mathcal{F}_s}\right] \sigma_1(\d z)\\
&=\int_{\mathbb{S}_{d-1}}|z-X_{s\wedge \tau^\odot}|^{\alpha -d}  \sigma_1(\d z),
\end{align*}
thereby suggesting that 
\[
M_t : = \int_{\mathbb{S}_{d-1}}|z-X_{t\wedge \tau^\odot}|^{\alpha -d}\sigma_1(\d z), \qquad t\geq 0,
\]
is a martingale. 
Recalling that $X$ is transient in the sense that $\lim_{t\to\infty}|X_t| = \infty$, we note that, since $d\geq 2>\alpha$, on the event $\{\tau^\odot=\infty\}$,
\[
\lim_{t\to\infty}\int_{\mathbb{S}_{d-1}}|z-X_{t\wedge \tau^\odot}|^{\alpha -d}\sigma_1(\d z)  = 0.
\]
Hence $(M_t, t\geq 0)$ is uniformly integrable with almost sure limit
\begin{align*}
M_\infty: = \lim_{t\to\infty}M_t =\int_{\mathbb{S}_{d-1}}|z-X_{ \tau^\odot}|^{\alpha -d}\sigma_1(\d z)\mathbf{1}_{(\tau^\odot<\infty)}\stackrel{d}{=}C \mathbf{1}_{(\tau^\odot<\infty)},
\end{align*}
where, despite the randomness in $X_{ \tau^\odot}$, by rotational symmetry, 
\begin{equation}
C = \int_{\mathbb{S}_{d-1}}|z-{\texttt 1}|^{\alpha -d}\sigma_1(\d z),
\label{C}
\end{equation}
and ${\texttt 1} = (1,0,\cdots,0)\in\mathbb{R}^d$ is the `North Pole' on $\mathbb{S}_{d-1}$. Note, it is not too difficult to see that $C$ is finite.

Preservation of martingale expectation for our martingale now ensures that 
\[
C\mathbb{P}_x(\tau^\odot<\infty) = \int_{\mathbb{S}_{d-1}}|z-x|^{\alpha -d}\sigma_1(\d z) .
\]
and hence taking the limit as $|x|\to0$, noting the right hand side above tends to unity, rotational symmetry and \eqref{C} ensures that, for all $y$ such that $|y| = 1$,
\[
 C =\int_{\mathbb{S}_{d-1}}|z-y|^{\alpha -d}\sigma_1(\d z)  = \frac{1}{\mathbb{P}(\tau^\odot<\infty) }.
\]
Thanks to scaling and Theorem \ref{everhit}, it is clear that  
\[
\mathbb{P}(\tau^\odot<\infty)=\frac{\Gamma\left(\frac{\alpha+d}{2}-1\right)\Gamma\left(\frac{\alpha}{2}\right)}{\Gamma \left(\frac{d}{2}\right)\Gamma(\alpha-1)} ,
\]
which completes the proof.
\end{proof}

The next result identifies the unique solution to a fixed point Riesz potential  equation as the  probability distribution that appears in Theorem \ref{hitasphere}. It is a key step in \cite{Port69}, but also in other works pertaining to first passage problems for stable processes during this era (as indeed we shall see in forthcoming sections); see e.g. \cite{Ray} for its application in even earlier work, not to mention in the original work of \cite{Riesz}. We have also taken inspiration from \cite{BH}.
\begin{lemma}[$\heartsuit$]\label{thecaseforuniqueness} Suppose that $\alpha\in(1,2)$.
Write $\mu^\odot_x(\d z) = \mathbb{P}_x(X_{\tau^\odot}\in \d z) $ on $\mathbb{S}_{d-1}$ where $x\in\mathbb{R}^d\backslash\mathbb{S}_{d-1}$. Then the measure $\mu_x^\odot$ is the unique solution to 
\begin{equation}\label{fixedpointmeasure}
|x-y|^{\alpha - d} = \int_{\mathbb{S}_{d-1}}| z- y|^{\alpha - d}\mu(\d z), \qquad y\in\mathbb{S}_{d-1}.
\end{equation}
\end{lemma}
\begin{proof}[Proof \emph{($\heartsuit$)}] 
We begin by recalling the expression for the potential of the stable process in Theorem \ref{freepotential} which states that, due to transience, 
\[
\int_0^\infty \mathbb{P}_x(X_t\in \d y) \d t = C(\alpha)|x-y|^{\alpha - d}\d y, \qquad x,y\in\mathbb{R}^d,
\]
 where $C(\alpha)$ is an unimportant constant in the following discussion. Suppose now that we fix an arbitrary $y\in\mathbb{S}_{d-1}$. Then  a straightforward application of the Strong Markov Property tells us that, for $x\in\mathbb{R}^d\backslash\mathbb{S}_{d-1}$,  
 \begin{align*}
 \int_0^\infty \mathbb{P}_x(X_t\in \d y) \d t & = \mathbb{E}_x\left[\mathbf{1}_{(\tau^{\odot}<\infty)} \int_0^\infty \mathbb{P}_z(X_t\in \d y)|_{z= X_{\tau^\odot}} \d t \right]\notag\\
 & = \mathbb{E}_x\left[\mathbf{1}_{(\tau^{\odot}<\infty)} C(\alpha)| X_{\tau^\odot}- y|^{\alpha - d}  \right]\d y\notag\\
&= C(\alpha)\int_{\mathbb{S}_{d-1}}| z- y|^{\alpha - d} \mathbb{P}_x(X_{\tau^\odot}\in \d z)\d y,
 \end{align*}
which shows that $\mu_x^\odot$ is a solution to \eqref{fixedpointmeasure}.

Let us now address the issue of uniqueness in \eqref{functionalout}. If $\mu(\d y)$ is any other probability measure supported on $\mathbb{S}_{d-1}$, then, for $|x|\neq 1$, 
\begin{align*}
\int_{\mathbb{S}_{d-1}}\int_{\mathbb{S}_{d-1}}| z- y|^{\alpha - d}\mu^\odot_x(\d z)\mu(\d y)
=\int_{\mathbb{S}_{d-1}}|x-y|^{\alpha - d}\mu(\d y)
%\leq \sup_{y\in \mathbb{S}_{d-1}} |x-y|^{\alpha - d}
 \leq (1+|x|)^{\alpha -d},
\end{align*}
where we have used rotational symmetry in the second equality so that the largest value the integrand can take occurs when $y=x/|x|$.
%where we have used $\alpha<2\leq d$ on the right-hand side. 
%It is the  finite nature of double integrals against products of the form  $\mu^\odot_x(\d z)\mu(\d y)$ which will help us show that it is the unique solution to the functional equation \eqref{functionalout}.

Now  suppose that $\nu$ is a signed measure on $\mathbb{S}_{d-1}$ which satisfies
\[
\int_{\mathbb{S}_{d-1}}\int_{\mathbb{S}_{d-1}}| z- y|^{\alpha - d}\,|\nu|(\d z)\,|\nu|(\d y)<\infty.
\]
Here, we understand $|\nu| = \nu^++\nu^-$, when we represent $\nu= \nu^+- \nu^-$. We claim that 
\begin{equation}
\int_{\mathbb{S}_{d-1}}| z- y|^{\alpha - d}\nu(\d z)=0
\label{signedzero}
\end{equation}
implies that  the measure $\nu\equiv0$.
The significance of this claim is that it would immediately imply that \eqref{fixedpointmeasure} has a unique solution in the class of probability measures.

Verifying the claim \eqref{signedzero} is not too difficult. Indeed, if we write $\texttt{p}(z,t)$ for the transition density of $X$, that is, $\mathbb{P}_x(X_t\in\d y) = \texttt{p}(y-x,t)\d y$, then a standard Fourier inverse tells us that, for $x\in\mathbb{R}^d$, 
\[
\texttt{p}(x,t) = (2\pi)^{-d} \int_{\mathbb{R}^d}{\rm e}^{{\rm i}\theta\cdot x} {\rm e}^{-|\theta|^\alpha t}\d \theta. 
\]
 note that 
\begin{align*}
&C(\alpha)\int_{\mathbb{S}_{d-1}}\int_{\mathbb{S}_{d-1}}| z- y|^{\alpha - d}\,\nu(\d z)\,\nu(\d y) \\
&\hspace{1cm} = \int_{\mathbb{S}_{d-1}}\int_{\mathbb{S}_{d-1}}\int_0^\infty \texttt{p}(z-y,t) \,\d t\,\nu(\d z)\,\nu(\d y)\\
&\hspace{1cm}=\int_{\mathbb{S}_{d-1}}\int_{\mathbb{S}_{d-1}}\int_0^\infty (2\pi)^{-d} \int_{\mathbb{R}^d}{\rm e}^{{\rm i}\theta\cdot (z-y)} {\rm e}^{-|\theta|^\alpha t}\d \theta \,\d t\,\nu(\d z)\,\nu(\d y)\\
&\hspace{1cm}= (2\pi)^{-d}\int_0^\infty  \d t  \int_{\mathbb{R}^d}\d \theta \, {\rm e}^{-|\theta|^\alpha t}\left(\int_{\mathbb{S}_{d-1}}{\rm e}^{-{\rm i}\theta\cdot y}\nu(\d y)\right)
\left(\int_{\mathbb{S}_{d-1}} {\rm e}^{{\rm i}\theta\cdot z}\nu(\d z)\right)\\
&\hspace{1cm}= (2\pi)^{-d}\int_0^\infty  \d t  \int_{\mathbb{R}^d}\d \theta \, {\rm e}^{-|\theta|^\alpha t}|\phi(\theta)|^2,
\end{align*}
where $\phi(\theta) = \int_{\mathbb{S}_{d-1}} {\rm e}^{{\rm i}\theta\cdot z}\nu(\d z)$. The assumption that $\int_{\mathbb{S}_{d-1}}| z- y|^{\alpha - d}\nu(\d z)=0$ thus implies that $\phi\equiv0$, which, in turn, implies $\nu\equiv0$, as claimed, and hence \eqref{fixedpointmeasure} has a unique solution.
\end{proof}

\subsection{Distribution of hitting location on $\mathbb{S}_{d-1}$}

\begin{proof}[Proof of Theorem \ref{hitasphere} \emph{($\heartsuit$)}]
We only prove the first part of the theorem. The last part has been established in Theorem \ref{everhit}.

First assume that $|x|>1$. Starting with the equality \eqref{findtheconstant}
 we want to apply the transformation \eqref{sphereinversion} through the sphere $\mathbb{S}_{d-1}(x, (|x|^2-1)^{1/2})$ remembering that   this transformation maps $\mathbb{S}_{d-1}$ to itself. 
 From \cite{blum} we know that, if we write $y = rA(\theta)$, where $r = |y|>0$ and $A(\theta) = {\rm Arg}(y)$ for parameterisation $\theta= (\theta_1,\cdots, \theta_{d-1})$, where $ \theta_j\in[0,\pi]$ and $\theta_{d-1}\in[0,2\pi)$ then 
\begin{equation}
\d y = r^{d-1}
\mathcal{J}(\theta)\d\theta\d r=\frac{2\pi^{d/2}}{\Gamma(d/2)}r^{d-1}\sigma_1(\d y)\d r,
\label{polardifferential}
\end{equation}
where $\mathcal{J}$ is the part of the Jacobian of $y$ with respect to $(r,\theta)$ which depends on $\theta$ and $\sigma_1(\d y)$ is the surface measure on $\mathbb{S}_{d-1}$ normalised to have unit mass. Suppose now that we write $z = wA(\theta)$ and set $y = Kz = w^{-1}A(\theta)$, then, if we set $r = w^{-1}$,
\begin{equation}
\d y= \left.r^{d-1}
\mathcal{J}(\theta)\right|_{r = w^{-1}}\d\theta \frac{\d r}{\d w}\d w=  w^{-2d} \cdot w^{d-1}\mathcal{J}(\theta)\d w\d\theta = |z|^{-2d} \d z.
\label{useforsphereinversion}
\end{equation}
Now taking account of the fact that, transforming through the sphere $\mathbb{S}_{d-1}(x,(|x|^2-1)^{1/2})$,   $z^* = x+K\tilde{z}$, where $\tilde{z} = (1-|x|^2)^{-1}(z-x)$, we can work with the change of variables
\begin{equation}
\d z^*=  (|x|^2-1)^{2d} |z-x|^{-2d}\prod_{i=1}^d\frac{\d z_i}{(|x|^2-1)} =( |x|^2-1)^{d} |z-x|^{-2d}\d z,
\label{usethat}
\end{equation}
where $z= (z_1,\cdots, z_d)$.
In particular, appealing to 
\eqref{polardifferential} and \eqref{difference}, this tells us that, using obvious notation, 
\[
\frac{2\pi^{d/2}}{\Gamma(d/2)}(r^*)^{d-1}\sigma_1(\d z^*)
\d r^* =  \left\{\frac{2\pi^{d/2}}{\Gamma(d/2)}\frac{(|x|^2-1)^{d-1}}{ |z-x|^{2d-2}}r^{d-1}\sigma_1(\d z)\right\}\left\{\frac{(|x|^2-1)}{ |z-x|^2}\d r\right\}.
\]
In particular,  since $|z|^* = 1$ if and only if $|z|=1$, the change of variable in the surface measure $\sigma_1$ satisfies
\[
\sigma_1(\d z^*) =\frac{(|x|^2-1)^{d-1}}{ |z-x|^{2d-2}}\sigma_1(\d z),\qquad z\in\mathbb{S}_{d-1}.
\] 
Taking account of \eqref{spheretosphere}, this can equivalently be written as 
\begin{equation}
\frac{1 }{|z^*-x|^{d-1}}\sigma_1(\d z^*) = \frac{1}{|z-x|^{d-1}}\sigma_1(\d z) , \qquad z\in\mathbb{S}_{d-1}.
\label{symmetry}
\end{equation}

Returning to  \eqref{findtheconstant}, with the help of \eqref{difference} and \eqref{spheretosphere}  for the transformation \eqref{sphereinversion} in $\mathbb{S}_{d-1}(x, (|x|^2-1)^{1/2})$ and \eqref{symmetry}, this gives us for $x\in\mathbb{R}^d\backslash\mathbb{S}_{d-1}$ and $y\in\mathbb{S}_{d-1}$,
 \begin{align}
\frac{\Gamma \left(\frac{d}{2}\right)\Gamma(\alpha-1)}{\Gamma\left(\frac{\alpha+d}{2}-1\right)\Gamma\left(\frac{\alpha}{2}\right)}
&=\int_{\mathbb{S}_{d-1}}|z^*-x|^{d-1}|z^*-y^*|^{\alpha- d}\frac{\sigma_1(\d z^*)}{|z^*-x|^{d-1}}\notag \\
&= \frac{(|x|^2-1)^{\alpha-1}}{|y-x|^{\alpha-d}}\int_{\mathbb{S}_{d-1}} 
\frac{|z-y|^{\alpha-d}}{|z-x|^{\alpha + d-2}}
\sigma_1(\d z).
\label{onlyonthesphere}
 \end{align}
 Which is to say 
 \[
 |x-y|^{\alpha-d} = \int_{\mathbb{S}_{d-1}}|z-y|^{\alpha-d}
\frac{\Gamma\left(\frac{\alpha+d}{2}-1\right)\Gamma\left(\frac{\alpha}{2}\right)}{\Gamma \left(\frac{d}{2}\right)\Gamma(\alpha-1)}
\frac{(|x|^2-1)^{\alpha-1}}{|z-x|^{\alpha + d-2}}
\sigma_1(\d z),
 \]
 which by uniqueness given in Lemma \ref{thecaseforuniqueness} establishes the statement of the theorem for $|x|>1$.

Finally for the case $|x|<1$, we can appeal to similar reasoning albeit now using $x^\diamond$ to invert through the sphere $\mathbb{S}_{d-1}(x, (1-|x|^2)^{1/2})$. The details are left to the reader.
\end{proof}

\begin{remark}
Although we have excluded the setting of Brownian motion, that is, the case $\alpha = 2$, our analysis can be easily adapted to include it.  In that case,  the conclusion of Theorems \ref{hitasphere} and  \ref{everhit} provide us with the classical Newtonian Poisson potential formula. Indeed, for $|x|<1$,
\begin{equation}
\mathbb{P}_x(X_{\tau^\odot}<\infty)=1= \int_{\mathbb{S}_{d-1}}\frac{( 1-|x|^2)}{|z-x|^{d}}
\sigma_1(\d z).
\label{provesNewton}
\end{equation}
Similarly we can also reproduce the classical conclusion that, for $|x|>1$,
\begin{equation}
\mathbb{P}_x(X_{\tau^\odot}<\infty)= |x|^{2-d} =\int_{\mathbb{S}_{d-1}}\frac{( |x|^2-1)}{|z-x|^{d}}
\sigma_1(\d z).
\label{BMhitballfromoutside}
\end{equation}
\end{remark}

\subsection{Resolvent on $\mathbb{R}^d\backslash\mathbb{S}_{d-1}$}

 Write the associated potential measure as
\[
U^\odot(x,\d y) = \int_0^\infty \mathbb{P}_x(X_t \in \d y, \, t<\tau^\odot)\d t, \qquad x,y\in\mathbb{R}^d\backslash\mathbb{S}_{d-1}.
\]
The next result again is taken from \cite{Port69}. We follow the original proof, with some additional efficiencies introduced thanks to the Riesz--Bogdan--\.Zak transform.

\begin{theorem}[$\heartsuit$] Suppose we write $Q(x) = \mathbb{P}_x(\tau^\odot<\infty)$, for $x\in\mathbb{R}^d$.
Then, for all $x,y\in\mathbb{R}^d\backslash\mathbb{S}_{d-1}$, 
\[
U^\odot(x, \d y)= 2^{-\alpha}\pi^{-d/2}\frac{\Gamma((d-\alpha)/2)}{\Gamma(\alpha/2)}|x-y|^{\alpha -d}\left(1- Q\left(
\frac{y}{|y-x|}\left|x  - \frac{y}{|y|^2}\right|
\right)\right)\d y.
\]
\end{theorem}
\begin{proof}[Proof \emph{($\diamondsuit$)}]
Let us preemptively assume that 
 $U^\odot(x,\d y)$ has density with respect to Lebesgue measure, written $u^\odot(x,y)$, $x,y\in\mathbb{R}^d\backslash\mathbb{S}_{d-1}$. As alluded to above, straightforward path counting tells us that
 \begin{align*}
 u^\odot(x,y)= C_\alpha |x-y|^{\alpha -d} -C_\alpha \int_{\mathbb{S}_{d-1}}|z-y|^{\alpha-d}h^\odot(x,z) \sigma_1(\d z),
 \end{align*}
 where for convenience we have written $C_\alpha =  2^{-\alpha}\pi^{-d/2}{\Gamma((d-\alpha)/2)}/{\Gamma(\alpha/2)}$.
In order to deal with the integral on the right-hand side above, we need to split our computations into the cases that $|x|<1$ and $|x|>1$. 

First assume that $|x|>1$. We appeal to a sphere inversion of the type \eqref{sphereinversion} via the sphere $\mathbb{S}_{d-1}(x, (|x|^2-1)^{1/2})$ in a manner similar to the computation in \eqref{onlyonthesphere}. Indeed, reading only the second equality of  \eqref{onlyonthesphere}, we see that 
\begin{align*}
&C_\alpha \int_{\mathbb{S}_{d-1}} |z-y|^{\alpha-d}h^\odot(x,z)\sigma_1(\d z)\\
& =C_\alpha\frac{\Gamma\left(\frac{\alpha+d}{2}-1\right)\Gamma\left(\frac{\alpha}{2}\right)}{\Gamma \left(\frac{d}{2}\right)\Gamma(\alpha-1)}
 |x-y|^{\alpha-d} \int_{\mathbb{S}_{d-1}}|z^*-y^*|^{\alpha- d}\sigma_1(\d z^*)\\
 &=C_\alpha |x-y|^{\alpha-d} Q(y^*),
\end{align*}
where in the second equality above, we have used Theorem \ref{Riesz}. Recalling that $y^* = x + (|x|^2-1)(y-x)/|y-x|^2$, a straightforward piece of algebra tells us that 
\[
|y^*|^2|y-x|^2 = |y|^2 \left|x- \frac{y}{|y|^2}\right|^2
\]
and the result follows as soon as we note that isometry implies that $Q(y^*) = Q(|y^*|y/|y|)$.

For the case that $|x|<1$, we can again appeal to the Riesz--Bogdan--\.Zak transform. The time change \eqref{etatimechange} implies that if we write $s = \eta(t)$ in the aforesaid transform, then ${\rm d}s/\d t = |X_s|^{2\alpha}$. Together with the  Doob $h$-transform \eqref{d-circ}, we have, for bounded measurable $f:\mathbb{R}^d\to\mathbb{R}^d$,
\begin{align*}
\int_{\mathbb{R}^d} f(y)\frac{|y|^{\alpha - d}}{|x|^{\alpha - d}}u^\odot(x, y)\d y&=\int_0^\infty \frac{|y|^{\alpha - d}}{|x|^{\alpha - d}}\mathbb{E}_x[f(X_t); \, t<\tau^\odot]\d t \\
&=\mathbb{E}^\circ_x\left[\int_0^\infty f(X_t)\mathbf{1}_{(t<\tau^\odot)}\d t \right]\\
&=\mathbb{E}_{Kx}\left[\int_0^\infty f(KX_s) \mathbf{1}_{(s<\tau^\odot)}|X_s|^{-2\alpha}\d s \right]\\
&=\int_{\mathbb{R}^d}f(Ky)|y|^{-2\alpha}u^\odot(Kx, y)\d y\\
&=\int_{\mathbb{R}^d}f(z)|z|^{2(\alpha-d)}u^\odot(Kx, Kz)\d z
, \qquad y\in\mathbb{R}^d\backslash\mathbb{S}_{d-1},
\end{align*}
where we recall that $\mathbb{E}^\circ_x$ is expectation with respect to $\mathbb{P}^\circ_x$, which was defined in \eqref{updownCOM7}, that  $Kx = x/|x|^2$ and we have used \eqref{useforsphereinversion} in the final equaltiy. We can now appeal to the expression we have just derived for $u^\odot$ previously on account of the fact that $|Kx| = 1/|x|>1 $. Equation \eqref{difference} for the transform $K$ tells us that $|Ky-Kx| = |x-y|/|x||y|$. Hence we have  for $|x|<1$ and $y\in\mathbb{R}^d\backslash\mathbb{S}_{d-1}$,
\begin{align*}
u^\odot(x, y) &= C_\alpha|x|^{\alpha - d}|y|^{\alpha - d}
|Kx-Ky|^{\alpha -d}\left(1- Q\left(
\frac{Ky}{|Ky-Kx|}\left|Kx  - KKy\right|
\right)\right)\\
&=C_\alpha
|x-y|^{\alpha -d}\left(1- Q\left(
\frac{y}{|y-x|}\left|x  -Ky\right|
\right)\right)
\end{align*}
as required.
\end{proof}

\section{First entrance and exit of the unit ball}\label{sphereinout}

\subsection{First passage laws for $\mathbb{B}_d$}Let us start by defining the stopping times 
\[
\tau^\oplus := \inf\{t>0 : |X_t| <1\} \text{ and }\tau^{\ominus}: = \inf\{t>0: |X_t|>1\}.
\]
Recall that $X$ is transient in dimension $d\geq 2$ and hence $\mathbb{P}_x(\tau^\oplus <\infty)<1$ for all $|x|\geq 1$ and $\mathbb{P}_x(\tau^\ominus <\infty)=1$ for all $|x|\leq 1$.

\cite{BGR} give the  distribution of $X_{\tau^\ominus}$ and then use a transformation which is tantamount to that of Riesz--Bogdan--\.{Z}ak,  attributing their method to \cite{Riesz}, to give directly the distribution of $X_{\tau^\oplus}$. We state their result below. 
\begin{theorem}[$\heartsuit$]\label{outofballintoball}Define the function 
\[
g(x,y) = \pi^{-(d/2+1)}\,\Gamma(d/2)\,\sin(\pi\alpha/2)\frac{\left|1 - |x|^2\right|^{\alpha/2}}{\left|1-|y|^2\right|^{\alpha/2}}|x-y|^{-d}
\]
for $x,y\in\mathbb{R}^d\backslash\mathbb{S}_{d-1}$.
\begin{description}
\item[(i)] Suppose that $|x|<1$, then 
\begin{equation}
\mathbb{P}_x(X_{\tau^\ominus}\in \d y) = g(x,y)\d y, \qquad |y|\geq 1.
\label{jumpout}
\end{equation}
\item[(ii)] Suppose that  $|x|>1$, then 
\begin{equation}
\mathbb{P}_x(X_{\tau^\oplus}\in \d y, \, \tau^\oplus<\infty) = g(x,y)\d y, \qquad |y|\leq 1.
\label{jumpin}
\end{equation}
\end{description}
\end{theorem}
The fact that the  Riesz--Bogdan--\.{Z}ak transform lies behind the relationship between the distributions of $X_{\tau^\ominus}$ and $X_{\tau^\oplus}$ should reassure the reader that one uses a single density function to describe both laws. As we shall shortly see, this function is symmetric in relation to the sphere inversion $(x,y)\mapsto (Kx, Ky)$.

We follow the original steps of \cite{BGR}, first  passing  through an intermediate result, which mirrors the approach of  Lemma \ref{thecaseforuniqueness}. We omit its proof as it follows an almost identical thread to that of Lemma \ref{thecaseforuniqueness}.

\begin{lemma}[$\heartsuit$]\label{guessandverify}We have that, for $|x|<1$, uniquely in the class of probability distributions supported on the exterior of $\mathbb{S}_{d-1}$, $\mu^\ominus_x(\d z): = \mathbb{P}_x(X_{\tau^\ominus}\in \d z) $ solves 
\begin{equation}
|x-y|^{\alpha - d} = \int_{|z|\geq 1}| z- y|^{\alpha - d}\mu(\d z), \qquad |y|> 1.
\label{functionalout}
\end{equation}
and, for $|x|>1$, again uniquely in the class of probability distributions supported on the interior of $\mathbb{S}_{d-1}$, $\mu^\oplus_x(\d z) : = \mathbb{P}_x(X_{\tau^\oplus}\in \d z, \, \tau^\oplus<\infty) $ uniquely solves 
\begin{equation}
|x-y|^{\alpha - d} = \int_{|z|\leq 1}| z- y|^{\alpha - d}\mu(\d z), \qquad |y|<  1.
\label{functionalin}
\end{equation}
\end{lemma}

We now turn our attention to showing that the unique solution to \eqref{functionalout} and \eqref{functionalin} are given by $\mu^{\ominus}_x(\d z) = g(x,z)\d z$, $|z|>1>|x|$ and $\mu^\oplus(\d z) = g(x,z)\d z$, $|z|< 1< |x|$.

\begin{proof}[Proof of Theorem \ref{outofballintoball} (i) \emph{($\heartsuit$)}] The proof is complete as soon as we can verify that 
\begin{equation}
|x-y|^{\alpha -d} = c_{\alpha,d}\int_{|z|\geq 1} |z-y|^{\alpha -d}\frac{|1- |x|^2|^{\alpha/2}}{|1-|z|^2|^{\alpha/2}}|x - z|^{-d}\d z
\label{needsverifying}
\end{equation}
for $|y|> 1>|x|$, where 
\[
c_{\alpha, d} = \pi^{-(1+d/2)}\,\Gamma(d/2)\,\sin(\pi\alpha/2).
\]

Starting with the integral on the right-hand side of \eqref{needsverifying},  we will appeal to the transformation \eqref{diamond} through the sphere $\mathbb{S}_{d-1}(x,(1-|x|^2)^{1/2})$, noting in particular that 
\begin{equation}
|z^\diamond - y^\diamond| = (1-|x|^2)\frac{|z-y|}{|z-x||y-x|}\,\,\text{ and }\,\, |z|^2- 1  = \frac{|z-x|^2}{1-|x|^2}(1- |z^\diamond|^2),
\label{usethis}
\end{equation}
where the second identity comes from \eqref{quadraticdistance}.
A similar analysis of the differential calculus associated to \eqref{diamond} gives us
\begin{equation}
\d z^\diamond 
%=  (1-|x|^2)^{2d} |z-x|^{-2d}\prod_{i=1}^d\frac{\d z_i}{(1-|x|^2)}
 = (1-|x|^2)^{d} |z-x|^{-2d}\d z, \qquad z\in\mathbb{R}^d,
\label{usethat}
\end{equation}
just as in \eqref{usethat}.

%As we will use these equalities to make substitutions on the right-hand side of \eqref{needsverifying}, we also want to understand how to change variables when converting an integral with respect to $\d z$ into an integral with respect to $\d z^\diamond$. Let us digress momentarily to this end, and consider the transformation $Kx = x/|x|^2$, $x\in\mathbb{R}^d\backslash\{0\}$. From Section \ref{Apolarintegrals} of the  appendix we know that, if we write $y = r\theta$, where $r>0$ and $\theta\in \mathbb{S}_{d-1}$, then 
%$
%\d y = r^{d-1}
%\mathcal{J}(\theta)\d\theta\d r
%$, where $\mathcal{J}$ is the part of the Jacobian which depends on $\theta$. Suppose now that we write $z = w\theta$ and set $y = Kz = w^{-1}\theta$, then, if we set $r = w^{-1}$,
%\begin{equation}
%\d y= \left.r^{d-1}
%\mathcal{J}(\theta)\right|_{r = w^{-1}}\d\theta \frac{\d r}{\d w}\d w=  w^{-2d} \cdot w^{d-1}\mathcal{J}(\theta)\d w\d\theta = |z|^{-2d} \d z.
%\label{useforsphereinversion}
%\end{equation}
%Now taking account of the fact that transforming through the sphere $\mathbb{S}_{d-1}(x,(1-|x|^2)^{1/2})$,   $z^\diamond = x-K\tilde{z}$, where $\tilde{z} = (1-|x|^2)^{-1}(z-x)$ we can work with the change of variables
%\begin{equation}
%\d z^\diamond =  (1-|x|^2)^{2d} |z-x|^{-2d}\prod_{i=1}^d\frac{\d z_i}{(1-|x|^2)} = (1-|x|^2)^{d} |z-x|^{-2d}\d z,
%\label{usethat}
%\end{equation}
%where $z= (z_1,\cdots, z_d)$.
%

Now we can use \eqref{usethis} and \eqref{usethat} to compute, for $|x|<1< |y|$,
\begin{align}
& c_{\alpha,d}\int_{|z|\geq 1} |z-y|^{\alpha -d}\frac{|1- |x|^2|^{\alpha/2}}{|1-|z|^2|^{\alpha/2}}|x - z|^{-d}\d z\notag\\
&= c_{\alpha,d}|y-x|^{\alpha-d}\int_{|z^\diamond|\leq 1}
\frac{|z^\diamond - y^\diamond|^{\alpha-d}}{|1-|z^\diamond|^2|^{\alpha/2}}\d z^\diamond.
\label{1/c}
\end{align}

Next we perform another transformation of the type \eqref{diamond}, albeit through the sphere $\mathbb{S}_{d-1}(y^\diamond, (1-|y^\diamond|^2)^{1/2})$. In a similar fashion to the calculation that led to the right-hand side of \eqref{1/c}, using \eqref{spheretospherediamond} and the second equality in \eqref{usethis}
we obtain
\begin{align}
& c_{\alpha,d}\int_{|z|\geq 1} |z-y|^{\alpha -d}\frac{|1- |x|^2|^{\alpha/2}}{|1-|z|^2|^{\alpha/2}}|x - z|^{-d}\d z\notag\\
&= c_{\alpha,d}|y-x|^{\alpha-d}\int_{|w|\geq 1} \frac{|1-|y^\diamond|^2|^{\alpha/2}}{|1-|w|^2|^{\alpha/2}}|w-y^\diamond|^{-d}\d w.
\label{1/2c}
\end{align}

The question now remains as to whether the integral on right-hand side of \eqref{1/2c} is equal to $1/c_{\alpha,d}$. We resolve this issue by recalling that the surface volume/area of a sphere of radius $r$ is given by ${2\pi^{d/2}}r^{d-1}/{\Gamma(d/2)}$. Moreover, writing, for $|y^\diamond|<1$,
\begin{align}
&\int_{|w|\geq 1} \frac{1}{|1-|w|^2|^{\alpha/2}}|w-y^\diamond|^{-d}\d w\notag\\
&= \frac{2\pi^{d/2}}{\Gamma(d/2)} \int_1^\infty r^{d-1}\d r\int_{\mathbb{S}_{d-1}(0,r)}
\frac{1}{|1-|z|^2|^{\alpha/2}}|z-y^\diamond|^{-d}\sigma_r(\d z)\notag\\
&= \frac{2\pi^{d/2}}{\Gamma(d/2)} \int_1^\infty \frac{r^{d-1}\d r}{|1-r^2|^{\alpha/2}}\int_{\mathbb{S}_{d-1}(0,r)}
|z-y^\diamond|^{-d}\sigma_r(\d z),
\label{1/3c}
\end{align}
where $\sigma_r(\d z)$ is the surface measure on $\mathbb{S}_{d-1}(0,r)$, normalised to have unit total mass. In order to continue, we remind the reader of a classical integral identity from  Newtonian $d$-dimensional potential theory. Specifically, the Poisson formula \eqref{provesNewton} gives us
\begin{equation}
\int_{\mathbb{S}_{d-1}(0,r)}
 \frac{r^{d-2}(r^2-|y^\diamond|^2)}{|z-y^\diamond|^{d}}\sigma_r(\d z)=1, \qquad |y^\diamond|<1< r.
\label{Poisson}
\end{equation}
(The reader will also note that the kernel in the integral above is the distribution of where a $d$-dimensional Brownian motion issued from $y^\diamond$ will hit  the sphere $\mathbb{S}_{d-1}(0,r)$.)

The identity \eqref{Poisson}  allows us to continue to develop the right-hand side of \eqref{1/3c} so that we have 
\begin{equation}
\int_{|v|\geq 1} \frac{1}{|1-|w|^2|^{\alpha/2}}|w-y^\diamond|^{-d}\d w
= \frac{\pi^{d/2}}{\Gamma(d/2)} \int_1^\infty \frac{2r}{(r^2-1)^{\alpha/2}(r^2-|y^\diamond|^2)}\d r.
\label{1/4c}
\end{equation}
A further change of variable, first $s = (r^2-1)/(1-|y^\diamond|^2)$ and the definition of the beta function
\begin{align*}
\int_1^\infty \frac{2r}{(r^2-1)^{\alpha/2}(r^2-|y^\diamond|^2)}\d r&
=\frac{1}{(1-|y^\diamond|^2)^{\alpha/2}}\int_0^\infty s^{-\alpha/2}(1+s)^{-1}\d s \\
&
 =\frac{1}{(1-|y^\diamond|^2)^{\alpha/2}}\Gamma(\alpha/2)\Gamma(1-\alpha/2)\\
 &=\frac{\pi}{\sin(\alpha\pi/2)}\frac{1}{(1-|y^\diamond|^2)^{\alpha/2}}.
 \end{align*}
Plugging back into \eqref{1/4c}, and then into \eqref{1/2c}, we end up with   
\[
\int_{|z|\geq 1} |z-y|^{\alpha -d}\frac{|1- |x|^2|^{\alpha/2}}{|1-|z|^2|^{\alpha/2}}|x - z|^{-d}\d z
=\frac{\pi^{1+d/2}}{\Gamma(d/2)\sin(\alpha\pi/2)}\\
=\frac{1}{c_{\alpha,d}}
\]
as required. 

The identity \eqref{needsverifying} is thus affirmed for all $|x|<1<|y|$ and hence the first part of Theorem \ref{outofballintoball} is proved. 
\end{proof}

\bigskip

\begin{proof}[Proof of Theorem  \ref{outofballintoball} (ii) \emph{($\diamondsuit$)}] We can appeal to the Riesz--Bogdan--\.Zak transform in Theorem \ref{BZ} and note that for Borel set $D\subseteq \{u: |u|\leq 1\}$ and $|x|>1$, 
\[
\mathbb{P}_x(X_{\tau^\oplus}\in D) = \mathbb{P}^\circ_{Kx}(KX_{\tau^\ominus}\in D),
\]
where we recall that $Kx = x/|x|^2$, $KD = \{Kx: x\in D\}$ and $\mathbb{P}^\circ_x$, $x\neq 0$, is the result of the Doob $h$-transform in \eqref{d-circ}. It follows that 
\begin{align}
&\mathbb{P}_x(X_{\tau^\oplus}\in D)  \notag\\
&= \int_{KD}\frac{|y|^{\alpha - d}}{|Kx|^{\alpha - d}}g(Kx, y)\d y \notag\\
&= c_{\alpha,d}\int_{KD}|z|^{d-\alpha}|Kx|^{d-\alpha}\frac{|1- |Kx|^2|^{\alpha/2}}{|1-|y|^2|^{\alpha/2}} |Kx - y|^{-d}\d y\notag\\
%&= \pi^{-(d/2+1)}\,\Gamma(d/2)\,\sin(\pi\alpha/2)\int_{D}|z|^{d-\alpha}|x|^{\alpha-d}\frac{|x|^\alpha|1- |x|^2|^{\alpha/2}}{|z|^\alpha |1-|z|^2|^{\alpha/2}}|Kx - Kz|^{-d}\d(Kz)\\
&= c_{\alpha,d}\int_{D}|z|^{2d}\frac{|1- |x|^2|^{\alpha/2}}{|1-|z|^2|^{\alpha/2}}|x - z|^{-d}\d(Kz),
\label{d(Kz)}
\end{align}
where we have used the change of variable $y = Kz$ together with \eqref{difference} in the final equality. 
%As we have seen in the proof of Theorem \ref{outofballintoball} (i), we know that, if we write $y = r\theta$, where $r>0$ and $\theta\in \mathbb{S}_{d-1}$, then 
%$
%\d y = r^{d-1}
%\mathcal{J}(\theta)\d r\d\theta
%$, where $\mathcal{J}$ is the part of the Jacobian which depends on $\theta$. If we change variable so that $r = w^{-1}$, which is equivalent to setting $y = Kz$ if $z = w\theta$, then  
%\[
%\d y= w^{-d-1}
%\mathcal{J}(\theta)\d w\d\theta =  w^{-2d} \cdot w^{d-1}\mathcal{J}(\theta)\d w\d\theta = |z|^{-2d} \d z.
%\]
Appealing to \eqref{useforsphereinversion}, we can now  complete the computation in \eqref{d(Kz)}. We have 
\begin{align*}
&\mathbb{P}_x(X_{\tau^\oplus}\in D) =c_{\alpha,d}\int_{D}\frac{|1- |x|^2|^{\alpha/2}}{|1-|z|^2|^{\alpha/2}}|x - z|^{-d}\d z = \int_{D}g(x, z)\d z,
\end{align*}
as required.
\end{proof}

Recalling that in dimensions $d\geq 2$, the stable process is transient. It thus makes sense to 
compute the probability that the unit ball  around the origin is never entered. That is to say, to compute the total mass of the measure $\mu^\oplus$. Naturally one can do this by marginalising the distribution density $g$, however it turns out to be simpler to make use of the Lamperti representation of $|X|$. We thus offer a new proof to a result of \cite{BGR}

\begin{lemma}[$\heartsuit$]\label{neverhit} We have for $|x|>1$,
\[
\mathbb{P}_x(\tau^\oplus = \infty) = \frac{\Gamma(d/2)}{\Gamma((d - \alpha)/2)\Gamma(\alpha/2)}\int_0^{|x|^2-1} (u+1)^{-d/2}u^{\alpha/2-1}\d u.
\]
\end{lemma}
\begin{proof}[Proof \emph{($\clubsuit$)}] From Theorem \ref{radialpsi} we have the Wiener--Hopf factorisation of the characteristic exponent of $\xi$, the L\'evy process appearing in the Lamperti transform of $|X|$. In particular, its descending ladder height process has Laplace exponent  given by 
$
{\Gamma((\lambda + \alpha)/2)}/{\Gamma(\lambda/2)}
$,
$\lambda\geq 0$. If we denote its descending ladder height potential measure by $\hat{U}_\xi$, then, from the discussion following \eqref{ladderpotentials}, we have that 
\[
\int_{[0,\infty)}{\rm e}^{-\lambda x} \hat{U}_\xi(\d x) = \frac{\Gamma((\lambda+d)/2)}{\Gamma((\lambda +d- \alpha)/2)}, \qquad \lambda \geq 0.
\]
This transform can be inverted explicitly and, pre-emptively assuming that it has a density, denoted by $\hat{u}_\xi(x)$, $x\geq 0$,  we have
\[
\int_{[0,\infty)}{\rm e}^{-2\lambda x} \hat{u}_\xi( x)\d x = \int_0^\infty {\rm e}^{-\lambda x} \frac{1}{\Gamma(\alpha/2)}{\rm e}^{-(d-\alpha)x/2} (1-{\rm e}^{-x})^{\alpha/2-1}\d x,
\]
so that
\[
\hat{U}_\xi(\d x)  = \frac{2}{\Gamma((d-\alpha)/2)} {\rm e}^{-(d-\alpha)x}(1-{\rm e}^{-2x})^{\alpha/2-1}\d x.
\]

Next, with the help of Lemma \ref{Utail} we note that 
\begin{align*}
\mathbb{P}_x(\tau^\oplus < \infty) &=\frac{\Gamma((d-\alpha)/2)}{\Gamma(d/2)\Gamma(\alpha/2)}\int_{\log |x|}^\infty 2 {\rm e}^{-(d-\alpha)y}(1-{\rm e}^{-2y})^{\alpha/2-1}\d y \\
&= \frac{\Gamma((d-\alpha)/2)}{\Gamma(d/2)\Gamma(\alpha/2)}\int_{\log |x|}^\infty (u+1)^{-d/2}u^{\alpha/2-1}\d u
\end{align*}
where, in the second equality,  we have made the substitution $u =  {\rm e}^{2y}-1$. Recalling from the definition of the beta function that 
\[
\frac{\Gamma(d/2)\Gamma(\alpha/2)}{\Gamma((d-\alpha)/2)}= 
\int_0^\infty(u+1)^{-d/2}u^{\alpha/2-1}\d u,
\]
the proof is complete. 
\end{proof}

\subsection{Resolvent inside and outside $\mathbb{B}_d$}
The conclusion of Theorem \ref{outofballintoball} also gives us the opportunity to study the potentials
\[
U^\ominus(x, \d y) = \int_0^{\infty}\mathbb{P}_x(X_t\in\d y, \, t<\tau^\ominus) \d t\]
and
\begin{equation}
U^\oplus(x, \d y) =\int_0^{\infty}\mathbb{P}_x(X_t\in\d y, \, t<\tau^\oplus) \d t,
\label{hasad}
\end{equation}
for $|x|, |y|<1$ and $|x|, |y|>1$, respectively.
Again due to \cite{BGR}, we have the following classical result. Again, we keep to their original proof, albeit with a more explicit implementation of the Riesz--Bogdan--\.Zak transform. 

\begin{theorem}[$\heartsuit$]\label{potentials}
Define the function 
\begin{equation}
h^{\ominus}(x, y) = 2^{-\alpha}\pi^{-d/2}\frac{\Gamma(d/2)}{\Gamma(\alpha/2)^2}|x-y|^{\alpha - d} \int_0^{\zeta^\ominus(x,y)}  (u+1)^{-d/2}u^{\alpha/2-1}\d u
\label{BGR3}
\end{equation}
\begin{description}
\item[(i)] In the case that  $|x|,|y|<1$, 
\[
U^\ominus(x, \d y)= h^{\ominus}(x, y) \d y,
\]
where $ \zeta^\ominus(x,y)= (1-|x|^2)(1-|y|^2)/|x-y|^{2}$.
\item[(ii)] In the case that $|x|,|y|>1$, 
\[
U^\oplus(x, \d y)= h^{\oplus}(x, y)\d y,
\]
where $h^\oplus$ has the same definition as $h^\ominus$ albeit that $\zeta^\ominus(x,y)$ is replaced by $\zeta^\oplus(x,y) = (|x|^2-1)(|y|^2-1)/|x-y|^{2}$.
\end{description}
\end{theorem}

\begin{proof}[Proof \emph{($\diamondsuit$)}]
The basic pretext of the proof of the first part again boils down to counting paths.  More precisely, we have by the Strong Markov Property that,
\[
\kappa_{\alpha,d}|x-y|^{\alpha-d} = h^{\ominus}(x,y) + \int_{|z|>1} \kappa_{\alpha,d}|z-y|^{\alpha -d}\mu^{\ominus}_x(\d z),
\]
 for $|x|, |y|<1$, where $ \kappa_{\alpha,d}= 2^{-\alpha}\pi^{-d/2}{\Gamma((d-\alpha)/2)}/{\Gamma(\alpha/2)}$. We are thus obliged to show that 
 \begin{align}
& |x-y|^{\alpha-d}-  \int_{|z|>1} |z-y|^{\alpha -d} g(x,z)\d z \notag\\
&= 
\frac{\Gamma(d/2)}{\Gamma(\alpha/2)\Gamma((\alpha -d)/2)} |x-y|^{\alpha - d} \int_0^{\zeta^\ominus(x,y)}  (u+1)^{-d/2}u^{\alpha/2-1}\d u,
\label{provethisforpotential}
 \end{align}
 where we recall that the density $g$ comes from Theorem \ref{outofballintoball}.
 
 To this end, start with the integral on the left-hand side of \eqref{provethisforpotential}. From \eqref{1/c} we have already shown that, by performing the transformation \eqref{diamond} through the sphere $\mathbb{S}_{d-1}(x, (1-|x|^2)^{1/2})$,
 \begin{align}
 \int_{|z|>1} 
 &
 |z-y|^{\alpha -d} g(x,z)\d z\notag\\
 &=\pi^{-(d/2+1)}\,\Gamma(d/2)\,\sin(\pi\alpha/2)
 \int_{|z|\geq 1} |z-y|^{\alpha -d}\frac{|1- |x|^2|^{\alpha/2}}{|1-|z|^2|^{\alpha/2}}|x - z|^{-d}\d z\notag\\
&=\pi^{-(d/2+1)}\,\Gamma(d/2)\,\sin(\pi\alpha/2)|y-x|^{\alpha-d}
\int_{|w|\leq 1}
\frac{|w - y^\diamond|^{\alpha-d}}{|1-|w|^2|^{\alpha/2}}\d w.
\label{lastbutone}
 \end{align}
 Next, we want to apply the transformation \eqref{sphereinversion} through the sphere $\mathbb{S}_{d-1}(y^\diamond,( |y^
 \diamond|^2-1)^{1/2})$, noting that a similar calculation to the one in \eqref{quadraticdistance}
will give us that, if $w^* =y^\diamond +|w-y^\diamond |^{-2} (w-y^\diamond)(|y^\diamond|-1) $, then
\[
|w^*|^2 -1  = \frac{|w-y^\diamond|}{|y^\diamond|^2-1}(|w|^2-1)
\]
and also a similar calculation to the one in \eqref{usethat} shows us that 
\[
\d w^* =   (|y^\diamond|^2-1)^{d} |w-y^\diamond|^{-2d}\d w.
\]
Following the manipulations in \eqref{1/3c} and \eqref{1/4c}, albeit using \eqref{BMhitballfromoutside} in place of \eqref{Poisson}, recalling that $|y^\diamond|>1$, we get
\begin{align}
\int_{|w|\leq 1}
&\frac{|w - y^\diamond|^{\alpha-d}}{|1-|w|^2|^{\alpha/2}}\d w
\notag\\
 &= \int_{|w^*|\leq 1} \frac{(|y^\diamond|^2-1)^{\alpha/2}}{|1-|w^*|^2|^{\alpha/2}}
|w^*-y^\diamond|^{-d}\d w^*\notag\\
&= 
\frac{2\pi^{d/2}}{\Gamma(d/2)} (|y^\diamond|^2-1)^{\alpha/2}\int_0^1 \frac{r^{d-1}\d r}{(1-r^2)^{\alpha/2}}\int_{\mathbb{S}_{d-1}(0,r)}
|z-y^\diamond|^{-d}\sigma_r(\d z)
\notag\\
&=\frac{\pi^{d/2}}{\Gamma(d/2)} (|y^\diamond|^2-1)^{\alpha/2}\int_0^1 \frac{2r}{(1-r^2)^{\alpha/2}(|y^\diamond|^2-r^2)}
\left(\frac{|y^\diamond|}{r}\right)^{2-d}
\d r\notag\\
&=\frac{\pi^{d/2}}{\Gamma(d/2)} (|y^\diamond|^2-1)^{\alpha/2}|y^\diamond|^{2 - d} \int_0^1v^{d/2-1} (1-v)^{-\alpha/2}(|y^\diamond|^2-v)^{-1}
\d v,
\label{feedback}
\end{align}
where we have made the change of variable $v=r^2$ in the final equality.  To complete the computation in \eqref{feedback}, we need three dentities for the hypergeometric function. These are:
\[
 _{2}F_1(a,b,c;z) =\frac{\Gamma(c)}{\Gamma(b)\Gamma(c-b)} \int_0^1 t^{b-1}(1-t)^{c-b-1}(1-zt)^{-a} \d t,
\]
for $|z|<1$, $a,b,c\in\mathbb{C}$ and $\re(c)>\re(b)>0$, see for example Chapter 9 of \cite{Lebedev},
\[
 _{2}F_1(c-a, c-b, c; z) = (1-z)^{a+b-c}\,\, _{2}F_1(a, b, c; z),
 \qquad |\arg(1-z)|<\pi,
\]
see equation (9.5.3) of \cite{Lebedev},
and
\[
\int_0^x s^{a-1}(1-s)^{b-1} \D s = \frac{x^a}{a}{_2}{F}_1(a, 1- b; a+1; x), \qquad x\in(0,1),
\]
see for example 6.6.8 of \cite{AbSt}.
Combining them, it is easy to show that  the integral in the right-hand side  of \eqref{feedback} satisfies
\begin{align}
&\int_0^1v^{d/2-1} (1-v)^{-\alpha/2}(|y^\diamond|^2-v)^{-1}
\d v\notag\\
&=|y^\diamond|^{-2} 
\frac{\Gamma(d/2)\Gamma(1-{\alpha}/{2})}{\Gamma(1+(d-\alpha)/2)}
{_{2}F_1}(1,\, d/2,\, 1+({d}-\alpha)/2; \,|y^\diamond|^{-2})\notag\\
%%up to here
&=|y^\diamond|^{-2}(1- |y^\diamond|^{-2})^{-\alpha/2}
\frac{\Gamma(d/2)\Gamma(1-{\alpha}/{2})}{\Gamma(1+(d-\alpha)/2)}\notag\\
&\hspace{3cm}\times
{_{2}F_1}((d-\alpha)/2, \, 1- {\alpha}/{2},\, 1+(d-\alpha)/2 ;|y^\diamond|^{-2})\notag\\
&=\frac{\Gamma(d/2)\Gamma(1-{\alpha}/{2})}{\Gamma((d-\alpha)/2)}
|y^\diamond|^{d-2}( |y^\diamond|^2-1)^{-\alpha/2}
\int_0^{|y^\diamond|^{-2}} s^{\frac{(d-\alpha)}{2}-1}(1-s)^{\frac{\alpha}{2}-1} \D s 
\label{feedback2}
\end{align}
Now putting \eqref{feedback2} into \eqref{feedback}, then into \eqref{lastbutone} and the latter into the left-hand side of \eqref{provethisforpotential}, we get
\begin{align*}
 & |x-y|^{\alpha-d}\left(1 -\frac{\Gamma(d/2)|y^\diamond|^{-2}}{\Gamma(\alpha/2)\Gamma((d-\alpha)/2)}
\int_0^{|y^\diamond|^{-2}} s^{\frac{(d-\alpha)}{2}-1}(1-s)^{\frac{\alpha}{2}-1} \D s \right)\\
  &= \frac{\Gamma(d/2)}{\Gamma(\alpha/2)\Gamma((d-\alpha)/2)} |x-y|^{\alpha-d}\int_{|y^\diamond|^{-2}}^1s^{\frac{(d-\alpha)}{2}-1}(1-s)^{\frac{\alpha}{2}-1} \D s \\
 &= \frac{\Gamma(d/2)}{\Gamma(\alpha/2)\Gamma((d-\alpha)/2)} |x-y|^{\alpha-d}\int^{|y^\diamond|^{2}-1}_0 (1+u)^{-\frac{d}{2}}u^{\frac{\alpha}{2}-1}\D u
\end{align*}
where, in the final equality, we made the change of variables $u = (1-s)/s$.  Recalling from the second equality in \eqref{usethis} that $|y^\diamond|^2-1 =\zeta^\ominus(x,y)$,  we finally come to rest at the conclusion that the left-hand side of \eqref{provethisforpotential} agrees with the required right-hand side, thus completing the proof.

\medskip

For part (ii) of the theorem we appeal again to the Riesz--Bogdan--\.Zak transform in Theorem \ref{th:BZ}. Recall that this transform states that, for $x\neq 0$, $(KX_{\eta(t)}, t\geq 0)$ under $
\mathbb{P}_{Kx}$ is equal in law to $(X_t, t\geq 0)$ under $\mathbb{P}^\circ_{x}$, where
 $\eta(t) = \inf\{s>0: \int_0^s|X_u|^{-2\alpha}\d u>t\}$. 
Noting that, since $\int_0^{\eta(t)}|X_u|^{-2\alpha}\d u=t$, if we write $s = \eta(t)$, then 
\[
|X_{s}|^{-2\alpha}\d s =\d t, \qquad t>0,
\]
and hence we have that, for $|x|>1$,
\begin{align*}
 \int_{|y|>1}\frac{|z|^{\alpha-d}}{|x|^{\alpha -d}}h^\oplus_u(x,z)f(z)\d z&= \mathbb{E}^\circ_{x}\left[\int_0^{\tau^{\oplus}}f(X_t)\d t\right]\\
&=\mathbb{E}_{Kx}\left[\int_0^{\tau^\ominus} f(KX_{\eta(t)})% |X_s|^{2\alpha}
\d t\right] \\
&=\mathbb{E}_{Kx}\left[\int_0^{\tau^\ominus} f(KX_{s}) |X_s|^{-2\alpha}
\d s\right] \\
&=\int_{|y|<1} h^\ominus(Kx,y)f(K y)|y|^{-2\alpha}\d y,
\end{align*}
where we have pre-emptively assumed that \eqref{hasad} has a density, which we have denoted by $h^\oplus (x,y)$. In the integral on the left-hand side above, we can make the change of variables $y= Kz$, which is equivalent to $z= Ky$. Noting that $\d y = \d z/{|z|^{2d}}$ and appealing to the identity \eqref{BGR3}, we get
\begin{align*}
 \int_{|y|>1}\frac{|z|^{\alpha-d}}{|x|^{\alpha -d}}h^\oplus(x,z)f(z)\d z&=
 \int_{|z|>1} h^\ominus(Kx,Kz)f(z)%|K_r z|^{2\alpha}
  \frac{|z|^{2\alpha}}{|z|^{2d}}\d z,
 \end{align*}
from which we can conclude that, for $|x|, |z|>1$,
\begin{align*}
h^\oplus(x,z)& = \frac{|x|^{\alpha -d}}{|z|^{\alpha-d}}h^\ominus(Kx,K z) \frac{ |z|^{2\alpha}}{|z|^{2d}}\notag\\
&=2^{-\alpha}\pi^{-d/2}\frac{\Gamma(d/2)}{\Gamma(\alpha/2)^2}\frac{|x|^{\alpha -d}}{|z|^{\alpha-d}}
\frac{ |z|^{2\alpha}}{|z|^{2d}}
|Kx-Kz|^{\alpha - d} \int_0^{\zeta^\ominus(Kx,Kz)}  (u+1)^{-d/2}u^{\alpha/2-1}\d u.
\end{align*}
Hence, after a little algebra, for $|x|, |z|>1$,
\begin{equation*}
h^\oplus(x,z)=2^{-\alpha}\pi^{-d/2}\frac{\Gamma(d/2)}{\Gamma(\alpha/2)^2}
|x-z|^{\alpha - d}
 \int_0^{\zeta^\oplus(x,z)}  (u+1)^{-d/2}u^{\alpha/2-1}\d u
\end{equation*}
where we have again used the fact that 
$|Kx - Kz| = |x-z|/|x||z|$ so that 
\[
 \zeta^\ominus(Kx,Kz)=(|x|^2-1)(|z|^{2}-1)/|x-z|^2 =: \zeta^\oplus(x,z)
\]
and the result is proved.
\end{proof}
\part*{Onward research} 

The results presented in this review article are no doubt interesting to anyone who has followed and enjoyed the explicit and rich nature of fluctuation identities that one can obtain for Brownian motion and, more generally, L\'evy processes.  
Given the historical applicability of such results in the general arena of applied probability, this in itself is a strong motivation to pursue more research in the setting of stable processes. The methods highlighted are indeed robust enough to deal with other scenarios that have not been covered in this article; we mention the questions addressed in \cite{BMR} and \cite{Luks}, who give the potential of the half-space and the complement of the hyperplane for stable processes with $\alpha>1$, as but one example thereof.

However, there is a stronger reason yet. Self-similarity is a very natural notion and self-similar Markov processes are talked about in abundance in scientific and mathematical  literature. It is therefore remarkable that the foundational theory of the aforesaid class has received relatively little attention since the principle treatment in \cite{L72}, where only positive processes were considered. Whilst there has been significant progression in the case of positive self-similar Markov processes in the last 10 or more years, see for example \cite{BY02, BC02, CP06, CC06b, CC06a, CKP09, P09, CR07, KP11, CKPR12, Patie}, less can be said in higher dimensions.  With only the sporadic works of e.g. \cite{GV, VG, Kiu} in the 1980s, only very recently has the bigger picture begun to emerge some 30 or more years later in \cite{X1, X2, CPR, KKPW, ACGZ, KRS, Deep1, Deep2, Deep3}.

As we have seen in this review article, what has been uncovered thus far in the literature  is an  intimate relationship between self-similar Markov processes and MAPs and, moreover, one that can be seen in clear detail for the stable case. To progress forward the theory of self-similar Markov processes, a much clearer understanding of the general theory of MAPs will be needed. The outlook is positive. We have already alluded to the fact that   discrete modulation MAPs have enjoyed prominence in a wide body of applied probability literature, most notably in queuing theory and storage processes, \cite{AsmussenQueue, Prabhu}. Moreover, some classical fluctuation theory for MAPs can also be found, \cite{Cinlar, Cinlar1, Cinlar2, Kaspi}, as well as the crucially associated Markov additive renewal theory, e.g. \cite{Kesten74, Alsmeyer1994, Alsmeyer2014} and references therein. 

There is  much more to do, both in the abstract and the concrete setting, in particular, when the modulation has an uncountable state space. The mathematical similarity of MAPs (conditional stationary and independent increments) to L\'evy processes (stationary and independent increments) offers significant insight. We predict that, with the right approach, a great deal of the expansive development of the abstract theory of L\'evy processes that has occurred in the last 25 years will find a natural generalisations in the MAP setting; see for example the Appendix of \cite{DDK} for a first step in this direction. Similarly, we predict that, outside of   stable processes and Brownian motion, just as  many natural families of L\'evy processes have been found in applications which exemplify in concrete terms the theory of L\'evy processes (e.g. Kuznetsov's $\theta$- and $\beta$-processes, meromorphic processes, hypergeometric processes, variance gamma processes, CGMY processes and  so on), tractable examples of MAPs will similarly emerge bearing relevance to a variety of applications where MAPs begin to emerge; see \cite{HS, S} to name but a couple of  very recent examples. 

\part*{Acknowledgements}
The author would like to thank Victor Perez Abreu for the invitation to submit this review article. He would also like to thank Alexey Kuznetsov, Mateusz Kwa\'snicki, Bo Li, Juan Carlos Pardo, Victor Rivero, Weerapat Satitkanitkul and Daniel Ng for several helpful discussions. I am deeply indebted to several anonymous referees whose extensive reading of an earlier draft of this document led to many improvements, as well as for the extra references that they pointed me towards.

    \addcontentsline{toc}{part}{References}
    \settocdepth{part}
\bibliographystyle{alea3}

\bibliography{references}

\begin{thebibliography}{81}
\providecommand{\natexlab}[1]{#1}
\providecommand{\url}[1]{\texttt{#1}}
\providecommand{\urlprefix}{URL }
\expandafter\ifx\csname urlstyle\endcsname\relax
  \providecommand{\doi}[1]{doi:\discretionary{}{}{}#1}\else
  \providecommand{\doi}{doi:\discretionary{}{}{}\begingroup
  \urlstyle{rm}\Url}\fi
\providecommand{\eprint}[2][]{\url{#2}}

\bibitem[{Abramowitz and Stegun(1964)}]{AbSt}
M.~Abramowitz and I.~A. Stegun.
\newblock \emph{Handbook of mathematical functions with formulas, graphs, and
  mathematical tables}, volume~55 of \emph{National Bureau of Standards Applied
  Mathematics Series}.
\newblock For sale by the Superintendent of Documents, U.S. Government Printing
  Office, Washington, D.C. (1964).

\bibitem[{Alili et~al.(2017)Alili, Chaumont, Graczyk and \.Zak}]{ACGZ}
L.~Alili, L.~Chaumont, P.~Graczyk and T.~\.Zak.
\newblock Inversion, duality and {D}oob {$h$}-transforms for self-similar
  {M}arkov processes.
\newblock \emph{Electron. J. Probab.} \textbf{22}, Paper No. 20, 18 (2017).

\bibitem[{Alsmeyer(1994)}]{Alsmeyer1994}
G.~Alsmeyer.
\newblock On the {M}arkov renewal theorem.
\newblock \emph{Stochastic Process. Appl.} \textbf{50}~(1), 37--56 (1994).

\bibitem[{Alsmeyer(2014)}]{Alsmeyer2014}
G.~Alsmeyer.
\newblock Quasistochastic matrices and {M}arkov renewal theory.
\newblock \emph{J. Appl. Probab.} \textbf{51A}~(Celebrating 50 Years of The
  Applied Probability Trust), 359--376 (2014).

\bibitem[{Asmussen(2003)}]{AsmussenQueue}
S.~Asmussen.
\newblock \emph{Applied probability and queues}, volume~51 of
  \emph{Applications of Mathematics (New York)}.
\newblock Springer-Verlag, New York, second edition (2003).
\newblock Stochastic Modelling and Applied Probability.

\bibitem[{Asmussen and Albrecher(2010)}]{AA}
S.~Asmussen and H.~Albrecher.
\newblock \emph{Ruin probabilities}, volume~14 of \emph{Advanced Series on
  Statistical Science \& Applied Probability}.
\newblock World Scientific Publishing Co. Pte. Ltd., Hackensack, NJ, second
  edition (2010).

\bibitem[{Bertoin(1996)}]{bertoin}
J.~Bertoin.
\newblock \emph{L\'evy processes}, volume 121 of \emph{Cambridge Tracts in
  Mathematics}.
\newblock Cambridge University Press, Cambridge (1996).

\bibitem[{Bertoin and Caballero(2002)}]{BC02}
J.~Bertoin and M.-E. Caballero.
\newblock Entrance from {$0+$} for increasing semi-stable {M}arkov processes.
\newblock \emph{Bernoulli} \textbf{8}~(2), 195--205 (2002).

\bibitem[{Bertoin and Werner(1996)}]{BW}
J.~Bertoin and W.~Werner.
\newblock Stable windings.
\newblock \emph{Ann. Probab.} \textbf{24}~(3), 1269--1279 (1996).

\bibitem[{Bertoin and Yor(2002)}]{BY02}
J.~Bertoin and M.~Yor.
\newblock The entrance laws of self-similar {M}arkov processes and exponential
  functionals of {L}\'evy processes.
\newblock \emph{Potential Anal.} \textbf{17}~(4), 389--400 (2002).

\bibitem[{Bliedtner and Hansen(1986)}]{BH}
J.~Bliedtner and W.~Hansen.
\newblock \emph{Potential theory}.
\newblock Universitext. Springer-Verlag, Berlin (1986).
\newblock An analytic and probabilistic approach to balayage.

\bibitem[{Blumenson(1960)}]{blum}
L.~E. Blumenson.
\newblock Classroom {N}otes: {A} {D}erivation of {$n$}-{D}imensional
  {S}pherical {C}oordinates.
\newblock \emph{Amer. Math. Monthly} \textbf{67}~(1), 63--66 (1960).

\bibitem[{Blumenthal and Getoor(1968)}]{BG}
R.~M. Blumenthal and R.~K. Getoor.
\newblock \emph{{M}arkov processes and potential theory}.
\newblock Pure and Applied Mathematics, Vol. 29. Academic Press, New
  York-London (1968).

\bibitem[{Blumenthal et~al.(1961)Blumenthal, Getoor and Ray}]{BGR}
R.~M. Blumenthal, R.~K. Getoor and D.~B. Ray.
\newblock On the distribution of first hits for the symmetric stable processes.
\newblock \emph{Trans. Amer. Math. Soc.} \textbf{99}, 540--554 (1961).

\bibitem[{Bogdan et~al.(2003)Bogdan, Burdzy and Chen}]{BBC}
K.~Bogdan, K.~Burdzy and Z.-Q. Chen.
\newblock Censored stable processes.
\newblock \emph{Probab. Theory Related Fields} \textbf{127}~(1), 89--152
  (2003).

\bibitem[{Bogdan and {\.Z}ak(2006)}]{BZ}
K.~Bogdan and T.~{\.Z}ak.
\newblock On {K}elvin transformation.
\newblock \emph{J. Theoret. Probab.} \textbf{19}~(1), 89--120 (2006).

\bibitem[{Bretagnolle(1971)}]{Bret}
J.~Bretagnolle.
\newblock R\'esultats de {K}esten sur les processus \`a accroissements
  ind\'ependants.
\newblock In \emph{S\'eminaire de {P}robabilit\'es, {V} ({U}niv. {S}trasbourg,
  ann\'ee universitaire 1969-1970)}, pages 21--36. Lecture Notes in Math., Vol.
  191. Springer, Berlin (1971).

\bibitem[{Byczkowski et~al.(2009)Byczkowski, Ma{\l}ecki and Ryznar}]{BMR}
T.~Byczkowski, J.~Ma{\l}ecki and M.~Ryznar.
\newblock Bessel potentials, hitting distributions and {G}reen functions.
\newblock \emph{Trans. Amer. Math. Soc.} \textbf{361}~(9), 4871--4900 (2009).

\bibitem[{Caballero and Chaumont(2006{\natexlab{a}})}]{CC06b}
M.~E. Caballero and L.~Chaumont.
\newblock Conditioned stable {L}\'evy processes and the {L}amperti
  representation.
\newblock \emph{J. Appl. Probab.} \textbf{43}~(4), 967--983
  (2006{\natexlab{a}}).

\bibitem[{Caballero and Chaumont(2006{\natexlab{b}})}]{CC06a}
M.~E. Caballero and L.~Chaumont.
\newblock Weak convergence of positive self-similar {M}arkov processes and
  overshoots of {L}\'evy processes.
\newblock \emph{Ann. Probab.} \textbf{34}~(3), 1012--1034 (2006{\natexlab{b}}).

\bibitem[{Caballero et~al.(2011)Caballero, Pardo and P{\'e}rez}]{CPP11}
M.~E. Caballero, J.~C. Pardo and J.~L. P{\'e}rez.
\newblock Explicit identities for {L}\'evy processes associated to symmetric
  stable processes.
\newblock \emph{Bernoulli} \textbf{17}~(1), 34--59 (2011).

\bibitem[{Chaumont et~al.(2009)Chaumont, Kyprianou and Pardo}]{CKP09}
L.~Chaumont, A.~E. Kyprianou and J.~C. Pardo.
\newblock Some explicit identities associated with positive self-similar
  {M}arkov processes.
\newblock \emph{Stochastic Process. Appl.} \textbf{119}~(3), 980--1000 (2009).

\bibitem[{Chaumont et~al.(2012)Chaumont, Kyprianou, Pardo and Rivero}]{CKPR12}
L.~Chaumont, A.~E. Kyprianou, J.~C. Pardo and V.~Rivero.
\newblock Fluctuation theory and exit systems for positive self-similar
  {M}arkov processes.
\newblock \emph{Ann. Probab.} \textbf{40}~(1) (2012).

\bibitem[{Chaumont et~al.(2013)Chaumont, Pant{\'\i} and Rivero}]{CPR}
L.~Chaumont, H.~Pant{\'\i} and V.~Rivero.
\newblock The {L}amperti representation of real-valued self-similar {M}arkov
  processes.
\newblock \emph{Bernoulli} \textbf{19}~(5B), 2494--2523 (2013).

\bibitem[{Chaumont and Pardo(2006)}]{CP06}
L.~Chaumont and J.~C. Pardo.
\newblock The lower envelope of positive self-similar {M}arkov processes.
\newblock \emph{Electron. J. Probab.} \textbf{11}, no. 49, 1321--1341 (2006).

\bibitem[{Chaumont and Rivero(2007)}]{CR07}
L.~Chaumont and V.~Rivero.
\newblock On some transformations between positive self-similar {M}arkov
  processes.
\newblock \emph{Stochastic Process. Appl.} \textbf{117}~(12), 1889--1909
  (2007).

\bibitem[{{\c{C}}inlar(1972)}]{Cinlar}
E.~{\c{C}}inlar.
\newblock {M}arkov additive processes. {I}, {II}.
\newblock \emph{Z. Wahrscheinlichkeitstheorie und Verw. Gebiete} \textbf{24},
  85--93; ibid. 24 (1972), 95--121 (1972).

\bibitem[{{\c{C}}inlar(1974/75)}]{Cinlar2}
E.~{\c{C}}inlar.
\newblock L\'evy systems of {M}arkov additive processes.
\newblock \emph{Z. Wahrscheinlichkeitstheorie und Verw. Gebiete} \textbf{31},
  175--185 (1974/75).

\bibitem[{{\c{C}}inlar(1976)}]{Cinlar1}
E.~{\c{C}}inlar.
\newblock Entrance-exit distributions for {M}arkov additive processes.
\newblock \emph{Math. Programming Stud.} ~(5), 22--38 (1976).
\newblock Stochastic systems: modeling, identification and optimization, I
  (Proc. Sympos., Univ. Kentucky, Lexington, Ky., 1975).

\bibitem[{Dereich et~al.(2017)Dereich, D\"oring and Kyprianou}]{DDK}
S.~Dereich, L.~D\"oring and A.~E. Kyprianou.
\newblock Real self-similar processes started from the origin.
\newblock \emph{Ann. Probab.} \textbf{45}~(3), 1952--2003 (2017).

\bibitem[{Doney(2007)}]{D}
R.~A. Doney.
\newblock \emph{Fluctuation theory for {L}\'evy processes}, volume 1897 of
  \emph{Lecture Notes in Mathematics}.
\newblock Springer, Berlin (2007).
\newblock Lectures from the 35th Summer School on Probability Theory held in
  Saint-Flour, July 6--23, 2005, Edited and with a foreword by Jean Picard.

\bibitem[{Doney and Kyprianou(2006)}]{DK}
R.~A. Doney and A.~E. Kyprianou.
\newblock Overshoots and undershoots of {L}\'evy processes.
\newblock \emph{Ann. Appl. Probab.} \textbf{16}~(1), 91--106 (2006).

\bibitem[{D\"oring and Kyprianou(2018)}]{DK18}
L.~D\"oring and A.~E. Kyprianou.
\newblock Entrance and exit at infinity for stable jump diffusions  (2018).
\newblock \texttt{ arXiv:1802.01672 [math.PR]}.

\bibitem[{Dynkin(1961)}]{Dynkin1961}
E.~B. Dynkin.
\newblock Some limit theorems for sums of independent random variables with
  infinite mathematical expectations.
\newblock In \emph{Select. {T}ransl. {M}ath. {S}tatist. and {P}robability,
  {V}ol. 1}, pages 171--189. Inst. Math. Statist. and Amer. Math. Soc.,
  Providence, R.I. (1961).

\bibitem[{Getoor(1966)}]{G}
R.~K. Getoor.
\newblock Continuous additive functionals of a {M}arkov process with
  applications to processes with independent increments.
\newblock \emph{J. Math. Anal. Appl.} \textbf{13}, 132--153 (1966).

\bibitem[{Graversen and Vuolle-Apiala(1986)}]{GV}
S.~E. Graversen and J.~Vuolle-Apiala.
\newblock {$\alpha$}-self-similar {M}arkov processes.
\newblock \emph{Probab. Theory Relat. Fields} \textbf{71}~(1), 149--158 (1986).

\bibitem[{Haas and Stephenson(2017)}]{HS}
B\'en\'edicte Haas and Robin Stephenson.
\newblock On the exponential functional of {M}arkov additive processes, and
  applications to multi-type self-similar fragmentation processes and trees
  (2017).
\newblock \texttt{arXiv:1612.06058 [math.PR]}.

\bibitem[{Horowitz(1972)}]{Horo}
J.~Horowitz.
\newblock Semilinear {M}arkov processes, subordinators and renewal theory.
\newblock \emph{Z. Wahrscheinlichkeitstheorie und Verw. Gebiete} \textbf{24},
  167--193 (1972).

\bibitem[{Jacod and Shiryaev(2003)}]{JS}
J.~Jacod and A.~N. Shiryaev.
\newblock \emph{Limit theorems for stochastic processes}, volume 288 of
  \emph{Grundlehren der Mathematischen Wissenschaften [Fundamental Principles
  of Mathematical Sciences]}.
\newblock Springer-Verlag, Berlin, second edition (2003).

\bibitem[{Janssen and Manca(2007)}]{JR}
J.~Janssen and R.~Manca.
\newblock \emph{Semi-{M}arkov risk models for finance, insurance and
  reliability}.
\newblock Springer, New York (2007).

\bibitem[{Kaspi(1982)}]{Kaspi}
H.~Kaspi.
\newblock On the symmetric {W}iener-{H}opf factorization for {M}arkov additive
  processes.
\newblock \emph{Z. Wahrsch. Verw. Gebiete} \textbf{59}~(2), 179--196 (1982).

\bibitem[{Kesten(1969{\natexlab{a}})}]{Kest2}
H.~Kesten.
\newblock A convolution equation and hitting probabilities of single points for
  processes with stationary independent increments.
\newblock \emph{Bull. Amer. Math. Soc.} \textbf{75}, 573--578
  (1969{\natexlab{a}}).

\bibitem[{Kesten(1969{\natexlab{b}})}]{Kest1}
H.~Kesten.
\newblock \emph{Hitting probabilities of single points for processes with
  stationary independent increments}.
\newblock Memoirs of the American Mathematical Society, No. 93. American
  Mathematical Society, Providence, R.I. (1969{\natexlab{b}}).

\bibitem[{Kesten(1969{\natexlab{c}})}]{Kestenpoint}
H.~Kesten.
\newblock \emph{Hitting probabilities of single points for processes with
  stationary independent increments}.
\newblock Memoirs of the American Mathematical Society, No. 93. American
  Mathematical Society, Providence, R.I. (1969{\natexlab{c}}).

\bibitem[{Kesten(1974)}]{Kesten74}
H.~Kesten.
\newblock Renewal theory for functionals of a {M}arkov chain with general state
  space.
\newblock \emph{Ann. Probability} \textbf{2}, 355--386 (1974).

\bibitem[{Kingman(1964)}]{Kingman64}
J.~F.~C. Kingman.
\newblock Recurrence properties of processes with stationary independent
  increments.
\newblock \emph{J. Austral. Math. Soc.} \textbf{4}, 223--228 (1964).

\bibitem[{Kiu(1980)}]{Kiu}
S.~W. Kiu.
\newblock Semistable {M}arkov processes in {${\bf R}^{n}$}.
\newblock \emph{Stochastic Process. Appl.} \textbf{10}~(2), 183--191 (1980).

\bibitem[{Kuznetsov et~al.(2014)Kuznetsov, Kyprianou, Pardo and Watson}]{KKPW}
A.~Kuznetsov, A.~E. Kyprianou, J.~C. Pardo and A.~R. Watson.
\newblock The hitting time of zero for a stable process.
\newblock \emph{Electron. J. Probab.} \textbf{19}, no. 30, 26 (2014).

\bibitem[{Kuznetsov and Pardo(2013)}]{KP}
A.~Kuznetsov and J.~C. Pardo.
\newblock Fluctuations of stable processes and exponential functionals of
  hypergeometric {L}\'evy processes.
\newblock \emph{Acta Appl. Math.} \textbf{123}, 113--139 (2013).

\bibitem[{Kyprianou(2014)}]{Kypbook}
A.~E. Kyprianou.
\newblock \emph{Fluctuations of {L}\'evy processes with applications}.
\newblock Universitext. Springer, Heidelberg, second edition (2014).
\newblock Introductory lectures.

\bibitem[{Kyprianou(2016)}]{Deep1}
A.~E. Kyprianou.
\newblock Deep factorisation of the stable process.
\newblock \emph{Electron. J. Probab.} \textbf{21}, Paper No. 23, 28 (2016).

\bibitem[{Kyprianou et~al.(2014)Kyprianou, Pardo and Watson}]{KPW}
A.~E. Kyprianou, J.~C. Pardo and A.~R. Watson.
\newblock Hitting distributions of {$\alpha$}-stable processes via path
  censoring and self-similarity.
\newblock \emph{Ann. Probab.} \textbf{42}~(1), 398--430 (2014).

\bibitem[{Kyprianou and Pardo(2019)}]{KPbook}
A.~E. Kyprianou and J.C. Pardo.
\newblock \emph{Stable L\'evy processes via Lamperti-type representations}.
\newblock Book project in progress. (2019).

\bibitem[{Kyprianou and Patie(2011)}]{KP11}
A.~E. Kyprianou and P.~Patie.
\newblock A {C}iesielski-{T}aylor type identity for positive self-similar
  {M}arkov processes.
\newblock \emph{Ann. Inst. Henri Poincar\'e Probab. Stat.} \textbf{47}~(3),
  917--928 (2011).

\bibitem[{Kyprianou et~al.(2016{\natexlab{a}})Kyprianou, Rivero and
  \c{S}engul}]{Deep2}
A.~E. Kyprianou, V.~Rivero and B.~\c{S}engul.
\newblock Deep factorisation of the stable process {II}; potentials and
  applications (2016{\natexlab{a}}).
\newblock \texttt{arXiv:1511.06356 [math.PR]}.

\bibitem[{Kyprianou et~al.(2015)Kyprianou, Rivero and Satitkanitkul}]{KRS}
A.~E. Kyprianou, V.~Rivero and W.~Satitkanitkul.
\newblock Conditioned real self-similar {M}arkov processes (2015).
\newblock \texttt{ arXiv:1510.01781 [math.PR]}.

\bibitem[{Kyprianou et~al.(2016{\natexlab{b}})Kyprianou, Rivero and
  Satitkanitkul}]{Deep3}
A.~E. Kyprianou, V.~Rivero and W.~Satitkanitkul.
\newblock Deep factorisation of the stable process iii: radial excursion theory
  and the point of closest reach (2016{\natexlab{b}}).
\newblock \texttt{arXiv:1706.09924 [math.PR]}.

\bibitem[{Kyprianou and Watson(2014)}]{KW}
A.~E. Kyprianou and A.~R. Watson.
\newblock Potentials of stable processes.
\newblock In \emph{S\'eminaire de {P}robabilit\'es {XLVI}}, volume 2123 of
  \emph{Lecture Notes in Math.}, pages 333--343. Springer, Cham (2014).

\bibitem[{Lamperti(1962)}]{Lamperti62}
J.~Lamperti.
\newblock An invariance principle in renewal theory.
\newblock \emph{Ann. Math. Statist.} \textbf{33}, 685--696 (1962).

\bibitem[{Lamperti(1972)}]{L72}
J.~Lamperti.
\newblock Semi-stable {M}arkov processes. {I}.
\newblock \emph{Z. Wahrscheinlichkeitstheorie und Verw. Gebiete} \textbf{22},
  205--225 (1972).

\bibitem[{Landkof(1972)}]{Landkof}
N.~S. Landkof.
\newblock \emph{Foundations of modern potential theory}.
\newblock Springer-Verlag, New York-Heidelberg (1972).
\newblock Translated from the Russian by A. P. Doohovskoy, Die Grundlehren der
  mathematischen Wissenschaften, Band 180.

\bibitem[{Lebedev(1972)}]{Lebedev}
N.~N. Lebedev.
\newblock \emph{Special functions and their applications}.
\newblock Dover Publications, Inc., New York (1972).
\newblock Revised edition, translated from the Russian and edited by Richard A.
  Silverman, Unabridged and corrected republication.

\bibitem[{Luks(2013)}]{Luks}
T.~Luks.
\newblock Boundary behavior of {$\alpha$}-harmonic functions on the complement
  of the sphere and hyperplane.
\newblock \emph{Potential Anal.} \textbf{39}~(1), 29--67 (2013).

\bibitem[{Nane et~al.(2010)Nane, Xiao and Zeleke}]{X2}
E.~Nane, Y.~Xiao and A.~Zeleke.
\newblock A strong law of large numbers with applications to self-similar
  stable processes.
\newblock \emph{Acta Sci. Math. (Szeged)} \textbf{76}~(3-4), 697--711 (2010).

\bibitem[{Patie(2009)}]{P09}
P.~Patie.
\newblock Infinite divisibility of solutions to some self-similar
  integro-differential equations and exponential functionals of {L}\'evy
  processes.
\newblock \emph{Ann. Inst. Henri Poincar\'e Probab. Stat.} \textbf{45}~(3),
  667--684 (2009).

\bibitem[{Patie(2012)}]{Patie}
P.~Patie.
\newblock Law of the absorption time of some positive self-similar {M}arkov
  processes.
\newblock \emph{Ann. Probab.} \textbf{40}~(2), 765--787 (2012).

\bibitem[{Port(1969)}]{Port69}
S.~C. Port.
\newblock The first hitting distribution of a sphere for symmetric stable
  processes.
\newblock \emph{Trans. Amer. Math. Soc.} \textbf{135}, 115--125 (1969).

\bibitem[{Port and Stone(1971)}]{PortStoneRT}
S.~C. Port and C.~J. Stone.
\newblock Infinitely divisible processes and their potential theory.
\newblock \emph{Ann. Inst. Fourier (Grenoble)} \textbf{21}~(2), 157--275; ibid.
  21 (1971), no. 4, 179--265 (1971).

\bibitem[{Port and Stone(1978)}]{PortStonebook}
S.~C. Port and C.~J. Stone.
\newblock \emph{Brownian motion and classical potential theory}.
\newblock Academic Press [Harcourt Brace Jovanovich, Publishers], New
  York-London (1978).
\newblock Probability and Mathematical Statistics.

\bibitem[{Prabhu(1965)}]{Prabhu}
N.~U. Prabhu.
\newblock \emph{Queues and inventories. {A} study of their basic stochastic
  processes}.
\newblock John Wiley \& Sons, Inc., New York-London-Sydney (1965).

\bibitem[{Profeta and Simon(2016)}]{PS}
C.~Profeta and T.~Simon.
\newblock On the harmonic measure of stable processes.
\newblock In \emph{S\'eminaire de {P}robabilit\'es {XLVIII}}, volume 2168 of
  \emph{Lecture Notes in Math.}, pages 325--345. Springer, Cham (2016).

\bibitem[{Protter(2005)}]{protter}
P.~E. Protter.
\newblock \emph{Stochastic integration and differential equations}, volume~21
  of \emph{Stochastic Modelling and Applied Probability}.
\newblock Springer-Verlag, Berlin (2005).
\newblock Second edition. Version 2.1, Corrected third printing.

\bibitem[{Ray(1958)}]{Ray}
D.~Ray.
\newblock Stable processes with an absorbing barrier.
\newblock \emph{Trans. Amer. Math. Soc.} \textbf{89}, 16--24 (1958).

\bibitem[{Riesz(1938)}]{Riesz}
M.~Riesz.
\newblock Int\'egrales de {R}iemann-{L}iouville et potentiels.
\newblock \emph{Acta. Sci. Math. Szeged.} \textbf{9}, 1--42 (1938).

\bibitem[{Rogozin(1972)}]{Rog}
B.~A. Rogozin.
\newblock Distribution of the position of absorption for stable and
  asymptotically stable random walks on an interval.
\newblock \emph{Teor. Verojatnost. i Primenen.} \textbf{17}, 342--349 (1972).

\bibitem[{Sato(2013)}]{Sato}
K-I Sato.
\newblock \emph{L\'evy processes and infinitely divisible distributions},
  volume~68 of \emph{Cambridge Studies in Advanced Mathematics}.
\newblock Cambridge University Press, Cambridge (2013).
\newblock Translated from the 1990 Japanese original, Revised edition of the
  1999 English translation.

\bibitem[{Schilling et~al.(2012)Schilling, Song and Vondra\v~cek}]{SSV}
R.~L. Schilling, R.~Song and Z.~Vondra\v~cek.
\newblock \emph{Bernstein functions}, volume~37 of \emph{De Gruyter Studies in
  Mathematics}.
\newblock Walter de Gruyter \& Co., Berlin, second edition (2012).
\newblock Theory and applications.

\bibitem[{Stephenson(2017)}]{S}
Robin Stephenson.
\newblock On the exponential functional of {M}arkov additive processes, and
  applications to multi-type self-similar fragmentation processes and trees
  (2017).
\newblock \texttt{arXiv:1706.03495 [math.PR]}.

\bibitem[{Vuolle-Apiala and Graversen(1986)}]{VG}
J.~Vuolle-Apiala and S.~E. Graversen.
\newblock Duality theory for self-similar processes.
\newblock \emph{Ann. Inst. H. Poincar\'e Probab. Statist.} \textbf{22}~(3),
  323--332 (1986).

\bibitem[{Widom(1961)}]{Widom}
H.~Widom.
\newblock Stable processes and integral equations.
\newblock \emph{Trans. Amer. Math. Soc.} \textbf{98}, 430--449 (1961).

\bibitem[{Xiao(1998)}]{X1}
Y.~Xiao.
\newblock Asymptotic results for self-similar {M}arkov processes.
\newblock In \emph{Asymptotic methods in probability and statistics ({O}ttawa,
  {ON}, 1997)}, pages 323--340. North-Holland, Amsterdam (1998).

\end{thebibliography}

\end{document}